\documentclass[11pt,reqno]{amsart}

\usepackage[utf8]{inputenc}
\usepackage{enumerate}
\usepackage{dsfont}
\usepackage{color}
\usepackage{a4wide}
\usepackage[all]{xy}
\usepackage[colorlinks=true]{hyperref}
\usepackage{bbm}
\usepackage{subcaption}
\hypersetup{%
    ,urlcolor=blue
    ,citecolor=red
    ,linkcolor=blue
    }
\usepackage{amssymb}
\usepackage{parskip}
\usepackage[framemethod=tikz]{mdframed}
\usepackage{pifont}
\usepackage[geometry]{ifsym}
\usepackage{comment}

\allowdisplaybreaks

\definecolor{mycolor}{rgb}{0.122, 0.435, 0.698}

\newmdenv[innerlinewidth=0.5pt, roundcorner=4pt,linecolor=mycolor,innerleftmargin=6pt,
innerrightmargin=6pt,innertopmargin=6pt,innerbottommargin=6pt]{mybox}

\usepackage[T1]{fontenc}
\usepackage{amsmath}
\usepackage{amsfonts}
\usepackage{amssymb}
\usepackage{graphicx}
\usepackage{geometry}
\usepackage{amsthm}
\usepackage[]{algorithm2e}

\theoremstyle{plain}

\DeclareMathOperator{\e}{e}
\newtheorem{thm}{Theorem}[section]
\newtheorem{prop}[thm]{Proposition}
\newtheorem{lem}[thm]{Lemma}

\newtheorem{rem}[thm]{Remark}
\newtheorem{cor}[thm]{Corollary}

\newcommand{\R}{\mathbb{R}}

\usepackage{fullpage}
\title{Interpolated Drift Implicit Euler MLMC Method for Barrier Option Pricing and application to  CIR and CEV Models}

\author{\textsc{Mouna Ben Derouich}}
\thanks{}
\address{Mouna Ben Derouich, Université Sorbonne Paris Nord, LAGA, CNRS, UMR 7539,  F-93430, Villetaneuse, France}
\email{benderouiche@math.univ-paris13.fr}

\author{\textsc{Ahmed Kebaier}}
\thanks{}
\address{Ahmed Kebaier, Laboratoire de Math\'ematiques et Mod\'elisation d'Evry, CNRS, UMR 8071, Université d'Evry, Universit\'e Paris-Saclay, 91037, Evry, France  }
\email{ahmed.kebaier@univ-evry.fr}

\date{\today}



\begin{document}
\maketitle
\begin{abstract}
Recently, Giles et al. \cite{GilDebRos} proved that the efficiency of  the Multilevel Monte Carlo (MLMC)  method for evaluating Down-and-Out barrier options  for a diffusion process $(X_t)_{t\in[0,T]}$ with globally Lipschitz coefficients, can be  improved by combining a  Brownian bridge technique and a conditional Monte Carlo method provided that the running minimum $\inf_{t\in[0,T]}X_t$ has a bounded density in the vicinity of the barrier. 
In the present work, thanks to the Lamperti transformation technique and using a Brownian interpolation of the drift implicit Euler scheme of Alfonsi \cite{alfonsi2013strong}, we show that  the efficiency of the MLMC  can be also improved for the evaluation of barrier options  for models with  non-Lipschitz diffusion coefficients under certain moment constraints. We study two example models: the Cox-Ingersoll-Ross (CIR)  and the Constant  of Elasticity of Variance  (CEV) processes for which we show that the conditions of our  theoretical framework are satisfied  under certain restrictions on the models parameters. In particular, we develop semi-explicit formulas for the densities of the running minimum and running maximum of both CIR and CEV processes which are of independent interest. Finally, numerical tests are processed to illustrate our results.
 \end{abstract}

\section{Introduction}
Barrier options are one of the most widely traded exotic options in the financial markets. Pricing and hedging such path-dependent option can quickly become very challenging especially when we need to achieve a good precision for the approximation. Evaluating barrier options by a classic Monte Carlo method introduces a systematic bias when approximating the continuous running maximum (resp. minimum) in the crossing-barrier indicator by  a discrete running maximum (resp. minimum). To overcome this difficulty, several numerical strategies exist in the literature among them the popular Brownian bridge technique  introduced in \cite{Baldi}  well known for its efficiency and ease of use (see also \cite{GOBET} for  related refinements).   The Brownian bridge technic uses an analytic expression for the probability of hitting the barrier between two known values in a simulated  path of the underlying asset. More recently, a combination of the Multilevel Monte Carlo (MLMC) method with the Brownian bridge technique has been developed in \cite{GilDebRos} for pricing barrier options. The Multilevel Monte Carlo method introduced in Giles \cite{giles2008} as an extension of the two-level Monte Carlo method  of \cite{Keb2005}, significantly reduces the time complexity of the classical Monte Carlo method. More precisely, for a given precision $\varepsilon>0$ and a Lipschitz payoff function, if  the underlying asset $(X_t)_{t\in[0,T] }$ is approximated using a discretization scheme $(\bar X_t)_{{t\in[0,T] }}$  with time step $h>0$ satisfying $\mathbb E|X_t-\bar X_t|^2=O(h^{\beta})$ and $|\mathbb E[X_t-\bar X_t]|=O(h^{\alpha})$ with $\alpha\ge \frac12$, then the time complexity of the MLMC methods is:  $O({\varepsilon}^{-2})$ when $\beta>1$,  $O({\varepsilon}^{-2}(\log \varepsilon)^2)$ when $\beta=1$ and  $O({\varepsilon}^{-2-\frac{1-\beta}{\alpha}})$ when $\beta\in(0,1)$. However, for the same precision $\varepsilon>0$  the optimal time complexity of a classic Monte Carlo method is $O({\varepsilon}^{-3})$. As the payoff function of a barrier option is not Lipschitz, Giles et al. \cite{GilDebRos} take advantage of the Brownian bridge to run the MLMC method for pricing such options, since this technique substitutes the barrier-crossing indicators by the probabilities that the approximation scheme $(\bar X_t)_{t\in[0,T] }$ hits the barrier between each two consecutive discretization times $t_i$ and $t_{i+1}$ and which are represented  as smooth functions of the realized points $\bar X_{t_i}$ and $\bar X_{t_{i+1}}$.  More precisely,  Giles et al. \cite{GilDebRos} consider an underlying asset  solution to a one-dimensional  stochastic differential equation (SDE) with globally Lipschitz smooth coefficients that is approximated by a high order strong approximation scheme namely the Milstein scheme  
$(\bar X^{\text{Milstein}}_t)_{{t\in[0,T] }}$  that satisfies $\mathbb E|X_t-\bar X^{\text{Milstein}}_t|^2=O(h^{2})$. For this case, they prove that the MLMC method reaches its optimal time complexity $O({\varepsilon}^{-2})$ for pricing a Down-and-Out barrier option \footnote{A Down-and-Out barrier Call (resp.~Put) is the option to buy (resp.~sell), at maturity $T$,  the underlying with a fixed strike if the underlying value never falls below the barrier before time $T$.} provided that $\inf_{t\in[0,T]}X_t$ has a bounded density in the neighborhood of the barrier.  This latter condition cannot be  easily checked even  when the SDE coefficients are Lipschitz except for very specific cases.  

 In the current paper, we are interested in studying the MLMC method for pricing barrier options  when the underlying asset is solution to a SDE with a non-Lipschitz diffusion coefficient  such as the popular Cox-Ingersoll-Ross (CIR) and the Constant  of Elasticity of Variance (CEV) processes. Only few works exist in the literature that studied the problem of pricing path-dependent options under such singular models (see e.g.~\cite{Davydov}).  To analyze the performance of the MLMC method, we consider in Section \ref{sec:GF} a general framework of  models with  non-Lipschitz  diffusion coefficients and use a Lamperti transformation to focus our study  on a new process $(Y_t)_{t\in[0,T]}$ with an additive noise diffusion but in counterpart with a possibly singular drift coefficient $L$.   Then, we introduce a Brownian interpolation scheme  $(\bar Y_t)_{t\in[0,T]}$ associated  to the drift implicit Euler scheme of Alfonsi \cite{alfonsi2013strong} for which we prove a strong convergence result with order one (see Theorem \ref{hypc}). In Section \ref{sec:MLMC}, we use the Brownian bridge technic that substitutes the crossing-indicators with smooth functions of realized points in the path of the scheme $(\bar Y_t)_{t\in[0,T]}$ to build  the corresponding  MLMC estimator. Next, under suitable assumptions on the drift $L$, we prove that the obtained MLMC method   for pricing Down-and-Out  (resp.~Up-and-Out) reaches its optimal time complexity  $O({\varepsilon}^{-2})$  provided that  $\inf_{t\in[0,T]}Y_t$  (resp. $\sup_{t\in[0,T]}Y_t$) has a bounded density in the neighborhood of the barrier (see Theorem \ref{MLMC:var} and Remark \ref{rem:param}).   In Sections  \ref{sec:CIR} and \ref{sec:CEV}, we provide two examples of processes satisfying  Theorem \ref{MLMC:var} conditions, namely  the CIR and the CEV  models.  It turns out that under additional constraints on the parameters of these two models ensuring the existence of finite negative moments up to a certain order, the MLMC method behaves exactly like a classical unbiased Monte Carlo estimator despite the use of approximation schemes.
To show that the conditions of our  theoretical framework are satisfied for these two models, we develop using fine asymptotic properties of confluent hypergeometric type functions, semi-explicit formulas for the densities of the running minimum and running maximum of both CIR and CEV processes which are  of independent interest (see Theorems \ref{thm:CIR-max}, \ref{thm:CIR-min}, \ref{thm:CEV-max} and \ref{thm:CEV-min}).  Finally, we proceed to several numerical tests illustrating our results.

\section{General framework}\label{sec:GF}
Let us consider  a process $(X_t)_{t\in[0,T]}$ solution to 
\begin{equation}
\label{eq1}
dX_t=b(X_t)dt+ \sigma(X_t) dW_t, 
\quad X_0=x,
\end{equation}
where $(W_t)_{t\ge 0}$ is a standard Brownian motion, $b: \R \rightarrow \R$ and  $\sigma:\R \rightarrow \R^*_+$ are locally Lipschitz-functions such that  $\frac 1{\sigma}$  is locally integrable. For $\phi(y)=\int_{y_0}^y\frac1{\sigma(x)}dx$, if $\sigma\in \mathcal C^1$ then by the Lamperti transform  $Y_t=\phi(X_t)$ solves the stochastic differential equation 
\begin{equation*}
 dY_t=L(X_t)dt+ dW_t, \quad
Y_0=\phi(x),
\end{equation*} 
with $L(x)=\left(\frac{b}{\sigma}- \frac{\sigma '}2\right)(\phi^{-1}(x))$. In this work, we are interested in approximating barrier option prices such as the Down-and-Out (D-O) and  the Up-and-Out (U-O) barrier options 
$$\pi_{{\mathcal B}_{D}}=\mathbb{E}\Big[f(X_T)\mathbbm{1}_{\{ \inf_{t\in[0,T]}X_t>{{\mathcal B}_{D}}\}}\Big] \mbox { and } \pi_{{\mathcal B}_{U}}=\mathbb{E}\Big[f(X_T)\mathbbm{1}_{\{\sup_{t\in[0,T]}X_t<{\mathcal B}_{U}\}}\Big].$$
The other types of barrier options such as the Down-and-In and the Up-and-In can be easily deduced from the  price of the vanilla option $\mathbb E[f(X_T)]$. As the function $\phi$ is monotonic,  by the  Lamperti transformation we reduce ourselves to a pricing problem with the process $(Y_t)_{t\in[0,T]}$. More precisely, we get $\pi_{{\mathcal B}_{D}}=\pi_{\mathcal D}$ and 
 $\pi_{{\mathcal B}_{U}}=\pi_{\mathcal U}$ where 
$$\pi_{{\mathcal D}}=\mathbb{E}\Big[g(Y_T)\mathbbm{1}_{\{ \inf_{t\in[0,T]}Y_t>{{\mathcal D}}\}}\Big], \;\,\pi_{{\mathcal U}}=\mathbb{E}\Big[g(Y_T)\mathbbm{1}_{\{\sup_{t\in[0,T]}Y_t<{\mathcal U}\}}\Big],$$
$g(x)=f\circ \phi^{-1}(x)$, $\mathcal D=\phi({\mathcal B}_{D})$ and $\mathcal U=\phi({\mathcal B}_{U})$. In the sequel,  we consider the general setting given in \cite{alfonsi2013strong} and let $(Y_t)_{t\ge0}$ denote  the SDE defined on $I=(0, +\infty)$ solution to
\begin{align}
\label{lam}
    dY_t&=L(Y_t)dt+ \gamma dW_t, \quad t\ge 0,\;Y_0=y\in I,\mbox{ with } \gamma\in \mathbb R^*,
\end{align}
where the drift coefficient $L$ is supposed to satisfy the following monotonicity assumption:
\begin{equation}
\label{ine1}
L: I\longrightarrow \R~~is~~C^2,~~\mbox{such that} ~~\exists~\kappa>0,~~ \forall y,y'\in I, y\leq y', L(y')-L(y)\leq \kappa(y'-y). 
\end{equation}
In addition, for an arbitrary point $d\in I$, we assume that 
\begin{equation}
\label{ine2}
v(x)=\displaystyle \int_d^x \int_d^y \exp\Big(-\frac{2}{\gamma^2}\displaystyle \int_z^y L(\xi) d\xi \Big) dz dy\;\text{ satisfies } \lim\limits_{x \to 0^+} v(x) =+\infty.\tag{H1}
\end{equation}
On the one hand, by  the Feller's test (see  e.g. \cite{karatzas1991brownian}), \eqref{ine1} and \eqref{ine2} ensure that the SDE \eqref{lam} admits a unique strong solution $(Y_t)_{t\ge 0}$ on $I$ that never reaches  the boundaries $0$ and  $+\infty$. 
On the other hand,  under these two conditions the below drift implicit continuous scheme introduced in \cite{alfonsi2013strong},
\begin{align}\label{sch:DI}
\widehat{Y}^n_{t}&=\widehat{Y}^n_{t_{i}}+L(\widehat{Y}^n_{t})(t-t_i)+\gamma (W_{t}-W_{t_i}),  \mbox{ with }t\in(t_i,t_{i+1}], t_i=\frac{i T}{n},\;0 \leq i \leq n-1,\\
\widehat{Y}^n_0&=y.\notag
\end{align}
 is well defined and for all $t\in[0,T]$, $Y^n_t\in I$.   Besides, if in addition we assume that for $p\ge 1$, we have
 \begin{align}\label{ine3}
    \mathbb{E}\Big[\Big(\displaystyle \int_0^T\vert L'(Y_u)L(Y_u)+\frac{\gamma^2}{2}L''(Y_u)\vert du \Big)^p\Big] < \infty~~and~~\mathbb{E}\Big[\Big(\displaystyle \int_0^T (L'(Y_u))^2 du \Big)^{\frac{p}{2}}\Big]< \infty,\tag{H2}
\end{align}
then by \cite{alfonsi2013strong},  there exists a positive constant $K_p$ such that 
$$
\mathbb E^{\frac 1p}\left[\sup_{t\in[0,T]}|\widehat{Y}^n_t-Y_t|^p\right]\le K_p\frac T n.
$$
For our purpose, we rather focus on a slightly different interpolated version of the drift implicit scheme. More precisely, we first introduce the discrete version of the drift implicit scheme given by 
\begin{equation}
 \label{scheme}
\left \{
\begin{array}{rcl}
\overline{Y}^n_{t_{i+1}}&=&\overline{Y}^n_{t_{i}}+L(\overline{Y}^n_{t_{i+1}})\frac Tn+\gamma (W_{t_{i+1}}-W_{t_i}),  \mbox{ with }\; t_i=\frac{i T}{n},\;0 \leq i \leq n-1,\\\\
\overline{Y}^n_0&=&y.
\end{array}
\right.\\
\end{equation}
% is well defined.  
%
%
%\subsection{A Brownian interpolation of the drift implicit scheme}
% We consider a time horizon $T>0$ and a regular time grid $t_i=\frac{i T}{n}$, $0 \leq i \leq n.$
% We approximate $(Y_t)_{0\le t\le T}$ solution to \eqref{lam} using the drift implicit Euler scheme proposed by \cite{alfonsi2005discretization}.
% 
 %We will approach $\mathbb{E} \Big[g(Y_T)\mathbbm{1}_{\{\inf_{t\in[0,T]} Y_t>B'\}}\Big]$by $\mathbb{E} \Big[g(\widehat{Y}_T^n)\mathbbm{1}_{\{\inf_{t\in[0,T]} \widehat{Y}_t^n>B'\}}\Big] $ to determine the price of a barrier option, where $\widehat{Y}_t^n$ is defined in \eqref{scheme},\\
% and satisfies $\widehat{Y_t}^n \in I.$
%To be able to use the Brownian bridge technique, known for pricing barrier options under classical stochastic differential equations that have strictly positive diffusion coefficients, for stochastic differential equations with non-Lipschitz diffusion coefficient, we introduce the following brownian interpolation of the drift implicit
%scheme. 
%   \begin{align}
%\label{interpolant}
%    \overline{Y}_t= \widehat{Y}_{t_{i}} +\lambda(t)(\widehat{Y}_{t_{i+1}}-\widehat{Y}_{t_{i}})+\gamma \Big( W_t-W_{t_i}-\lambda(t)(W_{t_{i+1}}-W_{t_i}) \Big),
%\end{align}
%%\label{schema1}
%where $ \lambda(t)= (t-t_i)/h$, for $t\in [t_i,t_{i+1}]$.
%We easily observe that  $\overline{Y}_{t_i}=\widehat{Y}_{t_i}^n$, $\forall\; 0\leq i \leq n$   and using  \eqref{scheme},  $\overline Y_t$  rewrites
and introduce the following interpolated drift implicit scheme
\begin{align}
\label{interp}
    \overline{Y}^n_t=\overline{Y}^n_{t_{i}}+L(\overline{Y}^n_{t_{i+1}})(t-t_{i})+\gamma(W_t-W_{t_{i}}),~~~~ \text{for}\, t \in [t_i,t_{i+1}[, \;\;0 \leq i \leq n-1.
\end{align}
The main advantages of this Brownian interpolation is that it preserves the rate of strong convergence of the original drift implicit scheme \eqref{sch:DI} and allows at the same time the use of the Brownian bridge technique for pricing Barrier options (see Section \ref{subsec:BB} below).
In what follows, we strengthen our assumption on the drift coefficient $L$ as follows:
\begin{multline}
\label{ine4}
L : I \to \mathbb R \mbox{ is }  \mathcal C^2 \mbox{ such that: }   L \mbox{ is decreasing on } (0,A) \mbox{ for }A>0, \\ \mbox{ and } L'  \text{ the first derivative of } L \mbox{ satisfies }  \exists  L'_A>0 \mbox{ s.t. }  \forall y\in(A,\infty), \;|L'(y)|\le L'_A.\tag{H3}
\end{multline}
\begin{thm}
\label{hypc}
Assume that conditions \eqref{ine3} and \eqref{ine4} hold true for a given $p>1$ and with $L'_A<\frac n{2T}$.  Then, there exists a constant $K_p>0$ such that
$$\mathbb{E}^{\frac 1p}\Big[\sup_{t \in [0,T]} \vert \overline{Y}^n_t - Y_t \vert^p \Big] \leq K_p \frac{T}{n}.$$
\end{thm}
\begin{proof} At first, for $p\ge 1$ and $t\in[0,T]$, we denote  $e_{t} = \overline{Y}^n_{t} - Y_{t}$.
By \eqref{lam}, we have for $0\le i\le n-1$,
\begin{align*}
e_{t_{i+1}} &= e_{t_{i}}+ L(\overline{Y}^n_{t_{i+1}}) (t_{i+1} - t_{i})  - \int_{t_{i}}^{t_{i+1}} L(Y_{s}) ds
%\\
%&=\bar{Y}_{t_{i}}^{n}+ L(\bar{Y}_{t_{i+1}}^{n}) (t_{i+1} - t_{i}) + \gamma (W_{t_{i+1}} - W_{t_{i}}) - Y_{t_{i}} - \int_{t_{i}}^{t_{i+1}} L(Y_{s}) ds - \gamma (W_{t_{i+1}} - W_{t_{i}}).
\end{align*}
since  for all  $0\le i\le n-1$, $\overline Y^n_{t_i}\in I$.
%then 
%from \eqref{ine1} there is $\beta_{t_{i+1}} \leq \kappa$ such that $L(\overline{Y}_{t_{i+1}})-L(Y_{t_{i+1}})=\beta_{t_{i+1}} e_{t_{i+1}}$ where $\beta_{t_{i+1}}=\int_0^1 L'\big( \theta Y_{t_{i+1}}+(1-\theta)\overline{Y}_{t_{i+1}}\big)d\theta$, for $0\leq i \leq n-1.$ 
As $L$ is of class $\mathcal C^2$ there exists a point  $\xi_{t_{i+1}}$ lying between $Y_{t_{i+1}}$ and $\overline{Y}^n_{t_{i+1}}$ such that 
 $L(\overline{Y}^n_{t_{i+1}})-L(Y_{t_{i+1}})= \beta_{t_{i+1}}(\overline{Y}^n_{t_{i+1}}-Y_{t_{i+1}})$ with $\beta_{t_{i+1}}=  L'\big( \xi_{t_{i+1}}\big)$. Besides,  according to the proof  \cite[Proposition 3]{alfonsi2013strong}, we know that
\begin{multline}
\label{discret}
\mathbb{E}\Big[\sup_{1\leq i\leq n}|e_{t_i}|^p\Big]
\leq K\Big(\frac{T}{n}\Big)^p\bigg[\mathbb{E}\Big[\Big(\int_0^{T}|L'(Y_u)L(Y_u)+\frac{\gamma^2}{2}L''(Y_u)|du\Big)^p\Big]\\+|\gamma|^p \mathbb{E}\Big[\Big(\int_0^{T}(L'(Y_u))^2du\Big)^{\frac{p}{2}}\Big]\bigg],
\end{multline}
where $K$ is a positive constant that depends on $T$ and $p$.
On the one hand,  we first use   \eqref{interp} to write
\begin{align*}
e_{t_{i+1}}&=e_{t_i}+L(\overline{Y}^n_{t_{i+1}})(t_{i+1}-t_i)-\int_{t_i}^{t_{i+1}} L(Y_s) ds\\
&=e_{t_i}+\Big[L(\overline{Y}^n_{t_{i+1}})-L(Y_{t_{i+1}})\Big](t_{i+1}-t_i)+\int_{t_i}^{t_{i+1}} \big(L(Y_{t_{i+1}})-L(Y_s)\big)ds\\
&=e_{t_i}+\beta_{t_{i+1}}e_{t_{i+1}}(t_{i+1}-t_i)+\int_{t_i}^{t_{i+1}} \big(L(Y_{t_{i+1}})-L(Y_s)\big)ds.
\end{align*}
It follows that
\begin{align}\label{et}
\big(1-\beta_{t_{i+1}}(t_{i+1}-t_i)\big)e_{t_{i+1}}&=e_{t_i}+\int_{t_i}^{t_{i+1}} \big(L(Y_{t_{i+1}})-L(Y_s)\big)ds.
\end{align}
%since $n> 2\kappa T$, we get
%\begin{align}
%\label{et}
%e_{t_{i+1}}=\frac{1}{(1-\beta(t_{i+1}-t_i)\big)}e_{t_i}+\frac{1}{(1-\beta(t_{i+1}-t_i)\big)}\int_{t_i}^{t_{i+1}} \big(L(Y_{t_{i+1}})-L(Y_s)\big)ds
%\end{align}
On the other hand, we have  for all $t\in[t_i,t_{i+1})$
\begin{align*}
\overline{Y}^n_t&=\overline{Y}^n_{t_{i}}+L(\overline{Y}^n_{t_{i+1}})(t-t_{i})+\gamma(W_t-W_{t_{i}})\\
&=\overline{Y}^n_{t_{i}}+L(\overline{Y}^n_{t_{i+1}})(t_{i+1}-t_{i})+\gamma(W_{t_{i+1}}-W_{t_{i}})-L(\overline{Y}^n_{t_{i+1}})(t_{i+1}-t)-\gamma(W_{t_{i+1}}-W_{t})\\
&=\overline{Y}^n_{t_{i+1}}-L(\overline{Y}^n_{t_{i+1}})(t_{i+1}-t)-\gamma(W_{t_{i+1}}-W_{t}).
\end{align*}
Then, it follows that for all $t\in[t_i,t_{i+1})$
\begin{align*}
\overline{Y}^n_t-Y_t&=\overline{Y}^n_{t_{i+1}}-Y_{t_{i+1}}+Y_{t_{i+1}}-Y_t-L(\overline{Y}^n_{t_{i+1}})(t_{i+1}-t)-\gamma(W_{t_{i+1}}-W_{t})\\
e_t&=e_{t_{i+1}}+ \int_t^{t_{i+1}}L(Y_s)ds+\gamma (W_{t_{i+1}}-W_{t})-L(\overline{Y}^n_{t_{i+1}})(t_{i+1}-t)-\gamma(W_{t_{i+1}}-W_{t})\\
&=e_{t_{i+1}}-\big(L(\overline{Y}^n_{t_{i+1}})-L(Y_{t_{i+1}})\big)(t_{i+1}-t)+\int_t^{t_{i+1}} L(Y_s)-L(Y_{t_{i+1}}) ds\\
&=e_{t_{i+1}}-\beta_{t_{i+1}}(\overline{Y}^n_{t_{i+1}}-Y_{t_{i+1}})(t_{i+1}-t)+\int_t^{t_{i+1}} L(Y_s)-L(Y_{t_{i+1}}) ds.
\end{align*}
So, we deduce  that for all $t\in[t_i,t_{i+1})$
\begin{align}
\label{et1}
e_t&=(1-\beta_{t_{i+1}}(t_{i+1}-t))e_{t_{i+1}}+\int_t^{t_{i+1}} L(Y_s)-L(Y_{t_{i+1}}) ds.
\end{align}
%Then, we replace the expression of $e_{t_{i+1}}$ by \eqref{et} we obtain
By assumption \eqref{ine4},  on $(0,A)$ $L$ is decreasing, so
it is easy to see that  $1<1-\beta_{t_{i+1}}(t_{i+1}-t)<1-\beta_{t_{i+1}}(t_{i+1}-t_i)$.
On $ (A, \infty)$, as $L'$ is bounded and since $n> 2 L'_A T$,
we have $|1-\beta_{t_{i+1}}(t_{i+1}-t)| \leq \frac 3 2$ and $1-\beta_{t_{i+1}}(t_{i+1}-t_i)> \frac{1}{2}$.  Then, it follows that  $\Big|\frac{1-\beta_{t_{i+1}}(t_{i+1}-t)}{1-\beta_{t_{i+1}}(t_{i+1}-t_i)}\Big|\le 3.$
Now, combining \eqref{et} and \eqref{et1} we easily get 
\begin{align}
\label{et2}
e_t=\frac{1-\beta_{t_{i+1}}(t_{i+1}-t)}{1-\beta_{t_{i+1}}(t_{i+1}-t_i)}\left(e_{t_{i}}+\int_{t_i}^{t_{i+1}} \big(L(Y_{t_{i+1}})-L(Y_s)\big)ds\right)+\int_t^{t_{i+1}} L(Y_s)-L(Y_{t_{i+1}}) ds.
\end{align}
Then, by Itô's formula and Fubini theorem we get
\begin{align*}
|e_t|&\leq 3 \left(|e_{t_i}|+ \Big|\int_{t_i}^{t_{i+1}} \big(L(Y_{t_{i+1}})-L(Y_s)\big)ds \Big|\right)+ \Big|\int_t^{t_{i+1}} L(Y_s)-L(Y_{t_{i+1}}) ds \Big|\\
&\leq 3\left(|e_{t_i}|+\frac{T}{n}\int_{t_i}^{t_{i+1}}\Big| L'(Y_u)L(Y_u)+\frac{\gamma^2}{2}L''(Y_u)\Big| du+|\gamma| \Big|\displaystyle \int_{t_i}^{t_{i+1}} (u-t_i)L'(Y_u)dW_u \Big|\right)\\
&+ \frac{T}{n}\int_t^{t_{i+1}}\Big| L'(Y_u)L(Y_u)+\frac{\gamma^2}{2}L''(Y_u)\Big| du + |\gamma| \Big|\displaystyle \int_{t}^{t_{i+1}} (u-t)L'(Y_u)dW_u \Big|. 
\end{align*}
Therefore, there exists a positive constant $C_p$ such that
\begin{align*}
|e_t|^p &\leq C_p \Bigg[ |e_{t_i}|^p +2\Big(\frac{T}{n}\Big)^p \Big(\int_{t_i}^{t_{i+1}}\Big| L'(Y_u)L(Y_u)+\frac{\gamma^2}{2}L''(Y_u)\Big| du \Big)^p
+|\gamma|^p \Big|\displaystyle \int_{t_i}^{t_{i+1}} (u-t_i)L'(Y_u)dW_u \Big|^p\\&+ |\gamma|^p \Big|\displaystyle \int_{t}^{t_{i+1}} uL'(Y_u)dW_u \Big|^p+|\gamma|^p T^p\Big|\displaystyle \int_{t}^{t_{i+1}} L'(Y_u)dW_u \Big|^p\Bigg]
\end{align*}
and thus,
\begin{align*}
\displaystyle \sup_{t \in [0,T]} \vert e_t \vert^p &\leq C_p\Bigg[\sup_{0\leq i \leq n}|e_{t_i}|^p+2\Big(\frac{T}{n}\Big)^p \Big(\displaystyle \int_0^T\Big| L'(Y_u)L(Y_u)+\frac{\gamma^2}{2}L''(Y_u)\Big| du \Big)^p\\
&+ |\gamma|^p \displaystyle \sup_{0 \leq s \leq t \leq T} \Big| \displaystyle \int_s^t (u-t_{\eta(u)})L'(Y_u)dW_u \Big|^p+|\gamma|^p \displaystyle \sup_{0 \leq s \leq t \leq T} \Big| \displaystyle \int_s^t uL'(Y_u)dW_u \Big|^p\\&+|\gamma|^p T^p\displaystyle \sup_{0 \leq s \leq t \leq T} \Big|\displaystyle \int_{s}^{t} L'(Y_u)dW_u\Big|^p\Bigg] \\
& \leq  C_p\Bigg[\sup_{0\leq i \leq n}|e_{t_i}|^p+2\Big(\frac{T}{n}\Big)^p \Big(\displaystyle \int_0^T\Big| L'(Y_u)L(Y_u)+\frac{\gamma^2}{2}L''(Y_u)\Big| du \Big)^p\\
&+ 2^{p-1}|\gamma|^p \displaystyle \sup_{0  \leq t \leq T} \Big| \displaystyle \int_0^t (u-t_{\eta(u)})L'(Y_u)dW_u \Big|^p+2^{p-1}|\gamma|^p \displaystyle \sup_{0  \leq t \leq T} \Big| \displaystyle \int_0^t uL'(Y_u)dW_u \Big|^p\\&+2^{p-1}|\gamma|^p T^p\displaystyle \sup_{0 \leq t \leq T} \Big|\displaystyle \int_{0}^{t} L'(Y_u)dW_u\Big|^p\Bigg].
\end{align*}
The result follows using \eqref{discret}  and the  Burkholder-Davis-Gundy inequality. 
%The extension to the case $0<p<1$ can be carried out using that  $p\mapsto \| \cdot\|_p$ is non-decreasing.
%inequality we obtain
%\begin{align*}
%    \mathbb{E}\Big[\displaystyle \sup_{t \in [0,T]} \vert e_t \vert^p\Big] &\leq 
%    C_p\Bigg[ (2(1+\kappa))^p \mathbb{E}\Big[\displaystyle \sup_{0\leq i \leq n} \vert e_{t_i} \vert^p\Big]\\
%    &+\big((2(1+\kappa))^p+1\big)\Big(\frac{T}{n}\Big)^p\mathbb{E}\Big[\Big(\displaystyle \int_0^T\Big| L'(Y_u)L(Y_u)+\frac{\gamma^2}{2}L''(Y_u)\Big| du \Big)^p\Big]\\
%    &+ 2^{p+1}C'_p \gamma^p \Big(\frac{T}{n}\Big)^p \mathbb{E}\Big[\Big(\int_0^T (L'(Y_u))^2 du \Big)^{p/2}\Big]\Bigg].
%\end{align*} 
%From \eqref{discret} and the fact that we have  $\mathbb{E}\Big[\Big(\displaystyle \int_0^T\Big| L'(Y_u)L(Y_u)+\frac{\gamma^2}{2}L''(Y_u)\Big| du \Big)^p\Big]< \infty$ and $\mathbb{E}\Big[\Big(\int_0^T (L'(Y_u))^2 du \Big)^{p/2}\Big]< \infty$ we conclude that 
%$$\mathbb{E}\big[\displaystyle \sup_{t \in[0, T]} |e_t|^p\big]\leq K'_p \Big(\frac{T}{n}\Big)^p$$
\end{proof}   
\begin{cor}
\label{cor0}
Assume that conditions of Theorem \ref{hypc} hold true for a given $p>1$ and $0<L'_A<\frac n{2T}$.
If in addition  the drift coefficient $L$ satisfies the following one-sided linear growth assumption:
\begin{equation}
\label{ine5}
\exists\, \alpha>0 \mbox{ such that } \;\forall\,y\in I,\;\; 
yL(y)\le \alpha(1+|y|^2)\tag{H4}
\end{equation}
then  $\mathbb{E}[\displaystyle\sup_{0\leq t \leq T} \vert \overline{Y}^n_t \vert ^p ]< \infty$.
\end{cor}
\begin{proof}
Under assumption \eqref{ine5}, \cite[Lemma 3.2]{HigMaoStu} ensures that for all $q>0$ 
$\mathbb E[\sup_{0\le t\le T} |Y_t|^q]<\infty$. Thus, by Theorem \ref{hypc} we get
$$\mathbb{E}[\displaystyle\sup_{0\leq t \leq T} \vert \overline{Y}^n_t \vert ^p ]\le 2^{p-1}\left(\mathbb E[\sup_{0\le t\le T} |Y_t|^p+  \mathbb{E}[\displaystyle\sup_{0\leq t \leq T} \vert \overline{Y}^n_t - \overline{Y}_t\vert ^p ]\right)<\infty.$$
\end{proof}

%\begin{proof}
%The result is a consequence of the above theorem indeed, we have
%\begin{equation}
%\left.\begin{array}{rl}
%\vert \overline{Y}_t \vert &\leq  \vert \overline{Y}_t-Y_t \vert + \vert Y_t \vert\\
%\vert \overline{Y}_t \vert^p & \leq 2^p \big( \vert \overline{Y}_t-Y_t \vert ^p+ \vert Y_t \vert^p \big)\\
%\mathbb{E}\big[\displaystyle\sup_{0\leq t \leq T} \vert \overline{Y}_t \vert ^p \big]& \leq  2^p \Big(\mathbb{E}\big[\displaystyle\sup_{0\leq t \leq T} \vert \overline{Y}_t-Y_t \vert ^p \big]+ \mathbb{E}\big[\displaystyle\sup_{0\leq t \leq T} \vert Y_t \vert ^p\big]\Big)
%\end{array}\right.
%\end{equation}
%On account of the above theorem and lemma \eqref{lem0}, we get  $\mathbb{E}[\displaystyle\sup_{0\leq t \leq T} \vert \overline{Y}_t \vert ^p ]< \infty$, $\forall p\geq 1$.
%
%\end{proof}
\section{The Multilevel Monte Carlo method for pricing Barrier options with the interpolated drift implicit scheme.}\label{sec:MLMC}
We are interested to approximate the following quantities of interest  of the form
$$\pi_{\mathcal D}=\mathbb{E}\Big[g(Y_T)\mathbbm{1}_{\{\tau_{\mathcal D} >T\}}\Big] \mbox { and } \pi_{\mathcal U}=\mathbb{E}\Big[g(Y_T)\mathbbm{1}_{\{\tau_{\mathcal U} >T\}}\Big],$$ where $\tau$ denotes the first passage time given by:
\begin{itemize}
\item   $\tau_{\mathcal D}=\displaystyle\inf\{t \in [0,T],Y_t \le {\mathcal D}\}$ with $y>{\mathcal D}>0$,  for a Down-Out (D-O) option,
 \item[] and
 \item   $\tau_{\mathcal U}=\displaystyle\inf\{t \in [0,T],Y_t \ge {\mathcal U}\}$ with $0<y<{\mathcal U}$, for  an Up-Out (U-O) option.
\end{itemize}
\subsection{Brownian bridge  and drift implicit scheme for pricing Barrier options}\label{subsec:BB}
For a given time grid $t_i=\frac{iT}{n}$, $i\in\{1,\dots, n\}$, we consider the Brownian interpolation of the drift implicit scheme $(\overline{Y}^n_{t})_{t\in[0,T]}$ defined in  \eqref{interp}. Then, the above option prices can be approximated respectively by 
\begin{align*}
{\overline \pi}_{\mathcal D}:=\mathbb{E}\Big[g(\overline{Y}^n_T)\displaystyle \prod_{i=0}^{n-1}\mathbf{1}_{\{\inf_{t \in[t_i,t_{i+1}]} \overline{Y}^n_t>{\mathcal D}\}}\Big]
\mbox{ and }
{\overline \pi}_{\mathcal U}:=\mathbb{E}\Big[g(\overline{Y}^n_T)\displaystyle \prod_{i=0}^{n-1}\mathbf{1}_{\{\sup _{t \in[t_i,t_{i+1}]} \overline{Y}^n_t<{\mathcal U}\}}\Big].
\end{align*}
 To get more accurate approximations, we use the Brownian bridge technique to substitute the barrier-crossing indicators by the probabilities that the approximation scheme $(\overline{Y}^n_{t})_{t\in[0,T]}$ do not cross the barrier in each time interval $[t_i,t_{i+1}]$, $i\in\{1,\dots, n\}$. In what follows, for $x\in \mathbb R$, $(x)_+$ stands for $\sup(x,0)$.
\begin{prop}\label{lem:BB} Under the above notation, 
for $h=\frac T n$, we have 
$${\overline \pi}_{\mathcal D}=\mathbb{E} \Big[g(\overline{Y}^n_T)\prod_{i=0}^{n-1}(1-\overline{q}_i)\Big],
\mbox{ where } \overline{q}_i:=\exp\Big(\frac{-2(\overline{Y}^n_{t_i}-{\mathcal D})_+(\overline{Y}^n_{t_{i+1}}-{\mathcal D})_+}{\gamma^2 h}\Big)$$
and
$${\overline \pi}_U=\mathbb{E} \Big[g(\overline{Y}^n_T)\prod_{i=0}^{n-1}(1-\overline{p}_i)\Big],
\mbox{ where } \overline{p}_i=\exp\Big(\frac{-2({\mathcal U}- \overline{Y}^n_{t_i})_+({\mathcal U}- \overline{Y}^n_{t_{i+1}})_+}{\gamma^2 h}\Big).$$
\end{prop}
\begin{proof} For the U-O barrier option, we first  notice that conditionally on $(\overline{Y}^n_{0},\overline{Y}^n_{t_1},\dots,\overline{Y}^n_{T})$,  the barrier-crossing indicators  $(\mathbf{1}_{\{\sup _{t \in[t_i,t_{i+1}]} \overline{Y}^n_t<{\mathcal U}\}}, i\in\{1,\dots, n\})$ are independent, we write
\begin{align*}
{\overline \pi}_{\mathcal U}&=\mathbb{E}\Big[g(\overline{Y}^n_T)\mathbb{E}\big[\prod_{i=0}^{n-1}\mathbbm{1}_{\{\sup _{t \in[t_i,t_{i+1}]} \overline{Y}^n_t<{\mathcal U}\}}|\overline{Y}^n_{0},\overline{Y}^n_{t_1},\dots,\overline{Y}^n_{T}\big]\Big]\\
&=\mathbb{E}\Big[g(\overline{Y}^n_T)\prod_{i=0}^{n-1}\mathbb{E}\big[\mathbbm{1}_{\{\sup_{t \in[t_i,t_{i+1}]} \overline{Y}^n_t<{\mathcal U}\}}|\overline{Y}^n_{t_i},\overline{Y}^n_{t_{i+1}}\big]\Big]\\
&=\mathbb{E}\Big[g(\overline{Y}^n_T)\prod_{i=0}^{n-1}\big(1-\varphi(\overline{Y}^n_{t_i},\overline{Y}^n_{t_{i+1}})\big)\Big],
\end{align*}
where, for $y_i,y_{i+1}\in I$, $\varphi(y_i,y_{i+1})=\mathbb{P}\big(\sup _{t \in[t_i,t_{i+1}]}\overline{Y}^n_t \ge {\mathcal U}|\overline{Y}^n_{t_i}={y_i},\overline{Y}^n_{t_{i+1}}=y_{i+1}\big)$. Without loss of generality, we may assume $\gamma>0$, the same arguments below work for $\gamma<0$ using that $(W_t)_{t\ge0}$ and $(-W_t)_{t\ge0}$ have the same law. By\eqref{interp}, we write 
\begin{align*}
\sup_{t\in[t_i,t_{i+1}]}\overline{Y}^n_t&= \overline{Y}^n_{t_i}+ \gamma \sup_{t\in[t_i,t_{i+1}]}\left[W_t-W_{t_i}+\frac{1}{\gamma}L(\overline{Y}^n_{t_{i+1}})(t-t_i)\right],\mbox{ with}\\
W_{t_{i+1}}-W_{t_i}&+\frac{1}{\gamma}L(\overline{Y}^n_{t_{i+1}})(t_{i+1}-t_i)=\frac{1}{\gamma}(\overline{Y}^n_{t_{i+1}}-\overline{Y}^n_{t_i}).
\end{align*}
By the stationarity property of the brownian increments and using a change of probability measure, we easily get that the law of 
$\sup_{t\in[t_i,t_{i+1}]}\left[W_t-W_{t_i}+\frac{1}{\gamma}L(y_{i+1})(t-t_i)\right]$ given $W_{t_{i+1}}-W_{t_i}+\frac{1}{\gamma}L(y_{i+1})(t_{i+1}-t_i)=\frac{1}{\gamma}(y_{i+1}-y_i)
$ is equal to the law of $\sup_{t\in[0,t_{1}]}W_t$ given $W_{t_{1}}=\frac{1}{\gamma}(y_{i+1}-y_i)$ which is given by $
\mathbb{P}\Big(\displaystyle\sup_{t\in[0, t_1]} W_{t_1}\geq y\Big| W_{t_1}=x\Big)= 
\e^{\frac{-2(y)_+(y-x)_+}{h}} 
$ (see e.g. \cite[p.\ 265]{karatzas1991brownian}).
Thus, we get
\begin{align*}
\varphi(y_i,y_{i+1})&=\mathbb{P}\Big(\sup _{t \in[0,t_1]} W_t \ge \frac{1}{\gamma}({\mathcal U}-y_i)|W_{t_1}=\frac{1}{\gamma}(y_{i+1}-y_i)\Big)\\
&=\exp\Big(\frac{-2({\mathcal U}- y_i)_+({\mathcal U}- y_{i+1})_+}{\gamma^2 h}\Big).
\end{align*}
The same arguments applied to $(-W_t)_{t\ge0}$ work for the Down-Out barrier option. \end{proof}
 \subsection{The interpolated drift implicit Euler scheme MLMC method analysis.}
%Giles et al. \cite{GilDebRos} proved that the MLMC efficiency for pricing barrier options using a Brownian bridge technique for diffusion models with globally Lipschitz coefficients, can be  improved by a conditional Monte Carlo method. 
% In this current work, thanks to the Lamperti transform technique and using our Brownian interpolation of the drift implicit scheme (see \eqref{interp}), we show that  the MLMC efficiency can be also improved for pricing barrier options using the Brownian bridge technique combined with the conditional Monte Carlo method for diffusion models with  non-Lipschitz coefficients. To be more precise,
 We consider the drift implicit scheme $(\overline{Y}_{t_i}^{2^{\ell}})_{0\le i \le 2^{\ell}}$
given in \eqref{scheme} that approximates  $(Y_t)_{0\le t\le T}$ solution to \eqref{lam} 
 using a time step $h_{\ell}=2^{-\ell}T$ for $\ell\in\{0,...,L\}$, with $L={\log n}/{\log 2}$, where $n$ denotes the finest time step number. Let  $(\overline{Y}_{t}^{2^{\ell}})_{0\le t \le T}$ denote  the  Brownian interpolation of  the  drift implicit scheme defined in \eqref{interp} with time step $h_{\ell}$. As the same arguments work for both   Down-Out and Up-Out  barrier options, we give details only for the latter one.  To do so, we introduce
 \begin{align}
 \label{sm}
     \overline{P}_{\ell}:=g(\overline{Y}_T^{2^{\ell}})\prod_{i=0}^{2^{\ell}-1}\mathbf 1_{\{ \sup_{t\in[t^{\ell}_i, t^{\ell}_{i+1}]}\overline{Y}_{t}^{2^{\ell}}<{\mathcal U}\}},
    \quad  \mbox{ where }  t^{\ell}_i=\frac{iT}{2^{\ell}} \quad \mbox{ for } \ell\in\{0,...,L\},
\end{align}
and write
 \begin{align}
 \label{mlmc}
   \overline{\pi}_{\mathcal U} =\mathbb{E}\big[\overline{P}_L\big]= \mathbb{E}\big[\overline{P}_0\big]+ \sum_{\ell=1}^L\mathbb{E}\big[\overline{P}_{\ell}- \overline{P}_{\ell-1}\big],
 \end{align}
 where $ \overline{\pi}_{\mathcal U}$ is introduced in subsection \ref{subsec:BB}.
On the one hand, applying Proposition \ref{lem:BB} yields
\begin{align}\label{pf}
\mathbb E\big[\overline{P}_{\ell}\big]=\mathbb E\big[\overline{P}_{\ell}^{f}\big], \mbox{ where }\overline{P}_{\ell}^{f}:&=g(\overline{Y}_T^{2^{\ell}})\prod_{i=0}^{2^{\ell}-1}(1-\overline{p}_{i}^{2^{\ell}}) \mbox{ with }\\
\overline{p}_{i}^{2^{\ell}}&=\exp\Big(\frac{-2({\mathcal U}-\overline{Y}_{t^{\ell}_i}^{2^{\ell}})_+({\mathcal U}-\overline{Y}_{t^{\ell}_{i+1}}^{2^{\ell}})_+}{\gamma^2 h_{\ell}}\Big)\notag.
\end{align}
On the other hand,  we write
\begin{align*}
 &\mathbb E\big[\overline{P}_{\ell-1}\big]%&= \mathbb E\big[g(\overline{Y}_T^{2^{\ell-1}})\prod_{i=0}^{2^{\ell-1}-1}\mathbb E\big[\mathbf 1_{\{ \sup_{t\in[t^{\ell-1}_i, t^{\ell-1}_{i+1}]}\overline{Y}_{t}^{2^{\ell-1}}<U\}}|\overline{Y}_{t^{\ell-1}_i}^{2^{\ell-1}},\overline{Y}_{t^{\ell-1}_{i+1}}^{2^{\ell-1}} \big]\big]\\
 =\mathbb E\big[g(\overline{Y}_T^{2^{\ell-1}})\prod_{i=0}^{2^{\ell-1}-1}\mathbb E\big[\mathbf 1_{\{ \sup_{t\in[t^{\ell-1}_i, t^{\ell-1}_{i+1}]}\overline{Y}_{t}^{2^{\ell-1}}<{\mathcal U}\}}|\overline{Y}_{t^{\ell-1}_i}^{2^{\ell-1}},\overline{Y}_{t^{\ell}_{2i+1}}^{2^{\ell-1}} ,\overline{Y}_{t^{\ell-1}_{i+1}}^{2^{\ell-1}} \big]\big]\\
 &=\mathbb E\big[g(\overline{Y}_T^{2^{\ell-1}})\prod_{i=0}^{2^{\ell-1}-1}\mathbb E\big[\mathbf 1_{\{ \sup_{t\in[t^{\ell-1}_i, t^{\ell}_{2i+1}]}\overline{Y}_{t}^{2^{\ell-1}}<{\mathcal U}\}}\mathbf 1_{\{ \sup_{t\in[t^{\ell}_{2i+1}, t^{\ell-1}_{i+1}]}\overline{Y}_{t}^{2^{\ell-1}}<{\mathcal U}\}}| \overline{Y}_{t^{\ell-1}_i}^{2^{\ell-1}},\overline{Y}_{t^{\ell}_{2i+1}}^{2^{\ell-1}} ,\overline{Y}_{t^{\ell-1}_{i+1}}^{2^{\ell-1}}\big]\big],
\end{align*}
where the coarse scheme $\overline{Y}_{t^{\ell}_{2i+1}}^{2^{\ell-1}} $ 
is computed using our Brownian interpolation scheme \eqref{interp}  that is
$$
 \overline{Y}_{t^{\ell}_{2i+1}}^{2^{\ell-1}}=\overline{Y}_{t^{\ell-1}_{i}}^{2^{\ell-1}}+L(\overline{Y}_{t^{\ell-1}_{i+1}}^{2^{\ell-1}})({t^{\ell}_{2i+1}}-t^{\ell-1}_{i})+\gamma(W_{t^{\ell}_{2i+1}}-W_{t^{\ell-1}_{i}}).
$$ 
Thus,  we rewrite $ \sup_{t\in[t^{\ell-1}_i, t^{\ell}_{2i+1}]}\overline{Y}_{t}^{2^{\ell-1}}$ and $\sup_{t\in[ t^{\ell}_{2i+1},t^{\ell-1}_{i+1}]}\overline{Y}_{t}^{2^{\ell-1}}$ as follows
\begin{align*}
 \sup_{t\in[t^{\ell-1}_i, t^{\ell}_{2i+1}]}\overline{Y}_{t}^{2^{\ell-1}}&=\overline{Y}_{t^{\ell-1}_{i}}^{2^{\ell-1}}+\gamma \sup_{t\in[t^{\ell-1}_i, t^{\ell}_{2i+1}]}\left(W_t- W_{t^{\ell-1}_{i}}+
 \frac1 \gamma L(\overline{Y}_{t^{\ell-1}_{i+1}}^{2^{\ell-1}})(t-t^{\ell-1}_{i})\right),\mbox{ with }\\
 W_{t^{\ell}_{2i+1}}- W_{t^{\ell-1}_{i}}&+
 \frac1 \gamma L(\overline{Y}_{t^{\ell}_{i+1}}^{2^{\ell-1}})({t^{\ell}_{2i+1}}-t^{\ell-1}_{i})=
 \frac 1\gamma \left( \overline{Y}_{t^{\ell}_{2i+1}}^{2^{\ell-1}}- \overline{Y}_{t^{\ell-1}_{i}}^{2^{\ell-1}}\right)
 \end{align*}
 and
 \begin{align*}
 \sup_{t\in[ t^{\ell}_{2i+1},t^{\ell-1}_{i+1}]}\overline{Y}_{t}^{2^{\ell-1}}&=\overline{Y}_{t^{\ell}_{2i+1}}^{2^{\ell-1}}
 +\gamma \sup_{t\in[t^{\ell-1}_i, t^{\ell}_{2i+1}]}\left(W_t- W_{t^{\ell}_{2i+1}}+
 \frac1 \gamma L(\overline{Y}_{t^{\ell-1}_{i+1}}^{2^{\ell-1}})(t-t^{\ell}_{2i+1})\right), \mbox{ with }\\
 W_{t^{\ell-1}_{i+1}}- W_{t^{\ell}_{2i+1}}&+
 \frac1 \gamma L(\overline{Y}_{t^{\ell-1}_{i+1}}^{2^{\ell-1}})(t^{\ell-1}_{i+1}-t^{\ell}_{2i+1})=
 \frac 1\gamma \left( \overline{Y}_{t^{\ell-1}_{i+1}}^{2^{\ell-1}}-\overline{Y}_{t^{\ell}_{2i+1}}^{2^{\ell-1}}\right).
\end{align*}
Then, using the independence of the Brownian increments and  the same arguments as in the proof of Proposition \ref{lem:BB}, we get
\begin{align}\label{pc}
\mathbb E\big[\overline{P}_{\ell-1}\big]=\mathbb E\big[\overline{P}_{\ell-1}^{c}\big], \mbox{ where }\overline{P}_{\ell-1}^{c}:&=g(\overline{Y}_T^{2^{\ell-1}})\prod_{i=0}^{2^{\ell-1}-1}(1-\overline{p}_{i,1}^{2^{\ell-1}})(1-\overline{p}_{i,2}^{2^{\ell-1}}) \mbox{ with }\\
\overline{p}_{i,1}^{2^{\ell-1}}&=\exp\Big(\frac{-2({\mathcal U}-\overline{Y}_{t^{\ell-1}_i}^{2^{\ell-1}})_+({\mathcal U}-\overline{Y}_{t^{\ell}_{2i+1}}^{2^{\ell-1}})_+}{\gamma^2 h_{\ell}}\Big)\notag,\\
\overline{p}_{i,2}^{2^{\ell-1}}&=\exp\Big(\frac{-2({\mathcal U}-\overline{Y}_{t^{\ell}_{2i+1}}^{2^{\ell-1}})_+({\mathcal U}-\overline{Y}_{t^{\ell-1}_{i+1}}^{2^{\ell-1}})_+}{\gamma^2 h_{\ell}}\Big),\notag
\end{align}
which can be rewritten as 
\begin{equation}\label{pcnum}
\overline{P}_{\ell-1}^{c}:=g(\overline{Y}_T^{2^{\ell-1}})\prod_{i=0}^{2^{\ell}-1}(1-\overline{p}_{i}^{2^{\ell-1}})\mbox{ with } \overline{p}_{i}^{2^{\ell-1}}=\exp\Big(\frac{-2({\mathcal U}-\overline{Y}_{t^{\ell}_i}^{2^{\ell-1}})_+({\mathcal U}-\overline{Y}_{t^{\ell}_{i+1}}^{2^{\ell-1}})_+}{\gamma^2 h_{\ell}}\Big),
\end{equation}
where the coarse approximation scheme $\overline{Y}_{t^{\ell}_i}^{2^{\ell-1}}$ for odd index $i\in\{0,...,2^{\ell}-1\}$ is computed using the Brownian interpolation scheme \eqref{interp}.
Thus, the improved MLMC method approximates $\bar \pi_U$ 
by
\begin{equation}\label{MLMC:U}
\bar P_U:=\frac{1}{N_0}\sum_{k=1}^{N_0} \overline{P}^f_{0,k} + \sum_{\ell =1}^L \frac{1}{N_{\ell}} \sum_{k=1}^{N_{\ell}} \Big(\overline{P}_{\ell,k}^{f}-\overline{P}_{\ell-1,k}^{c}\Big),
\end{equation}
 where the condition  $\mathbb E\big[\overline{P}_{\ell-1}^{f}\big]=\mathbb E\big[\overline{P}_{\ell-1}^{c}\big]$ is satisfied according to \eqref{pf} and \eqref{pc}.
Here,  $(\overline{P}_{\ell,k}^{f})_{1\le k\le N_{\ell}}$ and 
$(\overline{P}_{\ell-1,k}^{c})_{1\le k\le N_{\ell}}$ are respectively independent copies of  $\overline{P}_{\ell}^{f}$ and $\overline{P}_{\ell-1}^{c}$ given by \eqref{pf} and \eqref{pcnum}.  Similarly,  the improved MLMC method approximates $\bar \pi_{\mathcal D}$ 
by
\begin{equation}\label{MLMC:D}
\bar Q_{\mathcal D}:=\frac{1}{N_0}\sum_{k=1}^{N_0} \overline{Q}^f_{0,k} + \sum_{\ell =1}^L \frac{1}{N_{\ell}} \sum_{k=1}^{N_{\ell}} \Big(\overline{Q}_{\ell,k}^{f}-\overline{Q}_{\ell-1,k}^{c}\Big),
\end{equation}
 where  $(\overline{Q}_{\ell,k}^{f})_{1\le k\le N_{\ell}}$ and 
$(\overline{Q}_{\ell-1,k}^{c})_{1\le k\le N_{\ell}}$ are respectively independent copies of $\overline{Q}_{\ell}^{f}$ and $\overline{Q}_{\ell-1}^{c}$ given by 
\begin{align*}
\overline{Q}_{\ell}^{f}&:=g(\overline{Y}_T^{2^{\ell}})\prod_{i=0}^{2^{\ell}-1}(1-\overline{q}_{i}^{2^{\ell}}) \mbox{ with }
\overline{q}_{i}^{2^{\ell}}=\exp\Big(\frac{-2(\overline{Y}_{t^{\ell}_i}^{2^{\ell}}-{\mathcal D})_+(\overline{Y}_{t^{\ell}_{i+1}}^{2^{\ell}}-{\mathcal D})_+}{\gamma^2 h_{\ell}}\Big)\\
\overline{Q}_{\ell-1}^{c}&:=g(\overline{Y}_T^{2^{\ell-1}})\prod_{i=0}^{2^{\ell}-1}(1-\overline{q}_{i}^{2^{\ell-1}})\mbox{ with } \overline{q}_{i}^{2^{\ell-1}}=\exp\Big(\frac{-2(\overline{Y}_{t^{\ell}_i}^{2^{\ell-1}}-{\mathcal D})_+(\overline{Y}_{t^{\ell}_{i+1}}^{2^{\ell-1}}-{\mathcal D})_+}{\gamma^2 h_{\ell}}\Big).
\end{align*}

%% For the rest of the paper we note
% %$\overline{Y}_{t}^{m^l}$ by $\overline{Y}_{t}^{mn}$
%%and  $\overline{Y}_{t}^{m^{l-1}}$ by $\overline{Y}_{t}^n$
%\textcolor{blue}{ \begin{lem}
%If Y is a scalar random variable, $\mathbb{E}[Y^2]$ is uniformly bounded, and for each $p>0$, the indicator function $\mathbbm{1}_{E}$ (which takes value $1$ or $0$ depending
%whether or not a path lies within some set $E$) satisfies 
%$$\mathbb{E}[\mathbbm{1}_{E}]=o(h^p),$$
%then for each $p>0,$
%%$1\leq p < \dfrac{4a}{3\sigma^2}$,
%$$\mathbb{E}[|Y|\mathbbm{1}_{E}]=o(h^p).$$
%\end{lem}
%\begin{lem}
%\label{eq21}
%Provided $\eqref{ine1}$ is satisfied the expected value of  $\overline{P}_{\ell}^f$ and $\overline{P}_{\ell}^c$ 
%$$\mathbb{E}\Big[|\overline{P}_{\ell}^f|^p\Big]\leq C_p$$
%$$\mathbb{E}\Big[|\overline{P}_{\ell}^c|^p\Big]\leq C'_p $$
%for all $p>0$.
%\end{lem}
%\begin{proof}
%According to \eqref{pf} and \eqref{pcnum} we see that $\prod_{i=0}^{2^{\ell}-1}(1-\overline{p}_{i}^{2^{\ell}}) \leq 1$ and $\prod_{i=0}^{2^{\ell}-1}(1-\overline{p}_{i}^{2^{\ell-1}}) \leq 1$ because $\overline{p}_{i}^{2^{\ell}}$, $\overline{p}_{i}^{2^{\ell-1}}$ in $[0,1]$. So 
%\begin{align*}
%    \mathbb{E}\Big[|\overline{P}_{\ell}^f|^p\Big]&= \mathbb{E}\Big[\Big|g(\overline{Y}_T^{2^{\ell}})\prod_{i=0}^{2^{\ell}-1}(1-\overline{p}_{i}^{2^{\ell}})\Big|^p \Big]\\
%    & \leq \mathbb{E}\Big[\Big|g(\overline{Y}_T^{2^{\ell}})\Big|^p \Big]\\
%\end{align*}
%By using corollary \eqref{cor0} and the Lipschitz property we obtain:
%\begin{align*}
%\mathbb{E}\Big[|\overline{P}_{\ell}^f|^p\Big]&\leq C_p 
%\end{align*}
%\end{proof}
%}

In what follows, we  need to strengthen our assumption \eqref{ine3}  as below.
 \begin{align}\label{ine3tilde}
 \mbox{For $p\ge 1$, assumption \eqref{ine3} is valid and }  \sup_{t\in[0,T]}\mathbb{E}\Big[|L(Y_t)|^p\Big]< \infty\tag{$\rm{\tilde H}2$}.
\end{align}
\begin{lem}
\label{eq7} Assume that conditions \eqref{ine3tilde}, \eqref{ine4} and \eqref{ine5} are satisfied for a given $p>1$ and $0<L'_A<\frac 1{2h_{\ell}}$, with $h_{\ell}=2^{-\ell}T$  sufficiently small. Let  $\eta\in(0,1)$, the following extreme path events satisfy
\begin{align}\label{ext:1}
&\mathbb P\left (\max\Big(\sup_{0\leq i \leq 2^\ell}(\vert Y_{t^{\ell}_i}\vert, \vert \overline{Y}_{t^{\ell}_i}^{2^{\ell}}\vert,  \vert  \overline{Y}_{t^{\ell}_i}^{2^{\ell-1}} \vert)\Big) > h_{\ell}^{-\eta}\right)=o(h_{\ell}^q)\\\label{ext:2}
&\mathbb P\left (\max\Big(\sup_{0\leq i \leq 2^\ell} \big(\vert Y_{t^{\ell}_i}- \overline{Y}_{t^{\ell}_i}^{2^{\ell}} \vert, \vert Y_{t^{\ell}_i}-\overline{Y}_{t^{\ell}_i}^{2^{\ell-1}} \vert, \vert \overline{Y}_{t^{\ell}_i}^{2^{\ell}} -\overline{Y}_{t^{\ell}_i}^{2^{\ell-1}} \vert \big)\Big) > h_{\ell}^{1-\eta}\right)=o(h_{\ell}^q)\\
%\sup_{0\leq i \leq n} \vert \Delta W_n \vert > h^{\frac{1}{2}-\eta}\\
&\sup_{0\leq i \leq 2^\ell}  \mathbb P\left (\int_{t^{\ell}_{i}}^{t^{\ell}_{i+1}}|L(Y_s) |ds >h_{\ell}^{1-\eta} \right)=o(h_{\ell}^q) \label{ext:4}\\
& 
\mathbb P\left (\sup_{t\in[0,T]}\vert \overline{Y}_t^{2^{\ell}}-Y_t \vert >h_{\ell}^{1-\eta}\right)=o(h_{\ell}^q) \label{ext:3}\\
& \mbox{for all } 0<q< p\eta, \mbox{and }\notag\\
&\sup_{0\leq i \leq 2^\ell}  \mathbb P\left (\sup_{t\in[t^{\ell}_{i},t^{\ell}_{i+1}]}\vert W_{t}- W_{t^{\ell}_{i}}\vert > h_{\ell}^{\frac 12-\eta}\right)=o(h_{\ell}^q), \; \mbox{ for all } {q>0}.\label{ext:5}
\end{align}
%\textcolor{blue}{Moreover, if none of these extreme path events is satisfied, and $0<\eta < \frac{1}{2}$ then
%\begin{align*}
%\sup_{0\leq i \leq 2^\ell} \vert\overline{Y}_{t^{\ell}_i}^{2^{\ell}} - \overline{Y}_{t^{\ell}_{i-1}}^{2^{\ell}}  \vert =O(h_{\ell}^{\frac{1}{2}-2\eta}).
%\end{align*}
%}
\end{lem}  
 \begin{proof}
For the first extreme path property, we have
\begin{multline*}
\mathbb{P}\big(\max(\sup_{0\leq i \leq 2^\ell}(\vert Y_{t^{\ell}_i}\vert, \vert \overline{Y}_{t^{\ell}_i}^{2^{\ell}}\vert,  \vert  \overline{Y}_{t^{\ell}_i}^{2^{\ell-1}} \vert))> h_{\ell}^{-\eta}\big)\\
\leq  \mathbb{P}\big(\sup_{0\leq i \leq 2^\ell}|Y_{t^{\ell}_i}|>h_{\ell}^{-\eta} \big)+ \mathbb{P}\big(\sup_{0\leq i \leq 2^\ell}\vert \overline{Y}_{t^{\ell}_i}^{2^{\ell}}\vert > h_{\ell}^{-\eta} \big)+ \mathbb{P}\big(\sup_{0\leq i \leq 2^\ell}\vert \overline{Y}_{t^{\ell}_i}^{2^{\ell-1}}| > h_{\ell}^{-\eta}\big).
\end{multline*}
%\begin{align*}
%\mathbb{P}\big(\sup_{0\leq i \leq 2^\ell}(\sup(\vert Y_T\vert, \vert \overline{Y}_{t^{\ell}_i}^{2^{\ell}}\vert,  \vert  \overline{Y}_{t^{\ell}_i}^{2^{\ell-1}} \vert))> h^{-\eta}\big)&=
%\mathbb{P}\big(\sup_n (a_n)>h^{-\eta}\big)\\
%&\leq \sum_n \mathbb{P}\big(a_n>h^{-\eta}\big)\\
%& \leq \sum_n \mathbb{P}\big(\sup(\vert Y_T\vert, \vert \overline{Y}_{t^{\ell}_i}^{2^{\ell}}\vert,  \vert  \overline{Y}_{t^{\ell}_i}^{2^{\ell-1}} \vert)> h^{-\eta}\big)\\
%& \leq \sum_n \mathbb{P}\big(|Y_T|>h^{-\eta}\cup \vert \overline{Y}_{t^{\ell}_i}^{2^{\ell}}\vert> h^{\eta}\cup \vert\overline{Y}_{t^{\ell}_i}^{2^{\ell-1}} \vert> h^{-\eta}\big)\\
%& \leq \sum_n \Big[\mathbb{P}\big(|Y_T|>h^{-\eta} \big)+ \mathbb{P}\big(\vert \overline{Y}_{t^{\ell}_i}^{2^{\ell}}\vert > h^{-\eta} \big)+ \mathbb{P}\big(\vert \overline{Y}_{t^{\ell}_i}^{2^{\ell-1}}| > h^{-\eta}\big) \Big]\\
%\end{align*} 
Then, by Markov's inequality we get  for $m\ge 1$
\begin{multline*}
\mathbb{P}\big(\max(\sup_{0\leq i \leq 2^\ell}(\vert Y_{t^{\ell}_i}\vert, \vert \overline{Y}_{t^{\ell}_i}^{2^{\ell}}\vert,  \vert  \overline{Y}_{t^{\ell}_i}^{2^{\ell-1}} \vert))> h_{\ell}^{-\eta}\big)\leq\\
{h_{\ell}^{m\eta}}\Big(\mathbb{E}[\sup_{0\le t\le T}|Y_{t}|^{m}]+\mathbb{E}[\sup_{0\le t\le T}|\overline{Y}_{t}^{2^{\ell}}\vert^{m} ]+ \mathbb{E}[\sup_{0\le t\le T}\vert \overline{Y}_{t}^{2^{\ell-1}}|^{m}]\Big).
\end{multline*}
The result follows by  Corollary \ref{cor0} for $h_{\ell}$ sufficiently small with choosing $m$ such that  $0<\frac{q}{\eta}<m\le p$. For the second extreme path property, we proceed in the same way to get  for $m\ge1$ 
\begin{multline*}
\mathbb P\left (\max\Big(\sup_{0\leq i \leq 2^\ell} \big(\vert Y_{t^{\ell}_i}- \overline{Y}_{t^{\ell}_i}^{2^{\ell}} \vert, \vert Y_{t^{\ell}_i}-\overline{Y}_{t^{\ell}_i}^{2^{\ell-1}} \vert, \vert \overline{Y}_{t^{\ell}_i}^{2^{\ell}} -\overline{Y}_{t^{\ell}_i}^{2^{\ell-1}} \vert \big)\Big) > h_{\ell}^{1-\eta}\right)\\
\le {h_{\ell}^{-m(1-\eta) }}\Big(\mathbb{E}[\sup_{0\le t\le T}|Y_{t}- \overline{Y}_{t}^{2^{\ell}} |^{m}]+\mathbb{E}[\sup_{0\le t\le T}|Y_{t}- \overline{Y}_{t}^{2^{\ell-1}} |^{m}]+ \mathbb{E}[\sup_{0\le t\le T}\vert \overline{Y}_{t}^{2^{\ell}}-\overline{Y}_{t}^{2^{\ell-1}}|^{m}]\Big).
\end{multline*} 
Thus, we deduce the result using Theorem \ref{hypc} with choosing $m$ such that $0<\frac{q}{\eta}<m\le p$, for $h_{\ell}$ sufficiently small. 
For the third extreme path, we proceed in the same way to get that for all ${0\leq i \leq 2^\ell}$, using Jensen's inequality
\begin{align*}
\mathbb P\left (\int_{t^{\ell}_{i}}^{t^{\ell}_{i+1}}|L(Y_s) |ds >h_{\ell}^{1-\eta} \right)&\le h_{\ell}^{m(\eta-1)}  \mathbb E\left( \left|\int_{t^{\ell}_{i}}^{t^{\ell}_{i+1}}|L(Y_s) |ds \right|^m\right)\\
&\le h_{\ell}^{m\eta-1}  \mathbb E\left( \int_{t^{\ell}_{i}}^{t^{\ell}_{i+1}}|L(Y_s) |^mds \right)\\
&\le h_{\ell}^{m\eta}  \sup_{t\in[0,T]}\mathbb E\left( |L(Y_s) |^m \right).
\end{align*}
Then we conclude using \eqref{ine3tilde} by choosing $m$ such that $0<\frac{q}{\eta}<m\le p$, 
for $h_{\ell}$ sufficiently small.
The fourth  extreme path property  follows in the same way as the second one.  Finally, the last property follows using that
$
\mathbb E\left[\sup_{t\in[0,T]}|W_t|^m \right] 
$
is finite for any positive power $m$.
%For none of the extreme conditions we have 
%\begin{align*}
%    \vert \overline{Y}_{t_i}^n-\overline{Y}_{t_{i-1}}^{n} \vert \leq \vert h L(\overline{Y}_{t_i}^n)+\gamma h^{\frac{1}{2}-\eta}\vert
%\end{align*}
%$$\vert \overline{Y}_{t_i}^n-\overline{Y}_{t_{i-1}}^{n} \vert \leq h\big(c_1 h^{\frac{\alpha \eta}{1-\alpha}}+c_2 h^{-\eta}+c_3 h^{\eta}\big)+\gamma h^{\frac{1}{2}-\eta}$$
%We notice that $h^{\frac{1}{2}-\eta}$ is the dominant term for $\eta<\frac{1}{2}$, so $\sup_{0\leq i \leq n} \vert \overline{Y}_{t_i}^n-\overline{Y}_{t_{i-1}}^{n} \vert \prec h^{\frac{1}{2}-2\eta}.$
\end{proof}
Now, we are able to state our main theorem for the MLMC method to price barrier options when the underlying asset has possibly non-Lipschitz coefficients.
\begin{thm}\label{MLMC:var} Let $g$ denote a payoff function satisfying : $\exists C>0 \mbox{ s.t. }\forall x,y>0$, 
\begin{align}\label{cond:payoff}
  |g(x)-g(y)|\le C|x-y|(1+|x|^{\nu}+ |y|^{\nu}) \mbox{ and } |g(x)|\le C(1+|x|^{\nu+1}), \mbox{ with }\nu\in \mathbb R_+.
\end{align}
Moreover, assume that conditions  \eqref{ine3tilde}, \eqref{ine4} and \eqref{ine5} are satisfied for   $p>\frac{(1+\delta)(1+\gamma)[7(1+\varepsilon)+2\nu]}{\frac 12 -\delta }$, with $\varepsilon,\gamma>0$, $\delta \in(0,1/2)$ and $0<L'_A<\frac 1{2h_{\ell}}$, with $h_{\ell}=2^{-\ell}T$  sufficiently small. If in addition 
 $\inf_{t \in [0, T]}Y_{t}$  (resp.  $\sup_{t \in [0, T]}Y_{t}$) has a bounded density in the neighborhood of the barrier ${\mathcal D}$ (resp. ${\mathcal U}$), then the multilevel estimator $\bar Q_{\mathcal D}$ given by \eqref{MLMC:D} (resp. $\bar P_{\mathcal U}$ given by \eqref{MLMC:U}  ) for the D-O (resp.~U-O) barrier option satisfies  
 $\rm{Var}(\overline{Q}_{\ell}^f-\overline{Q}_{\ell}^c)= O(h_{\ell}^{1 + \delta})$  
 (resp.  $\rm{Var}(\overline{P}_{\ell}^f-\overline{P}_{\ell}^c)= O(h_{\ell}^{1 + \delta})$).
\end{thm}
\begin{rem}\label{rem:param}
$\bullet\;$ Combining the complexity theorem in \cite[Theorem 3.1]{giles2008} with the above result, we deduce that for any $\delta\in(0,\frac12)$ the MLMC estimators $\bar Q_{\mathcal D}$ and  $\bar P_{\mathcal U}$  reach the optimal time complexity $O(\varepsilon^{-2})$, for a given precision $\varepsilon >0$, and behave like an unbiased Monte Carlo estimator. \\
$\bullet\;$ Taking $\delta$ close to $\frac12$ achieves a smaller variance of the difference between the finer and coarse approximations which is of order  $O(h_{\ell}^{\beta})$ with $\beta$ close to $\frac32$ similar to the case  of diffusion with Lipschitz coefficients studied in \cite[Theorem 3.15]{GilDebRos}, but clearly leads to very restrictive conditions on the finiteness of the moments  of $(Y_t)_{t\in[0,T]}$ and $(\bar Y^n_t)_{t\in[0,T]}$. 
\end{rem}
\begin{proof} We only give a proof for the D-O barrier option since the proof for the U-O barrier option is quite similar. At first, following the extreme path approach given in \cite{GilHigMao,GilDebRos}, we write 
\begin{align*}
    \rm{Var}[\overline{Q}_{\ell}^f-\overline{Q}_{\ell}^c] &\leq \mathbb{E}[(\overline{Q}_{\ell}^f-\overline{Q}_{\ell}^c)^2]\\
    &= \mathbb{E}[(\overline{Q}_{\ell}^f-\overline{Q}_{\ell}^c)^2\mathbbm{1}_{A_1}]+\mathbb{E}[(\overline{Q}_{\ell}^f-\overline{Q}_{\ell}^c)^2\mathbbm{1}_{A_2}]+\mathbb{E}[(\overline{Q}_{\ell}^f-\overline{Q}_{\ell}^c)^2\mathbbm{1}_{A_3}]
\end{align*}
where we split the paths into the following three events.\\

\paragraph*{\bf First event $A_1$}We consider any of the extreme path events given in Lemma \ref{eq7} that satisfy \eqref{ext:1}-\eqref{ext:5}, with some {$ \eta > 0$} to be fixed later on. For $\gamma>0$, we use Hölder's inequality to get
\begin{align*}
 \mathbb{E}[(\overline{Q}_{\ell}^f-\overline{Q}_{\ell}^c)^2\mathbbm{1}_{A_1}]&\le  {\mathbb{E}}^{\frac{\gamma}{1+\gamma}}\Big[|\overline{Q}_{\ell}^f-\overline{Q}_{\ell}^c|^{\frac{2(1+\gamma)}{\gamma}}\Big]  {\mathbb{E}}^{\frac{1}{1+\gamma}}[\mathbbm{1}_{A_1} ]\\&\le 2^{\frac{2+\gamma}{1+\gamma}}\left({\mathbb{E}}^{\frac{\gamma}{1+\gamma}}[|\overline{Q}_{\ell}^f|^{\frac{2(1+\gamma)}{\gamma}}]+{\mathbb{E}}^{\frac{\gamma}{1+\gamma}}[|\overline{Q}_{\ell}^c|^{\frac{2(1+\gamma)}{\gamma}}]\right)\Big({\mathbb{P}}[{A_1} ]\Big)^{\frac{1}{1+\gamma}}.
\end{align*}
As the payoff function $g$ satisfies assumption \eqref{cond:payoff}, we deduce using  Corollary \ref{cor0}, that for  $h_{\ell}$  sufficiently small $\mathbb{E}[(\overline{Q}_{\ell}^f)^{\frac{2(1+\gamma)}{\gamma}}]$ and $\mathbb{E}[(\overline{Q}_{\ell}^c)^{\frac{2(1+\gamma)}{\gamma}}] $ are finite. By Lemma \ref{eq7}, we get that  
\begin{equation}\label{est:1}
\mathbb{E}[(\overline{Q}_{\ell}^f-\overline{Q}_{\ell}^c)^2\mathbbm{1}_{A_1}]=o(h_{\ell}^\frac{q}{1+\gamma})\mbox{ for all $q$ such that } 0<\frac q{\eta}\le p.
\end{equation}
\paragraph*{\bf Second event $A_2$}This event corresponds to the non-extreme paths satisfying $$|\inf_{t\in[0,T]}Y_t - {\mathcal D} | > {h_{\ell}}^{\frac{1}{2}-\eta(1+\varepsilon)}\mbox{ for }\eta \in(0,1/{2(1+\varepsilon)}) \mbox{ with } \varepsilon >0.$$
Let us assume that  $\inf_{t\in[0,T]}Y_t =Y_{\tau}$ with $\tau\in[t^{\ell}_i,t^{\ell}_{i+1}]$ for a given $i\in \{ 0,\dots, 2^{\ell}\}$. \\
$\bullet\,$ First case:  for $Y_{\tau} <{\mathcal D}- {h_{\ell}}^{\frac{1}{2}-\eta(1+\varepsilon)}$, 
we write  
\begin{align*}
\vert \overline{Y}_{t^{\ell}_i}^{2^{\ell}}-Y_{\tau}\vert &\leq \vert\overline{Y}_{t^{\ell}_i}^{2^{\ell}}-{Y}_{t^{\ell}_i} \vert +\vert {Y}_{t^{\ell}_i}-Y_{\tau}\vert\\
&\leq \vert\overline{Y}_{t^{\ell}_i}^{2^{\ell}}-{Y}_{t^{\ell}_i} \vert  + \int_{t^{\ell}_i}^{t^{\ell}_{i+1}}|L(Y_s)|ds + \gamma \sup_{t\in[t^{\ell}_{i},t^{\ell}_{i+1}]} |W_{t} -W_{t^{\ell}_i}|.
\end{align*}
Then, as we work on the non-extreme paths events  (see Lemma \ref{eq7} for the extreme paths events), for ${h_{\ell}}$ sufficiently small  we  have that  $\Big|\overline{Y}_{t^{\ell}_i}^{2^{\ell}}-Y_{\tau} \Big| =O( {h_{\ell}}^{\frac{1}{2}-\eta})$ and   then $\Big|\overline{Y}_{t^{\ell}_i}^{2^{\ell}}-Y_{\tau} \Big| < {h_{\ell}}^{\frac{1}{2}-\eta(1+\varepsilon)}$, which yields that $\overline{Y}_{t^{\ell}_i}^{2^{\ell}}<{\mathcal D}$.
Once again, as we are in the case of non  extreme paths we have $|\overline{Y}_{t^{\ell}_i}^{2^{\ell}}-\overline{Y}_{t^{\ell}_i}^{2^{\ell-1}}|< {h_{\ell}}^{1-\eta}$  and so $\Big|\overline{Y}_{t^{\ell}_i}^{2^{\ell-1}}-Y_{\tau} \Big| =O( {h_{\ell}}^{\frac{1}{2}-\eta})$. Hence, we also have $\overline{Y}_{t^{\ell}_i}^{2^{\ell-1}}<{\mathcal D}$, for sufficiently small $h_{\ell}$ which yields that $\overline{Q}_{\ell}^f-\overline{Q}_{\ell}^c=0$.\\
$\bullet\,$ Second case:  for  $Y_{\tau}> {\mathcal D}+{h_{\ell}}^{\frac{1}{2}-\eta(1+\varepsilon)}$, we proceed in the same way and we easily check that  for  ${h_{\ell}}$ sufficiently small $\prod_{i=0}^{2^{\ell}-1}(1-\overline{q}_{i}^{2^{\ell}})$ and $\prod_{i=0}^{2^{\ell}-1}(1-\overline{q}_{i}^{2^{\ell-1}})$ are both equal to $1+o({h_{\ell}}^a)$ for all  $a>0$. Consequently, as the payoff function $g$ satisfies condition \eqref{cond:payoff} and as we work with the non-extreme paths events,  we deduce that $\mathbb{E}[(\overline{Q}_{\ell}^f-\overline{Q}_{\ell}^c)^2\mathbbm{1}_{A_2}]=O({h_{\ell}}^{2(1-\eta) -2\eta\nu})$.     
 \paragraph*{\bf Third event $A_3$} This last event corresponds to the rest of the non extreme paths. At first, let us note that 
 \begin{align*}
\vert \overline{Y}_{t^{\ell}_{i+1}}^{2^{\ell}}- \overline{Y}_{t^{\ell}_{i}}^{2^{\ell}}\vert &\leq \vert\overline{Y}_{t^{\ell}_{i+1}}^{2^{\ell}}-{Y}_{t^{\ell}_{i+1}} \vert +\vert {Y}_{t^{\ell}_{i+1}}-Y_{t^{\ell}_{i}}\vert+ \vert\overline{Y}_{t^{\ell}_{i}}^{2^{\ell}}-{Y}_{t^{\ell}_{i}} \vert \\
&\leq \vert\overline{Y}_{t^{\ell}_i}^{2^{\ell}}-{Y}_{t^{\ell}_i} \vert + \vert\overline{Y}_{t^{\ell}_{i}}^{2^{\ell}}-{Y}_{t^{\ell}_{i}} \vert  + \int_{t^{\ell}_i}^{t^{\ell}_{i+1}}|L(Y_s)|ds + |\gamma| \sup_{t\in[t^{\ell}_{i},t^{\ell}_{i+1}]} |W_{t} -W_{t^{\ell}_i}|.
\end{align*}
Then, similarly as above we deduce   that  $\Big| \overline{Y}_{t^{\ell}_{i+1}}^{2^{\ell}}- \overline{Y}_{t^{\ell}_{i}}^{2^{\ell}} \Big| =O( {h_{\ell}}^{\frac{1}{2}-\eta})$ since we work on the non-extreme paths events. Thus,  it is clear that if any one of $\overline{Y}_{t^{\ell}_i}^{2^{\ell}},\overline{Y}_{t^{\ell}_{i+1}}^{2^{\ell}}, \overline{Y}_{t^{\ell}_i}^{2^{\ell-1}}, \overline{Y}_{t^{\ell}_{i+1}}^{2^{\ell-1}}$ is greater than ${\mathcal D}+{h_{\ell}}^{\frac{1}{2}-\eta(1+\varepsilon)}$, then the others will be greater than ${\mathcal D}+\frac{1}{2}{h_{\ell}}^{\frac{1}{2}-\eta(1+\varepsilon)}$, for ${h_{\ell}}$ sufficiently small. In this case, if $R$ denotes the set of indices for which none of $\overline{Y}_{t^{\ell}_i}^{2^{\ell}},\overline{Y}_{t^{\ell}_{i+1}}^{2^{\ell}}, \overline{Y}_{t^{\ell}_i}^{2^{\ell-1}}, \overline{Y}_{t^{\ell}_{i+1}}^{2^{\ell-1}}$ is greater than ${\mathcal D}+{h_{\ell}}^{\frac{1}{2}-\eta(1+\varepsilon)}$ and $R^c$ its complementary set,  then 
similarly as above we get  $\prod_{i\in R^c}(1-\overline{q}_{i}^{2^{\ell}})=1+o({h_{\ell}}^a)$ and 
$\prod_{i\in R^c}(1-\overline{q}_{i}^{2^{\ell-1}})=1+o({h_{\ell}}^a)$ for all $a>0$. Thus, we have 
\begin{equation}\label{split:R}
\prod_{i=0}^{2^{\ell}-1}(1-\overline{q}_{i}^{2^{\ell}})
=\prod_{i\in R}(1-\overline{q}_{i}^{2^{\ell}})+o({h_{\ell}}^a)
\mbox{ and }\prod_{i=0}^{2^{\ell}-1}(1-\overline{q}_{i}^{2^{\ell-1}})=\prod_{i\in R}(1-\overline{q}_{i}^{2^{\ell-1}})+o({h_{\ell}}^a), \mbox{ for all }a>0.
\end{equation}
Now, for $i\in R$  we have
\begin{align*}
\big \vert \log\overline{q}_{i}^{2^{\ell}} - \log\overline{q}_{i}^{2^{\ell-1}}  \big\vert &=\frac{2}{\gamma^2 {h_{\ell}}}\left \vert {(\overline{Y}_{t^{\ell}_i}^{2^{\ell}}-{\mathcal D})_+(\overline{Y}_{t^{\ell}_{i+1}}^{2^{\ell}}-{\mathcal D})_+}-{(\overline{Y}_{t^{\ell}_i}^{2^{\ell-1}}-{\mathcal D})_+(\overline{Y}_{t^{\ell}_{i+1}}^{2^{\ell-1}}-{\mathcal D})_+}\right\vert\\
 &\leq \frac{1}{\gamma^2 {h_{\ell}}} {\Big|(\overline{Y}_{t^{\ell}_i}^{2^{\ell}}-{\mathcal D})_+-(\overline{Y}_{t^{\ell}_i}^{2^{\ell-1}}-{\mathcal D})_+\Big|\Big|(\overline{Y}_{t^{\ell}_{i+1}}^{2^{\ell}}-{\mathcal D})_++(\overline{Y}_{t^{\ell}_{i+1}}^{2^{\ell-1}}-{\mathcal D})_+\Big|}\\
&+ \frac{1}{\gamma^2 {h_{\ell}}}{\Big|(\overline{Y}_{t^{\ell}_i}^{2^{\ell}}-{\mathcal D})_++(\overline{Y}_{t^{\ell}_i}^{2^{\ell-1}}-{\mathcal D})_+\Big|\Big|(\overline{Y}_{t^{\ell}_{i+1}}^{2^{\ell}}-{\mathcal D})_+-(\overline{Y}_{t^{\ell}_{i+1}}^{2^{\ell-1}}-{\mathcal D})_+\Big|},
\end{align*}
where we used the relation $f_1 g_1-f_2 g_2=\frac{1}{2}(f_1-f_2)(g_1+g_2) + 
\frac{1}{2}(f_1+f_2)(g_1-g_2)$.
By the Lipschitz property of  the map $x\mapsto(x-{\mathcal D})_+ $ and as $i\in R$, there exists a positive constant $C$ that may vary from line to line such that  \begin{align*}
\big \vert \log\overline{q}_{i}^{2^{\ell}} - \log\overline{q}_{i}^{2^{\ell-1}}  \big\vert  & \leq C{{h_{\ell}}^{-\frac{1}{2}-\eta(1+\varepsilon)}}\Big(|\overline{Y}_{t^{\ell}_i}^{2^{\ell}}-\overline{Y}_{t^{\ell}_i}^{2^{\ell-1}}| + |\overline{Y}_{t^{\ell}_{i+1}}^{2^{\ell}}- \overline{Y}_{t^{\ell}_{i+1}}^{2^{\ell-1}}|\Big)\\
& \leq C {h_{\ell}}^{\frac{1}{2}-\eta(2+\varepsilon)}.
\end{align*}
So, for $h_{\ell}$ sufficiently small we have $\big \vert \log\overline{q}_{i}^{2^{\ell}} - \log\overline{q}_{i}^{2^{\ell-1}} \big\vert  <{h_{\ell}}^{\frac{1}{2}-2\eta(1+\varepsilon)}$. Thus, we have
\begin{align*}
1-\overline{q}_{i}^{2^{\ell-1}}
&=(1-\overline{q}_{i}^{2^{\ell}})+\overline{q}_{i}^{2^{\ell}}\big(1-\exp( \log\overline{q}_{i}^{2^{\ell-1}} - \log\overline{q}_{i}^{2^{\ell}})\big)\\
&\le (1-\overline{q}_{i}^{2^{\ell}})+\overline{q}_{i}^{2^{\ell}}\big(1-\exp( -{h_{\ell}}^{\frac{1}{2}-2\eta(1+\varepsilon)})\big).
\end{align*}
Consequently, we get
\begin{align*}
\prod_{i\in R}(1-\overline{q}_{i}^{2^{\ell-1}})&\leq \prod_{i\in R} \Big( (1-\overline{q}_{i}^{2^{\ell}})+\overline{q}_{i}^{2^{\ell}}\big(1-\exp( -{h_{\ell}}^{\frac{1}{2}-2\eta(1+\varepsilon)})\big)\Big)\\
&\leq  \prod_{i\in R}(1-\overline{q}_{i}^{2^{\ell}})+\big(1-\exp( -{h_{\ell}}^{\frac{1}{2}-2\eta(1+\varepsilon)})\big),
\end{align*}
where we used the convexity property of the function $\phi:x\in[0,1]\mapsto \prod_{i\in R} \Big( (1-\overline{q}_{i}^{2^{\ell}})+\overline{q}_{i}^{2^{\ell}}x\Big)-\prod_{i\in R}(1-\overline{q}_{i}^{2^{\ell}})-x$ and  that $\phi(0)$ and $\phi(1)$ are both non-positive.
 In the same way, we get
$
\prod_{i\in R}(1-\overline{q}_{i}^{2^{\ell}})\le \prod_{i\in R}(1-\overline{q}_{i}^{2^{\ell-1}})+\big(1-\exp( -{h_{\ell}}^{\frac{1}{2}-2\eta(1+\varepsilon)})\big)
$
and  we deduce that
$$
\left | \prod_{i\in R}(1-\overline{q}_{i}^{2^{\ell}})-\prod_{i\in R}(1-\overline{q}_{i}^{2^{\ell-1}})\right|\le \big(1-\exp( -{h_{\ell}}^{\frac{1}{2}-2\eta(1+\varepsilon)})\big)\le {h_{\ell}}^{\frac{1}{2}-2\eta(1+\varepsilon)}.
$$
Then, by \eqref{split:R} we get 
$$
\left|\prod_{i=0}^{2^{\ell}-1}(1-\overline{q}_{i}^{2^{\ell}})-\prod_{i=0}^{2^{\ell}-1}(1-\overline{q}_{i}^{2^{\ell-1}})\right|=O( {h_{\ell}}^{\frac{1}{2}-2\eta(1+\varepsilon)}).
$$
Therefore, as we work on the non-extreme paths events, we deduce using condition \eqref{cond:payoff}  on the payoff function $g$ that 
\begin{align*}
|\overline{Q}_{\ell}^f-\overline{Q}_{\ell}^c|^2&=\left |g(\overline{Y}_{T}^{2^{\ell}})\prod_{i=0}^{2^{\ell}-1}(1-\overline{q}_{i}^{2^{\ell}}) 
- g(\overline{Y}_{T}^{2^{\ell-1}})\prod_{i=0}^{2^{\ell}-1}(1-\overline{q}_{i}^{2^{\ell-1}})\right|^2\\
&\le C|\overline{Y}_{T}^{2^{\ell}} - \overline{Y}_{T}^{2^{\ell-1}}|^2\left(1+|\overline{Y}_{T}^{2^{\ell}}|^{2\nu} + |\overline{Y}_{T}^{2^{\ell-1}}|^{2\nu}\right)+ {h_{\ell}}^{{1}-4\eta(1+\varepsilon)} |g(\overline{Y}_{T}^{2^{\ell}})|^2\\
&\le C\left({h_{\ell}}^{2(1-\eta) -2\eta\nu} +{h_{\ell}}^{1-4\eta(1+\varepsilon)-2\eta -2\eta \nu} \right)\\
&\le C {h_{\ell}}^{1-6\eta(1+\varepsilon)-2\eta \nu},
\end{align*}
where   $C$ is a positive constant that may varies from line to line.
Thus, 
\begin{align*}
\mathbb{E}[(\overline{Q}_{\ell}^f-\overline{Q}_{\ell}^c)^2\mathbbm{1}_{A_3}]&=O({h_{\ell}}^{{1}-6\eta(1+\varepsilon) -2\eta \nu}\times\mathbb P(|\inf_{t\in[0,T]}Y_t - {\mathcal D} | \le {h_{\ell}}^{\frac{1}{2}-\eta(1+\varepsilon)}))
\\&=O({h_{\ell}}^{\frac 32-7\eta(1+\varepsilon)-2\eta  \nu})
\end{align*}
since the random variable $\inf_{t\in[0,T]}Y_t $ has a bounded density on the neighbuorhood of $D$. To complete the proof, we choose  
$$
\eta=\frac{\frac12 -\delta}{7(1+\varepsilon)+2 \nu},
$$ which yields  $\mathbb{E}[(\overline{Q}_{\ell}^f-\overline{Q}_{\ell}^c)^2\mathbbm{1}_{A_3}]=O(h_{\ell}^{1+\delta})$.  Now concerning the second event noticing that 
$
\eta\le \frac{\frac12 -\delta}{7+2 \nu}
$
we easily see that $2-2\eta(1+2\nu)>1+\delta$, for $\delta\in(0,1/2)$, which yields $\mathbb{E}[(\overline{Q}_{\ell}^f-\overline{Q}_{\ell}^c)^2\mathbbm{1}_{A_2}]=o(h_{\ell}^{1+\delta})$. Finally, for the first event, we choose  $q=(1+\gamma)(1+\delta)$ to guarantee that   $\mathbb{E}[(\overline{Q}_{\ell}^f-\overline{Q}_{\ell}^c)^2\mathbbm{1}_{A_1}]=O(h_{\ell}^{1+\delta})$ which is satisfied as soon as   $p>\frac{(1+\delta)(1+\gamma)[7(1+\varepsilon)+2\nu]}{\frac 12 -\delta }$.
 \end{proof}
The following result considers a different assumption on the payoff function, covering an application for the local volatility CEV model.
\begin{thm}\label{MLMC:var-CEV} Let    $g$ denotes a bounded payoff function  satisfying~: $\exists C>0 \mbox{ s.t. }\forall x,y>0$, 
\begin{align}\label{cond:payoff-CEV}
  |g(x)-g(y)|\le C|x-y|(1+|x|^{-\nu} +|x|^{-\nu} ), \mbox{ for } \nu>0.
\end{align}
%We consider the Down-and-Out barrier option price  given by 
%$$\Pi^{\text{D-O},X}_{\mathcal D}:=\mathbb{E}\Big[g(Y_T)\mathbbm{1}_{\{\inf_{t\in[0,T]}Y_t>{\mathcal D}\}}\Big].$$ 
Moreover, assume that conditions  \eqref{ine3tilde}, \eqref{ine4} and \eqref{ine5} are satisfied for   $p>\frac{(1+\delta)(1+\gamma)5(1+\varepsilon)}{\frac 12 -\delta }$, with $\varepsilon,\gamma>0$, $\delta \in(0,1/2)$ and $0<L'_A<\frac 1{2h_{\ell}}$, with $h_{\ell}=2^{-\ell}T$  sufficiently small. If in addition 
 $\inf_{t \in [0, T]}Y_{t}$ has a bounded density in the neighborhood of the barrier, then the multilevel estimator $\bar Q_{\mathcal D}$ given by \eqref{MLMC:D}  for the D-O barrier option satisfies  
 $\rm{Var}(\overline{Q}_{\ell}^f-\overline{Q}_{\ell}^c)= O(h_{\ell}^{1 + \delta})$ .
 \end{thm}
\begin{proof}  We use a similar  decomposition as in the proof of Theorem \ref{MLMC:var}  
\begin{align*}
    \rm{Var}[\overline{Q}_{\ell}^f-\overline{Q}_{\ell}^c] &\leq \mathbb{E}[(\overline{Q}_{\ell}^f-\overline{Q}_{\ell}^c)^2]\\
    &= \mathbb{E}[(\overline{Q}_{\ell}^f-\overline{Q}_{\ell}^c)^2\mathbbm{1}_{A_1}]+\mathbb{E}[(\overline{Q}_{\ell}^f-\overline{Q}_{\ell}^c)^2\mathbbm{1}_{A_2}]+\mathbb{E}[(\overline{Q}_{\ell}^f-\overline{Q}_{\ell}^c)^2\mathbbm{1}_{A_3}]
\end{align*}
where we split the paths into the following three events.\\

\paragraph*{\bf First event $A_1$}We consider any of the extreme path events given in Lemma \ref{eq7} that satisfy \eqref{ext:1}-\eqref{ext:5}, with some {$ \eta > 0$} to be fixed later on. For $\gamma>0$, we use Hölder's inequality to get
\begin{align*}
 \mathbb{E}[(\overline{Q}_{\ell}^f-\overline{Q}_{\ell}^c)^2\mathbbm{1}_{A_1}]&\le  {\mathbb{E}}^{\frac{\gamma}{1+\gamma}}\Big[|\overline{Q}_{\ell}^f-\overline{Q}_{\ell}^c|^{\frac{2(1+\gamma)}{\gamma}}\Big]  {\mathbb{E}}^{\frac{1}{1+\gamma}}[\mathbbm{1}_{A_1} ]\\&\le 2^{\frac{2+\gamma}{1+\gamma}}\left({\mathbb{E}}^{\frac{\gamma}{1+\gamma}}[|\overline{Q}_{\ell}^f|^{\frac{2(1+\gamma)}{\gamma}}]+{\mathbb{E}}^{\frac{\gamma}{1+\gamma}}[|\overline{Q}_{\ell}^c|^{\frac{2(1+\gamma)}{\gamma}}]\right)\Big({\mathbb{P}}[{A_1} ]\Big)^{\frac{1}{1+\gamma}}.
\end{align*}
As the payoff function $f$ is bounded, we deduce by Lemma \ref{eq7} that for $h_{\ell}$ sufficiently small 
\begin{equation}\label{est:1}
\mathbb{E}[(\overline{Q}_{\ell}^f-\overline{Q}_{\ell}^c)^2\mathbbm{1}_{A_1}]=o(h_{\ell}^\frac{q}{1+\gamma})\mbox{ for all $q$ such that } 0<\frac q{\eta}\le p.
\end{equation}
\paragraph*{\bf Second event $A_2$}This event corresponds to the non-extreme paths satisfying $$|\inf_{t\in[0,T]}Y_t - {\mathcal D} | > {h_{\ell}}^{\frac{1}{2}-\eta(1+\varepsilon)}\mbox{ for }\eta \in(0,1/{2(1+\varepsilon)}) \mbox{ with } \varepsilon >0.$$
Let us assume that  $\inf_{t\in[0,T]}Y_t =Y_{\tau}$ with $\tau\in[t^{\ell}_i,t^{\ell}_{i+1}]$ for a given $i\in \{ 0,\dots, 2^{\ell}\}$.
Now, we write  
\begin{align*}
\vert \overline{Y}_{t^{\ell}_i}^{2^{\ell}}-Y_{\tau}\vert &\leq \vert\overline{Y}_{t^{\ell}_i}^{2^{\ell}}-{Y}_{t^{\ell}_i} \vert +\vert {Y}_{t^{\ell}_i}-Y_{\tau}\vert\\
&\leq \vert\overline{Y}_{t^{\ell}_i}^{2^{\ell}}-{Y}_{t^{\ell}_i} \vert  + \int_{t^{\ell}_i}^{t^{\ell}_{i+1}}|L(Y_s)|ds + \gamma \sup_{t\in[t^{\ell}_{i},t^{\ell}_{i+1}]} |W_{t} -W_{t^{\ell}_i}|.
\end{align*}
Then, as we work on the non-extreme paths events then for ${h_{\ell}}$ sufficiently small  we  have that  
$\Big|\overline{Y}_{t^{\ell}_i}^{2^{\ell}}-Y_{\tau} \Big| =O( {h_{\ell}}^{\frac{1}{2}-\eta})$ and therefore $\Big|\overline{Y}_{t^{\ell}_i}^{2^{\ell}}-Y_{\tau} \Big| < {h_{\ell}}^{\frac{1}{2}-\eta(1+\varepsilon)}$. Similarly, we deduce that  $\Big|\overline{Y}_{t^{\ell}_i}^{2^{\ell-1}}-Y_{\tau} \Big| < {h_{\ell}}^{\frac{1}{2}-\eta(1+\varepsilon)}$  since on the non-extreme paths events we have that   $|\overline{Y}_{t^{\ell}_i}^{2^{\ell}}-\overline{Y}_{t^{\ell}_i}^{2^{\ell-1}}|< {h_{\ell}}^{1-\eta}.$\\

$\bullet\,$ First case:  for $Y_{\tau} <{\mathcal D}- {h_{\ell}}^{\frac{1}{2}-\eta(1+\varepsilon)}$, 
the above estimate yields that $\overline{Y}_{t^{\ell}_i}^{2^{\ell}}<{\mathcal D}$.
and so $\Big|\overline{Y}_{t^{\ell}_i}^{2^{\ell-1}}-Y_{\tau} \Big| =O( {h_{\ell}}^{\frac{1}{2}-\eta})$. Hence, we also have $\overline{Y}_{t^{\ell}_i}^{2^{\ell-1}}<{\mathcal D}$, for sufficiently small $h_{\ell}$ which yields that $\overline{Q}_{\ell}^f-\overline{Q}_{\ell}^c=0$.\\
$\bullet\,$ Second case:  for  $Y_{\tau}> {\mathcal D}+{h_{\ell}}^{\frac{1}{2}-\eta(1+\varepsilon)}$, we 
have 
$$
{Y}_{t^{\ell}_i}> {\mathcal D}+{h_{\ell}}^{\frac{1}{2}-\eta(1+\varepsilon)}.
$$
So as $ |\overline{Y}_{t^{\ell}_i}^{2^{\ell}}-{Y}_{t^{\ell}_i}|< {h_{\ell}}^{1-\eta}$ then
$
\overline{Y}_{t^{\ell}_i}^{2^{\ell}}>{\mathcal D}+{h_{\ell}}^{\frac{1}{2}-\eta(1+\varepsilon)}.
$
Using similar arguments we also get  $
\overline{Y}_{t^{\ell}_i}^{2^{\ell-1}}>{\mathcal D}+{h_{\ell}}^{\frac{1}{2}-\eta(1+\varepsilon)}.
$
Then, we easily check that  for  ${h_{\ell}}$ sufficiently small $\prod_{i=0}^{2^{\ell}-1}(1-\overline{q}_{i}^{2^{\ell}})$ and $\prod_{i=0}^{2^{\ell}-1}(1-\overline{q}_{i}^{2^{\ell-1}})$ are both equal to $1+o({h_{\ell}}^a)$ for all  $a>0$. 
Besides, by Taylor formula we have  $|g(x)-g(y)|\le \frac{[f]_{\text{Lip}}}{\nu} |x-y|(|x|^{-\nu - 1}+|y|^{-\nu-1})$ for all  $x,y\in I=(0,+\infty)$.
Therefore, using that  $\overline{Y}_{t^{\ell}_i}^{2^{\ell}}>{\mathcal D}$ and $\overline{Y}_{t^{\ell}_i}^{2^{\ell-1}}>{\mathcal D}$ we  deduce that $\mathbb{E}[(\overline{Q}_{\ell}^f-\overline{Q}_{\ell}^c)^2\mathbbm{1}_{A_2}]=O({h_{\ell}}^{2(1-\eta) })$.     
 \paragraph*{\bf Third event $A_3$} For this last event, we proceed exactly as in Theorem \ref{MLMC:var}  to get 
$$
\left|\prod_{i=0}^{2^{\ell}-1}(1-\overline{q}_{i}^{2^{\ell}})-\prod_{i=0}^{2^{\ell}-1}(1-\overline{q}_{i}^{2^{\ell-1}})\right|=O( {h_{\ell}}^{\frac{1}{2}-2\eta(1+\varepsilon)}).
$$
Therefore, as we work on the non-extreme paths events, we deduce using condition \eqref{cond:payoff}  on the payoff function $g$ that 
\begin{align*}
|\overline{Q}_{\ell}^f-\overline{Q}_{\ell}^c|^2&=\left |g(\overline{Y}_{T}^{2^{\ell}})\prod_{i=0}^{2^{\ell}-1}(1-\overline{q}_{i}^{2^{\ell}}) 
- g(\overline{Y}_{T}^{2^{\ell-1}})\prod_{i=0}^{2^{\ell}-1}(1-\overline{q}_{i}^{2^{\ell-1}})\right|^2\\
&\le C|\overline{Y}_{T}^{2^{\ell}} - \overline{Y}_{T}^{2^{\ell-1}}|^2\left(1+|\overline{Y}_{T}^{2^{\ell}}|^{-2\nu} + |\overline{Y}_{T}^{2^{\ell-1}}|^{-2\nu}\right)+ {h_{\ell}}^{{1}-4\eta(1+\varepsilon)} |g(\overline{Y}_{T}^{2^{\ell}})|^2.
\end{align*}
where   $C$ is a positive constant that may varies from line to line.
The second term of the above upper bound is clearly $O({h_{\ell}}^{{1}-4\eta(1+\varepsilon)})$ since 
$g$ is bounded. For the first term, let us recall that we are in the case where
$ 
Y_{T}\ge Y_\tau\ge D- {h_{\ell}}^{\frac{1}{2}-\eta(1+\varepsilon)}
$
and since  on the non-extreme paths events we have that
$$
\overline{Y}_{T}^{2^{\ell}}-Y_T>-{h_{\ell}}^{1-\eta(1+\varepsilon)} \mbox{ and } 
\overline{Y}_{T}^{2^{\ell}-1}-Y_T>-{h_{\ell}}^{1-\eta(1+\varepsilon)} 
$$
then
$$
\overline{Y}_{T}^{2^{\ell}}>D- {h_{\ell}}^{\frac{1}{2}-\eta(1+\varepsilon)}-{h_{\ell}}^{1-\eta(1+\varepsilon)}>\frac D 2 \mbox{ for } {h_{\ell}} \mbox{ small enough}.
$$
By similar arguments we get that $\overline{Y}_{T}^{2^{\ell-1}}>\frac D2$  for $h_{\ell}$ small enough.
Then, since $\nu>0$, we get 
$$
|\overline{Y}_{T}^{2^{\ell}} - \overline{Y}_{T}^{2^{\ell-1}}|^2\left(1+|\overline{Y}_{T}^{2^{\ell}}|^{-2(\nu+1)} + |\overline{Y}_{T}^{2^{\ell-1}}|^{-2(\nu+1)}\right)=O({h_{\ell}}^{2-2\eta})
$$
Therefore, we have
\begin{align*}
\mathbb{E}[(\overline{Q}_{\ell}^f-\overline{Q}_{\ell}^c)^2\mathbbm{1}_{A_3}]&=O({h_{\ell}}^{{1}-4\eta(1+\varepsilon) }\times\mathbb P(|\inf_{t\in[0,T]}Y_t - {\mathcal D} | \le {h_{\ell}}^{\frac{1}{2}-\eta(1+\varepsilon)}))
\\&=O({h_{\ell}}^{\frac 32-5\eta(1+\varepsilon)})
\end{align*}
since the random variable $\inf_{t\in[0,T]}Y_t $ has a bounded density on the neighborhood of $D$. To complete the proof, we choose  
$$
\eta=\frac{\frac12 -\delta}{5(1+\varepsilon)},
$$ which yields  $\mathbb{E}[(\overline{Q}_{\ell}^f-\overline{Q}_{\ell}^c)^2\mathbbm{1}_{A_3}]=O(h_{\ell}^{1+\delta})$.  Now concerning the second event noticing that 
$
\eta\le \frac{\frac12 -\delta}{5}
$
we easily see that $2(1-\eta)>1+\delta$, for $\delta\in(0,1/2)$, which yields $\mathbb{E}[(\overline{Q}_{\ell}^f-\overline{Q}_{\ell}^c)^2\mathbbm{1}_{A_2}]=o(h_{\ell}^{1+\delta})$. Finally, for the first event, we choose  $q=(1+\gamma)(1+\delta)$ to guarantee that   $\mathbb{E}[(\overline{Q}_{\ell}^f-\overline{Q}_{\ell}^c)^2\mathbbm{1}_{A_1}]=O(h_{\ell}^{1+\delta})$ which is satisfied as soon as   $p>\frac{(1+\delta)(1+\gamma)5(1+\varepsilon)}{\frac 12 -\delta }$.
 \end{proof}

 \section{Application to the CIR process}\label{sec:CIR}
In this section, we consider the problem of pricing  D-O and  U-O barrier options 
$$\pi_{\mathcal D}=\mathbb{E}\Big[f(X_T)\mathbbm{1}_{\{ \inf_{t\in[0,T]}X_t>{\mathcal D}\}}\Big] \mbox { and } \pi_{\mathcal U}=\mathbb{E}\Big[f(X_T)\mathbbm{1}_{\{\sup_{t\in[0,T]}X_t<{\mathcal U}\}}\Big],$$
where $f$ is a Lipschitz payoff function with Lipschitz constant $[f]_{\text{Lip}}$ and $(X_t)_{0\le t\le T}$ denotes the Cox-Ingersoll-Ross (CIR) process solution to
\begin{equation}
\label{eq}
\left \{
\begin{array}{rcl}
dX_t&=&(a-\kappa X_t)dt+\sigma \sqrt{X_t}dW_t\\
X_0&=&x>0,\\
\end{array}
\right.\\
\end{equation}
 with $a \geq \sigma^2/2$, $\kappa \in\R$,  $\sigma > 0$,  $X_0=x>0$. It is well known that  
this SDE admits a unique strong positive solution. 
Applying the  Lamperti transformation, the process $(Y_t)_{0\le t\le T}$ given by  $Y_t=\sqrt{X_t}$ satisfies
 \begin{equation}
\label{lamperti-CIR}
\left \{
\begin{array}{rcl}
 dY_t&=&L(Y_t)dt+ \gamma dW_t\\
Y_0&=&\sqrt{x},\\
\end{array}
\right.\\
\end{equation}
where  $L(y)=\dfrac{a-\sigma^2/4}{2y}- \dfrac{\kappa}{2}y$ and $\gamma = \dfrac{\sigma}{2}.$ Thus, for $g :x \in\mathbb R \mapsto g(x)=f(x^2)$ we get 
$$\pi_{\mathcal D}=\mathbb{E}\Big[g(Y_T)\mathbbm{1}_{\{ \inf_{t\in[0,T]}Y_t>\sqrt{{\mathcal D}}\}}\Big] \mbox { and } \pi_{\mathcal U}=\mathbb{E}\Big[g(Y_T)\mathbbm{1}_{\{\sup_{t\in[0,T]}Y_t<\sqrt{{\mathcal U}}\}}\Big].$$
As $a-\sigma^2/4>0$, we easily check assumptions \eqref{ine2} and \eqref{ine5}. Besides, noticing that $\lim\limits_{y \to 0^+}L'(y)= \lim\limits_{y \to 0^+}-\frac{(a-\sigma^2/4)}{2y^2}-\frac{\kappa}{2}=-\infty$, we deduce that  $L$ is decreasing on $(0, \tilde\varepsilon)$ for $\tilde\varepsilon>0$ small enough.  It is also globally Lipschitz on $[\tilde\varepsilon, +\infty)$ so that assumption \eqref{ine4} is satisfied with $A=\tilde\varepsilon$ and $L'_{A}=\frac{\left| a-\sigma^2/4\right|}{2\varepsilon^2}+\frac{\kappa}{2}. $
 Now, to  check \eqref{ine3tilde} it is enough to show that 
 \begin{equation*}
 \sup_{t\in[0,T]}\mathbb{E}\big[|L'(Y_t)L(Y_t)|^p +|L''(Y_t)|^p+|L'( Y_t)|^{(2 \vee p)}+ |L(Y_t)|^p \big]< \infty
 \end{equation*}
 which is clearly satisfied as soon as 
 \begin{equation}\label{cond:mom}
 \sup_{t\in[0,T]}\mathbb{E}\big[Y_t^{-(4\vee 3p)}\big]=\sup_{t\in[0, T]}\mathbb{E}\big[ X_t^{-(2\vee \frac{3}{2}p)} \big] <\infty \;\mbox{ and }\;\sup_{t\in[0,T]}\mathbb{E}\big[Y_t\big]=\sup_{t\in[0,T]}\mathbb{E}\big[|X_t|^{\frac 12}\big]<\infty. 
 \end{equation}
Recalling that 
 $\sup_{t\in[0, T]}\mathbb{E}\big[ X_t^q\big] <\infty$ for all $q>-\frac{2a}{\sigma^2}$ (see e.g. \cite{dereich2012euler,BAK2}), 
 we easily conclude that condition \eqref{cond:mom} is satisfied when $\sigma^2 < a$ and $p< \frac{4}{3}\frac{a}{\sigma^2}$. Now, as $|g(x)-g(y)|\le[f]_{\text{Lip}} |x-y|(|x|+|y|)$ for all  $x,y\in I=(0,+\infty)$, then $g$ satisfies condition \eqref{cond:payoff} with $\nu=1$.  Consequently, for $\delta\in(0,1/2)$, if we choose    $\frac 43\frac{a}{\sigma^2}>p>\frac{(1+\delta)(1+\gamma)[7(1+\varepsilon)+2\nu]}{\frac 12 -\delta }>18$ for $\varepsilon$, $\delta$ close to zero and  $h_{\ell}$ sufficiently small s.t. $L'_A=\frac{\left| a-\sigma^2/4\right|}{2\varepsilon^2}+\frac{\kappa}{2}<\frac{1}{2h_{\ell}}$ then Theorem \ref{MLMC:var} is valid provided that $\inf_{t \in [0, T]}Y_{t}$  (resp.  $\sup_{t \in [0, T]}Y_{t}$) has a bounded density in the neighborhood of the barrier $\sqrt{{\mathcal D}}$ (resp. $\sqrt{{\mathcal U}}$). 
 More precisely, by the monotone property of $x\in\mathbb R_+^*\mapsto \sqrt{x}$ we have the relationship $\inf_{t\in[0,T]}Y_t= \sqrt{\inf_{t\in[0,T]}X_t}$  (resp. $\sup_{t\in[0,T]}Y_t= \sqrt{\sup_{t\in[0,T]}X_t}$ ), then it is sufficient to prove that 
$\inf_{t \in [0, T]}X_{t}$  (resp.  $\sup_{t \in [0, T]}X_{t}$) has a continuous density in the neighborhood of the barrier which is the aim of the following subsection. 
\subsection{Running maximum of the CIR process}
The aim of this subsection is  to prove that the maximum  of the CIR process \eqref{eq} admits a  continuous density. To do so, let us introduce firstly  the confluent hypergeometric function $_1F_1(x,b, y)$ defined for all $y$, $x \in \mathbb{C}$ and $b \in \mathbb{C}\setminus \{0, -1, -2,\cdots\}$ by
\begin{align}
\label{Kummer}
_1F_1(x,b,y)= \sum_{n=0}^{\infty} \frac{(x)_n}{(b)_n n!} y^n,
\end{align}
where $(x)_n=x(x+1)...(x+n-1)$ stands for the Pochhammer symbol.
\label{maximum density}
\begin{thm}\label{thm:CIR-max}
Let $(X_t)_{0\le t\le T}$ denote the CIR process solution to  \eqref{eq} with $\kappa >0$. Then  $\sup_{t\in[0,T]} X_t$ has a continuous density  on any compact set  $K\subset(X_0,+\infty)$, given by 
\begin{equation}
\label{CIR-Max}
z\in K\mapsto P_{{\rm CIR, Max}}(z)=\\\frac{1}{2\pi}\int_{-\infty}^{+\infty} \e^{(1+iu)T}  {\hat \phi}(u,z)du
\end{equation}
with 
$$
 {\hat \phi}(u,z)=-\dfrac{{_1F_1}((1+iu)/\kappa,2 a/\sigma^2,2\kappa X_0/\sigma^2){_1F_1}((1+iu)/\kappa+1,2 a/\sigma^2+1,2\kappa z/\sigma^2)}{a{_1F_1}((1+iu)/\kappa,2 a/\sigma^2,2\kappa z/\sigma^2)^2}.
$$
\end{thm}
\begin{proof}
At first, let us recall that for the CIR process, we have
\begin{align*}
\mathbb{P}\big[\sup_{0\leq s \leq t} X_s>z\big]&= \frac{1}{2\pi i}\int_{1-i\infty}^{1+i\infty}\frac{e^{t s}}{s}\frac{{_1F_1}(s/\kappa,2 a/\sigma^2,2\kappa X_0/\sigma^2)}{{_1F_1}(s/\kappa,2 a/\sigma^2,2\kappa z/\sigma^2)} ds,\\
&=\frac{1}{2\pi}\int_{-\infty}^{+\infty} \e^{(1+iu)t}\phi(u,z)du, \\
\mbox{ with }\phi(u,z)&:=\frac{{_1F_1}((1+iu)/\kappa,2 a/\sigma^2,2\kappa X_0/\sigma^2)}{(1+iu){_1F_1}((1+iu)/\kappa,2 a/\sigma^2,2\kappa z/\sigma^2)},
\end{align*}
where we recall that for $b,y>0$, the $x$-zeros of $_1F_1(x,b,y)$ are negative real and simple (see e.g. \cite{linetsky2004computing} and \cite{gerhold2020running}). Thereafter,
by  formula (13.4.8) in \cite{gautsch1964handbook} the derivative of the Kummer confluent hypergeometric function is given by 
\begin{equation}\label{eq:der-F}
\frac{\partial _1F_1(a,b,z)}{\partial z}=\frac{a}{b}  {_1F_1(a+1,b+1,z)},
\end{equation}
which gives that
 $$\frac{\partial \phi}{\partial z}(u,z)=-\dfrac{{_1F_1}((1+iu)/\kappa,2 a/\sigma^2,2\kappa X_0/\sigma^2){_1F_1}((1+iu)/\kappa+1,2 a/\sigma^2+1,2\kappa z/\sigma^2)}{a{_1F_1}((1+iu)/\kappa,2 a/\sigma^2,2\kappa z/\sigma^2)^2}.$$  
 On the one hand,  by formula (10.3.53) of \cite{zbMATH06350954} we have for $x,b,y\in \mathbb C$ 
\begin{align*}
{_1F_1}(x,b,y)\sim \left(\frac{y}x\right)^{\frac{1-b}{2}}\frac{\Gamma(b)\Gamma(x-b+1)}{\Gamma(x)}
e^{\frac y2}I_{b-1}(2\sqrt{xy}),  \;\text{as } x\to+\infty
\end{align*}
which is valid inside the sector $|\arg(x)|<\pi$ and uniformly in bounded $b$ and $y$-domains, where $I_{\nu}$ stands for the modified Bessel functions of  the first kind. By formula (9.3.14), we have 
$
I_{\nu}(x)\sim\frac{\e^x}{\sqrt{2\pi x}}, \;\text{as }x\to \infty
$
uniformly in the sector $|\arg(x)|<\frac{\pi}2$, we then deduce that
\begin{align*}
{_1F_1}((1+iu)/\kappa+j, 2a/\sigma^2+j , 2\kappa v/\sigma^2 )&\sim 
 \frac{\e^{\frac{\kappa v}{\sigma^2}}}{\sqrt{4\pi}} \left(\frac{\frac{2\kappa v}{\sigma^2}}{\frac{iu+1}{\kappa}+j}\right)^{\frac 12 -\frac j2-\frac a{\sigma^2}}
\frac{\Gamma(\frac{2a}{\sigma^2}+j) \Gamma(\frac{iu+1}{\kappa}-\frac{2a}{\sigma^2}+1)}{\Gamma(\frac{iu+1}{\kappa}+j)}\notag\\
&\times \left(\frac {2\kappa v}{\sigma^2 }\left( \frac{iu+1}{\kappa}+j\right)\right)^{-\frac 14}
\exp\left(2\sqrt{\frac{2\kappa v}{\sigma^2}(\frac{1+iu}{\kappa}+j)}\right)
\end{align*}
as $u\to\infty$ uniformly on $v$-bounded domain. Using that
$
{\Gamma(x+a)}/{\Gamma(x+b)}\sim x^{a-b}$ as $x\to \infty$ uniformly inside the sector $|\arg(x)|<\pi$,
(see e.g. (6.5.72) of \cite{zbMATH06350954}), and that $$
 \exp\left(2\sqrt{\frac{2\kappa v}{\sigma^2}(\frac{1+iu}{\kappa}+j)}\right)\sim 
 \exp\left(\frac{2\sqrt{v}}{\sigma}(1+i)\sqrt{u}\right) 
$$
as $ u\to +\infty$ uniformly in bounded $v$-domain, we get
\begin{align}
\label{kummer}
{_1F_1}((1+iu)/\kappa+j, 2a/\sigma^2+j , 2\kappa v/\sigma^2 )&\sim
\frac{\e^{\frac{\kappa v}{\sigma^2}}}{2\sqrt{\pi}} \left({\frac{2\kappa v}{\sigma^2}}\right)^{\frac 14 -\frac j2-\frac a{\sigma^2}}
\Gamma(\frac{2a}{\sigma^2}+j) \left(\frac{iu+1}{\kappa}+j\right)^{\frac 14-\frac j2-\frac{a}{\sigma^2}}\notag\\
&\times
%e^{-i\pi \left(\frac{a}{\sigma^2}+\frac{j}2+\frac 14\right)}
e^{2(1+i)\sqrt{\frac{v}{\sigma^2}u}}
\end{align}
as $u\rightarrow \infty$ uniformly in bounded $v$-domain.
%%%%%%%% OLD Part 
%\textcolor{red}{Here}
%using formula (10.3.51)  combined with (10.3.40), (10.3.49),  (9.3.14) and (6.5.72) in \cite{zbMATH06350954}, we get for $j\in\{0,1\}$ and $v\in\{X_0,z\}$\begin{align}
%\label{kummer}
%{_1F_1}((1+iu)/\kappa+j, &2a/\sigma^2+j , 2\kappa v/\sigma^2 )\sim \notag\\&\frac{\Gamma(\frac{2a}{\sigma^2}+j)}{\sqrt{4\pi}} \exp\left(\frac{\kappa v}{\sigma^2}+2\sqrt{\frac{2\kappa v}{\sigma^2}(\frac{1+iu}{\kappa}+j)}\right)\left(\frac{2\kappa v}{\sigma^2}(\frac{1+iu}{\kappa}+j)\right)^{\frac{1}{4}-\frac{2a/\sigma^2+j}{2}}
%\end{align}
%as $u\rightarrow \infty$ uniformly in bounded $v$-domain. Thus, noticing that 
%$$
% \exp\left(\frac{\kappa v}{\sigma^2}+2\sqrt{\frac{2\kappa v}{\sigma^2}(\frac{1+iu}{\kappa}+j)}\right)\sim 
% \exp\left(\frac{\kappa v}{\sigma^2}+\frac{2\sqrt{v}}{\sigma}(1+i)\sqrt{u}\right), \mbox{ as } u\to +\infty
%$$
%uniformly in bounded $v$-domain, 
Thus, we deduce that 
\begin{align*}
\frac{\partial \phi}{\partial z}(u,z) \sim&-\frac{\sqrt{2}}{\sigma} \Big(X_0\Big)^{-\frac{a}{\sigma^2}+\frac{1}{4}} \Big({z}\Big)^{\frac{a}{\sigma^2}-\frac{3}{4}}\big(1+iu\big)^{\frac{a}{\sigma^2}-\frac{1}{4}}\big((1+\kappa)+iu \big)^{-\frac{a}{\sigma^2}-\frac{1}{4}}\\ &\exp\Big((\kappa/\sigma^2)(X_0-z)\Big)\exp\left[\frac{2}{\sigma}\sqrt{u}(1+i)(\sqrt{X_0}-\sqrt{z})\right],
\end{align*}
and therefore 
\begin{align*}
\left|\frac{\partial \phi}{\partial z}(u,z)\right|
&\sim \frac{\sqrt{2}}{\sigma} \Big({X_0}\Big)^{-\frac{a}{\sigma^2}+\frac{1}{4}} \Big({z}\Big)^{\frac{a}{\sigma^2}-\frac{3}{4}} \exp\Big[(\kappa/\sigma^2)(X_0-z)\Big]  u^{-\frac{1}{2}}
\exp\left[\frac{2}{\sigma}\sqrt{u}(\sqrt{X_0}-\sqrt{z})\right]
\end{align*}
as $u\rightarrow \infty$, uniformly in bounded $z$-domain. Hence,  $
\int_{1}^{+\infty} \left|\e^{(1+iu)t}\frac{\partial \phi}{\partial z}(u,z)\right| du <\infty 
$ is uniformly convergent in bounded $z$-domain too.
%%%%%%%%%%%%%%%%
On the other hand, for the integral from $-\infty$ to $-1$, by (10.3.58) of \cite{zbMATH06350954} we have for $x,b,y\in \mathbb C$ 
\begin{align*}
{_1F_1}(-x,b,y)\sim \left(\frac{y}x\right)^{\frac{1-b}{2}}\frac{\Gamma(b)\Gamma(x+1)}{\Gamma(x+b)}
e^{\frac y2}J_{b-1}(2\sqrt{xy}),  \;\text{as } x\to+\infty
\end{align*}
which is valid inside the sector $|\arg(x)|<\pi$ and uniformly in bounded $b$ and $y$-domains, where $J_{\nu}$ stands for the Bessel functions of  the first kind. By (9.2.1)  of \cite{AbrSte}, as $|x|\to\infty$   we have that
\begin{align*}
J_{\nu}(x)&=\sqrt{\frac2{\pi x}}\left(\cos\left(x-\frac{\nu \pi}2-\frac{\pi}4\right)+\e^{|\Im(x)|}\text{O}\left(|x|^{-1}\right)\right), \;|\arg(x)|<\pi.
\end{align*}
Combining  this with the following standard asymptotic expansions  valid for any  $\alpha \in \mathbb R$, $\beta>0$ and $u\to +\infty$
$\cos(\alpha+i \beta u)=\frac 12\e^{-i\alpha+\beta u}+o(e^{\beta |u|})$, we get 
\begin{align*}
{_1F_1}((1-iu)/\kappa+j, 2a/\sigma^2+j , 2\kappa v/\sigma^2 )&\sim \frac{\e^{\frac{\kappa v}{\sigma^2}}}{\sqrt{\pi}} \left(\frac{\frac{2\kappa v}{\sigma^2}}{\frac{iu-1}{\kappa}-j}\right)^{\frac 12 -\frac j2-\frac a{\sigma^2}}
\frac{\Gamma(\frac{2a}{\sigma^2}+j) \Gamma(\frac{iu-1}{\kappa}-j+1)}{\Gamma(\frac{iu-1}{\kappa}+\frac{2a}{\sigma^2})}\\
&\hspace{-3cm}\times \left(\frac {2\kappa v}{\sigma^2 }\left( \frac{iu-1}{\kappa}-j\right)\right)^{-\frac 14}
\cos\left(2\sqrt{\frac {2\kappa v}{\sigma^2 }\left(\frac{iu-1}{\kappa}-j\right)}+\frac{\pi}2(\frac{2a}{\sigma^2}+j-1)+\frac{\pi}4 \right).
\end{align*}
Using that 
$$
\cos\left( 2\sqrt{\frac {2\kappa v}{\sigma^2 }\left(\frac{iu-1}{\kappa}-j\right)}+\frac{\pi}2(\frac{2a}{\sigma^2}+j-1)+\frac{\pi}4  \right)
\\\sim \frac 12 e^{-i\pi \left(\frac{a}{\sigma^2}+\frac{j}2+\frac 14\right)}
e^{2(1-i)\sqrt{\frac{v}{\sigma^2}u}},
$$
as $u\to\infty$ uniformly on $v$-bounded domain and that
$
{\Gamma(z+a)}/{\Gamma(z+b)}\sim z^{a-b}$ as $z\to \infty$ uniformly inside the sector $|\arg(z)|<\pi$,
(see e.g. (6.5.72) of \cite{zbMATH06350954}), we get 
\begin{align}\label{eq:F-inf}
{_1F_1}((1-iu)/\kappa+j, 2a/\sigma^2+j , 2\kappa v/\sigma^2 )&\sim \frac{\e^{\frac{\kappa v}{\sigma^2}}}{2\sqrt{\pi}} \left(\frac{\frac{2\kappa v}{\sigma^2}}{\frac{iu-1}{\kappa}-j}\right)^{\frac 12 -\frac j2-\frac a{\sigma^2}}
\Gamma(\frac{2a}{\sigma^2}+j) \left(\frac{iu-1}{\kappa}-j\right)^{1-j-\frac{2a}{\sigma^2}}\notag\\
&\times \left(\frac {2\kappa v}{\sigma^2 }\left( \frac{iu-1}{\kappa}-j\right)\right)^{-\frac 14}
e^{-i\pi \left(\frac{a}{\sigma^2}+\frac{j}2+\frac 14\right)}
e^{2(1-i)\sqrt{\frac{v}{\sigma^2}u}}\notag\\
&
\sim \frac{\e^{\frac{\kappa v}{\sigma^2}}}{2\sqrt{\pi}} \left({\frac{2\kappa v}{\sigma^2}}\right)^{\frac 14 -\frac j2-\frac a{\sigma^2}}
\Gamma(\frac{2a}{\sigma^2}+j) \left(\frac{iu-1}{\kappa}-j\right)^{\frac 14-\frac j2-\frac{a}{\sigma^2}}\notag\\
&\times
e^{-i\pi \left(\frac{a}{\sigma^2}+\frac{j}2+\frac 14\right)}
e^{2(1-i)\sqrt{\frac{v}{\sigma^2}u}}
\end{align}
%%%%%%%%%%%%%%%%%
%On the other hand, for the integral from $-\infty$ to $-1$, we can use a connection formula such as  (10.1.9) in \cite{zbMATH06350954} to write
%$$
%{_1F_1}((1-iu)/\kappa+j, 2a/\sigma^2+j , 2\kappa v/\sigma^2 )=\exp\left( \frac{2\kappa v}{\sigma^2}\right){_1F_1}(2a/\sigma^2 - (1-iu)/\kappa, 2a/\sigma^2+j , - 2\kappa v/\sigma^2 ).
%$$
%Then, as $\frac{2a}{\sigma^2}\ge \frac 1{\kappa}$,  we use the same formulas as above to deduce that 
%for $j\in\{0,1\}$ and $v\in\{X_0,z\}$\begin{align}
%\label{expansion2}
%{_1F_1}((1-iu)/\kappa+j, &2a/\sigma^2+j , 2\kappa v/\sigma^2 )\sim \notag\\&\frac{\Gamma(\frac{2a}{\sigma^2}+j)}{\sqrt{4\pi}} \exp\left(\frac{\kappa v}{\sigma^2}+2\sqrt{\frac{2 v}{\sigma^2}(1-\frac{2a\kappa}{\sigma^2}-{iu})}\right)\left(-\frac{2\kappa v}{\sigma^2}(\frac{2a}{\sigma^2}+\frac{iu-1}{\kappa})\right)^{\frac{1}{4}-\frac{2a/\sigma^2+j}{2}},
%\end{align}
%as $u\rightarrow +\infty$, uniformly in bounded $v$-domain. Thus, using  that
%$$
%\exp\left(\frac{\kappa v}{\sigma^2}+2\sqrt{\frac{2 v}{\sigma^2}(1-\frac{2a\kappa}{\sigma^2}-{iu})}\right)\sim \exp\left(\frac{\kappa v}{\sigma^2}+\frac{2}{\sigma}(1-{i})\sqrt{ vu}\right), \mbox{ as } u\to +\infty
%$$
%uniformly in bounded $v$-domain we deduce that 
%\begin{align*}
%\frac{\partial \phi}{\partial z}(-u,z) \sim& -\frac{\sqrt{2}}{\sigma}(X_0)^{\frac 14 -\frac{a}{\sigma^2}}(z)^{\frac{a}{\sigma^2}-\frac 34}\Big( 1-\frac{2a\kappa}{\sigma^2}-iu\Big)^{-\frac 12}\\&\times\exp\Big(\frac{3 \kappa}{\sigma^2}(X_0-z)\Big)
%\exp\Big(\frac{2}{\sigma}(1-i)\sqrt{u} (\sqrt{X_0}- \sqrt{z})\Big)
% \end{align*}
and therefore 
\begin{align*}
\left|\frac{\partial \phi}{\partial z}(-u,z)\right|
\sim \frac{\sqrt{2}}{\sigma}(X_0)^{\frac 14 -\frac{a}{\sigma^2}}(z)^{\frac{a}{\sigma^2}-\frac 34}\exp\Big(\frac{ \kappa}{\sigma^2}(X_0-z)\Big)
u^{-\frac{1}2}\exp\Big(\frac{2}{\sigma}\sqrt{u}(\sqrt{X_0}- \sqrt{z})\Big)
\end{align*}
as $u\rightarrow \infty$, uniformly in bounded $z$-domain.
Hence,  we deduce that 
$\int_{-\infty}^{-1} \left|\e^{(1+iu)t}\frac{\partial \phi}{\partial z}(u,z)\right| du <\infty $  uniformly in bounded $z$-domain. Finally,  we complete the proof by noticing that  $(u,z)\in \mathbb R\times K \mapsto \e^{(1+iu)t}\frac{\partial \phi}{\partial z}(u,z)$ is a continuous function  for any compact set  $ K\subset (X_0,+\infty)$. (see e.g. \cite[Theorem B.3]{PinZaf})
\end{proof}
\subsection{Running minimum of the CIR process}
 In the current subsection, we focus on studying the density of the running minimum of the  CIR process \eqref{eq}. For this aim, we introduce the Tricomi confluent hypergeometric function $U(a,b,z)$ defined for all $a$, $z\in \mathbb{C}$ and $b \in \mathbb{C} \setminus\{\pm 0, \pm{1}, \pm{2},...\}$ by 
\begin{equation}\label{def:U}
U(a,b,z)=\frac{\Gamma(1-b)}{\Gamma(1+a-b)}{_1F_1}(a,b,z)+\frac{\Gamma(b-1)}{\Gamma(a)}z^{1-b}{_1F_1}(1+a-b,2-b,z).
\end{equation}
Let us denote by 
$\tau_{X_0 \downarrow z}:= \inf\{t\geq 0 : X_t=z\}$ the first time that the CIR process
$(X_t)_{t\ge0}$ starting at $X_0$ hits the level $z$ satisfying $0 <z<X_0$.
By \cite[Theorem 3]{zbMATH05047805}, the Laplace Transform of  $\tau_{X_0 \downarrow z}$ is explicit and given by 
\begin{equation}
\label{eq:Lap-min-CIR}
\bold{E}[\e^{-s\tau_{X_0 \downarrow z}}]=\frac{U(s/\kappa,2a/\sigma^2,2\kappa X_0/\sigma^2)}{U(s/\kappa,2a/\sigma^2,2\kappa z/\sigma^2)}, \mbox{ for } s>0.
\end{equation}
\begin{thm}\label{thm:CIR-min}
Let $(X_t)_{0\le t\le T}$ denote the CIR process solution to  \eqref{eq} with $\kappa>0$.  Then  $\inf_{t\in[0,T]} X_t$ has a continuous density  on any compact set  $K\subset(0,X_0)$, given by 
\begin{equation}
\label{CIR-Min}
z\in K\mapsto P_{{\rm CIR, Min}}(z)=\frac{1}{2\pi}\int_{-\infty}^{+\infty} \e^{(1+iu)T} {\hat \psi}(u,z)du
\end{equation}
with 
$$
{\hat \psi}(u,z)=\dfrac{2 U((1+iu)/\kappa,2 a/\sigma^2,2\kappa X_0/\sigma^2)U((1+iu)/\kappa+1,2 a/\sigma^2+1,2\kappa z/\sigma^2)}{\sigma^2 U((1+iu)/\kappa,2 a/\sigma^2,2\kappa z/\sigma^2)^2}.$$
\end{thm}
\begin{proof}
Making use of an inverse Laplace transform, the cumulative distribution function of the running minimum  CIR process can be expressed as 
\begin{align*}
    \mathbb{P}\left[ \inf_{0\leq s \leq t} X_s \leq z\right]&= \mathbb{P}\left[ \tau_{X_0 \downarrow z} \leq t \right]\\
    &= \frac{1}{2\pi i}\int_{1-i\infty}^{1+i\infty}\frac{e^{t s}}{s}\frac{U(s/\kappa,2a/\sigma^2,2\kappa X_0/\sigma^2)}{U(s/\kappa,2a/\sigma^2,2\kappa z/\sigma^2)} ds\\
    &=\frac{1}{2\pi}\int_{-\infty}^{+\infty} \e^{(1+iu)t} \psi(u,z)du,\\
    \mbox{ with } \psi(u,z)&:= \frac{U((1+iu)/\kappa,2 a/\sigma^2,2\kappa X_0/\sigma^2)}{(1+iu)U((1+iu)/\kappa,2 a/\sigma^2,2\kappa z/\sigma^2)}
\end{align*}
and where we recall that for $b,y>0$, the $x$-zeros of $U(x,b,y)$ are negative real and simple (see e.g. \cite{linetsky2004computing}). 
By formula (13.4.21) in \cite{gautsch1964handbook} the derivative of the Tricomi confluent hypergeometric function is given by
\begin{equation}\label{Kummer-deriv}
\frac{\partial U(a,b,z)}{\partial z}=-aU(a+1,b+1,z)
\end{equation}
which gives that
 $$\frac{\partial \psi(u,z)}{\partial z}=\dfrac{2 U((1+iu)/\kappa,2 a/\sigma^2,2\kappa X_0/\sigma^2)U((1+iu)/\kappa+1,2 a/\sigma^2+1,2\kappa z/\sigma^2)}{\sigma^2 U((1+iu)/\kappa,2 a/\sigma^2,2\kappa z/\sigma^2)^2}.$$
Now using formulas (10.3.37), (10.3.31) and (9.1.3) in \cite{zbMATH06350954}, we get for $x,b$ and $y\in \mathbb C$
\begin{equation}\label{asymp:U+}
U(x, b, y) \sim \frac{\sqrt{\pi}}{\Gamma(x)}x^{-\frac{3}{4}+\frac{1}{2}b} y^{\frac{1}{4}-\frac{1}{2}b} \e^{\frac y{2}-2\sqrt{xy}}, \; \text{ as } x\to+\infty,
\end{equation}
inside the sector $|\arg(x)|<\pi$ and uniformly in bounded $b$ and $y$-domains.
For $j\in\{0,1\}$ and $v\in\{X_0,z\}$, we get 
\begin{align}
\label{tric}
U(&(1+iu)/\kappa+j, 2a/\sigma^2+j , 2\kappa v/\sigma^2 )\sim \notag\\&\frac{\sqrt{\pi}}{\Gamma\left(\frac{1+iu}{\kappa}+j\right)} \exp\left(\frac{\kappa v}{\sigma^2}-2\sqrt{\frac{2\kappa v}{\sigma^2}(\frac{1+iu}{\kappa}+j)}\right)\left(\frac{2\kappa v}{\sigma^2}\right)^{\frac 1 4 -\frac{a}{\sigma^2}-\frac{j}2}\left(\frac{1+iu}{\kappa}+j\right)^{-\frac{3}{4}+\frac{a}{\sigma^2}+\frac j 2}
\end{align}
as $u\rightarrow \infty$ uniformly in bounded $v$-domain.  Besides, we easily check that
\begin{equation}\label{eq:expo}
 \exp\left(\frac{\kappa v}{\sigma^2}-2\sqrt{\frac{2\kappa v}{\sigma^2}(\frac{1+iu}{\kappa}+j)}\right)\sim 
 \exp\left(\frac{\kappa v}{\sigma^2}-\frac{2\sqrt{v}}{\sigma}(1+i)\sqrt{u}\right), \mbox{ as } u\to +\infty,
\end{equation}
also uniformly in bounded $v$-domain. We deduce that 
\begin{align*}
   \frac{\partial \psi(u,z)}{\partial z}&\sim  \frac{2}{\sigma^2}    \left(\frac{1+iu}{\kappa} \right)^{-\frac{1}{4}-\frac{a}{\sigma^2}} \left(\frac{2\kappa X_0}{\sigma^2}\right)^{\frac{1}{4}-\frac{a}{\sigma^2}} \left(\frac{(1+\kappa)+iu}{\kappa} \right)^{-\frac{1}{4}+\frac{a}{\sigma^2}}\left(\frac{2\kappa z}{\sigma^2}\right)^{-\frac{3}{4}+\frac{a}{\sigma^2}}\\&
   \exp\Big((\kappa/\sigma^2)(X_0-z)\Big) \exp\left(\frac{2}{\sigma}{(1+i)}(\sqrt{z}-\sqrt{X_0})\sqrt{u}\right)
\end{align*}
and then 
\begin{align*}
   \left|\frac{\partial \psi(u,z)}{\partial z}\right|&\sim  \frac{\sqrt{2}}{\sigma}    \left({X_0}\right)^{\frac{1}{4}-\frac{a}{\sigma^2}} \left({z}\right)^{-\frac{3}{4}+\frac{a}{\sigma^2}}\exp\Big((\kappa/\sigma^2)(X_0-z)\Big)u^{-\frac12}
    \exp\left(\frac{2}{\sigma}(\sqrt{z}-\sqrt{X_0})\sqrt{u}\right)
\end{align*}
as $u\rightarrow \infty$ uniformly in bounded $z$-domain  in $(0,X_0)$.  Hence,  $
\int_{1}^{+\infty} \left|\e^{(1+iu)t}\frac{\partial \psi}{\partial z}(u,z)\right| du <\infty 
$ is uniformly convergent in any bounded $z$-domain  in $(0,X_0)$. 

On the other hand, for the integral from $-\infty$ to $-1$, we have by formula  (10.3.68) of \cite{zbMATH06350954} that for $x,b$ and $y\in \mathbb C$
\begin{equation}\label{asymp:U-}
U(-x, b, y) \sim\left(\frac yx\right)^{\frac{1-b}{2}} \Gamma(x+1)\e^{\frac y2} 
\left(\cos(\pi x)J_{b-1}(2\sqrt{xy}) + \sin(\pi x)Y_{b-1}(2\sqrt{xy})\right), \; \text{ as } x\to+\infty,
\end{equation}
which is valid inside the sector $|\arg(x)|<\pi$ and uniformly in bounded $b$ and $y$-domains, where $J_{\nu}$ and $Y_{\nu}$ stand for the Bessel functions of respectively  the first and the second kind. By (9.2.1) and (9.2.2)  of \cite{AbrSte}  we have as $|x|\to\infty$ 
\begin{align}
J_{\nu}(x)&=\sqrt{\frac2{\pi x}}\left(\cos\left(x-\frac{\nu \pi}2-\frac{\pi}4\right)+\e^{|\Im(x)|}\text{O}\left(|x|^{-1}\right)\right), \;|\arg(x)|<\pi\label{eq:j}\\
Y_{\nu}(x)&=\sqrt{\frac2{\pi x}}\left(\sin\left(x-\frac{\nu \pi}2-\frac{\pi}4\right)+\e^{|\Im(x)|}\text{O}\left(|x|^{-1}\right)\right),\;|\arg(x)|<\pi.\label{eq:y}
\end{align}
Combining all this with the following standard asymptotic expansions  valid for any  $\alpha \in \mathbb R$, $\beta>0$ and $u\to +\infty$
$\cos(\alpha+i \beta u)=O(e^{\beta u})$,  $\sin(\alpha+i \beta u)=O(e^{\beta u})$ and with the relation 
$\cos(z_1)\cos(z_2)+\sin(z_1)\sin(z_2)=\cos(z_1-z_2)$, $z_1, z_2\in \mathbb C$,  we get
\begin{align*}
U\left(-(\frac{iu-1}{\kappa}-j), \frac{2a}{\sigma^2+j} , 2\frac{\kappa v}{\sigma^2} \right)&\sim
\frac{\e^{\frac{\kappa v}{\sigma^2}} }{\sqrt{\pi}}
\left( \frac{\frac {2\kappa v}{\sigma^2 }}{ \frac{iu-1}{\kappa}-j}\right)^{\frac{1-j}{2}-\frac{a}{\sigma^2}}
\left(\frac {2\kappa v}{\sigma^2 }\left( \frac{iu-1}{\kappa}-j\right)\right)^{-\frac 14}\\
&\hspace{-5cm}\times\Gamma\left(\frac{iu-1}{\kappa}-j+1\right)\cos\left( \pi\left( \frac{iu-1}{\kappa}-j\right) -2\sqrt{\frac {2\kappa v}{\sigma^2 }\left(\frac{iu-1}{\kappa}-j\right)}+\frac{\pi}2(\frac{2a}{\sigma^2}+j-1)+\frac{\pi}4  \right)
\end{align*}
as $u\rightarrow +\infty$ uniformly in bounded $v$-domain.
Using that 
\begin{multline*}
\cos\left( \pi\left( \frac{iu-1}{\kappa}-j\right) -2\sqrt{\frac {2\kappa v}{\sigma^2 }\left(\frac{iu-1}{\kappa}-j\right)}+\frac{\pi}2(\frac{2a}{\sigma^2}+j-1)+\frac{\pi}4  \right)
\\\sim \frac 12 e^{-i\pi \left(\frac{a}{\sigma^2}- \frac 1{\kappa}-\frac{j}2+\frac 14\right)}
e^{\frac{\pi}{\kappa} u-2(1-i)\sqrt{\frac{v}{\sigma^2}u}},
\end{multline*}
we get
\begin{align*}
U\left(-(\frac{iu-1}{\kappa}-j), \frac{2a}{\sigma^2+j} , 2\frac{\kappa v}{\sigma^2} \right)&\sim \frac{e^{\frac {\kappa v}{\sigma^2 }}}{2\sqrt{\pi}}e^{-i\pi \left(\frac{a}{\sigma^2}- \frac 1{\kappa}-\frac{j}2+\frac 14\right)}
e^{\frac{\pi}{\kappa} u-2(1-i)\sqrt{\frac{v}{\sigma^2}u}}\left(\frac {2\kappa v}{\sigma^2 }\right)^{\frac 14-{{\frac j2-\frac{a}{\sigma^2}}{}}}
\\&\times\Gamma\left(\frac{iu-1}{\kappa}-j+1\right)
\left( \frac{iu-1}{\kappa}-j\right)^{-\frac 34+ \frac j2+\frac a{\sigma^2}},
\end{align*}
as $u\rightarrow +\infty$ uniformly in bounded $v$-domain. Thus,
\begin{align*}
   \left|\frac{\partial \psi(u,z)}{\partial z}\right|&\sim  \frac{\sqrt{2}}{\sigma}    \left({X_0}\right)^{\frac{1}{4}-\frac{a}{\sigma^2}} \left({z}\right)^{-\frac{3}{4}+\frac{a}{\sigma^2}}\exp\Big((\kappa/\sigma^2)(X_0-z)\Big)u^{-\frac12}
    \exp\left(\frac{2}{\sigma}(\sqrt{z}-\sqrt{X_0})\sqrt{u}\right)
\end{align*}
as $u\rightarrow \infty$ uniformly in bounded $z$-domain  in $(0,X_0)$. Hence,  we deduce that 
\begin{equation*}
\int_{-\infty}^{-1} \left|\e^{(1+iu)t}\frac{\partial \psi}{\partial z}(u,z)\right| du <\infty 
\end{equation*}
  uniformly in any bounded $z$-domain subset of $(0,X_0)$. Finally,  we complete the proof by noticing that  $(u,z)\in\mathbb R\times K\mapsto \e^{(1+iu)t}\frac{\partial \psi}{\partial z}(u,z)$ is a continuous function for any compact set  $K\subset(0,X_0)$ (see e.g. \cite[Theorem B.3]{PinZaf}). 
\end{proof}
%%%%%%%%%%%%%%%%%%%%%%
\subsection{Numerical tests}
For these numerical tests, we consider the problem of pricing  D-O and  U-O barrier options 
$\pi_{\mathcal D}=\mathbb{E}\Big[f(X_T)\mathbbm{1}_{\{ \inf_{t\in[0,T]}X_t>{\mathcal D}\}}\Big] \mbox { and } \pi_{\mathcal U}=\mathbb{E}\Big[f(X_T)\mathbbm{1}_{\{\sup_{t\in[0,T]}X_t<{\mathcal U}\}}\Big],$
where the payoff function  $f(x)=\e^{-rT}(x-K)_+$ and $(X_t)_{0\le t\le T}$ is the CIR process solution to \eqref{eq}. By the Lamperti transform we get 
$$\pi_{\mathcal D}=\mathbb{E}\Big[g(Y_T)\mathbbm{1}_{\{ \inf_{t\in[0,T]}Y_t>\sqrt{{\mathcal D}}\}}\Big] \mbox { and } \pi_{\mathcal U}=\mathbb{E}\Big[g(Y_T)\mathbbm{1}_{\{\sup_{t\in[0,T]}Y_t<\sqrt{{\mathcal U}}\}}\Big],$$
where $g(x)=\e^{-rT}(x^2-K)_+$ and $(Y_t)_{t\in[0,T]}$ is solution to \eqref{lamperti-CIR}.  We approximate $\pi_{\mathcal D}$ (resp. $\pi_{\mathcal U}$)  by the improved MLMC algorithm  $\bar Q_{\mathcal D}$ given in   \eqref{MLMC:D}  (resp.  $\bar P_{\mathcal U}$ given in   \eqref{MLMC:U}), where we used our  interpolated drift implicit scheme
\begin{align*}
    \overline{Y}^n_t&=    \overline{Y}^n_{t_{i}}+ \left( \dfrac{a-\gamma^2}{2\overline{Y}^n_{t_{i+1}}}- \dfrac{\kappa}{2}\overline{Y}^n_{t_{i+1}}
\right)(t-t_{i})+\gamma(W_t-W_{t_{i}}),~~~~ \text{for}\, t \in [t_i,t_{i+1}]\\
Y_0&=\sqrt{X_0},
\end{align*}
with $\gamma = \dfrac{\sigma}{2}$. 
For $n$ large enough, the positive solution to the above implicit scheme is explicit and given by 
$$
\overline{Y}^n_{t_{i+1}}=\frac{\sqrt{(2+\kappa\frac Tn)( a-\gamma^2)\frac Tn + (\gamma(W_{t_{i+1}}-W_{t_i})+\overline{Y}^n_{t_{i}})^2}+\gamma(W_{t_{i+1}}-W_{t_i})+\overline{Y}^n_{t_{i}} }{2+\kappa\frac Tn}.
$$
To illustrate the performance of our MLMC algorithms we consider the same comparison procedure as in \cite{giles2015}.  We considered the same model and option parameters proposed by  \cite{Davydov}. We take $r=0.1$, $X_0=100$, $a=0$, $\kappa=-0.1$, $\sigma=2.5$ and $T=0.5$.  For the D-O option the strike is $K=95$, and the barrier ${\mathcal D}=90$ and  for the U-O option the strike is $K=105$ and the barrier ${\mathcal U}=120$.   The benchmark prices given in \cite{Davydov}  for the D-O  (resp. U-O)  option is 10.6013 (resp. 0.7734). The performance of the improved MLMC is given in the tables and figure below. 
\begin{table}[h!]
\begin{center}
\begin{tabular}{ |c | c | c | c | c |}
\hline
 Accuracy & Price & MLMC cost & MC cost & Saving\\ 
 \hline
 \hline
$10^{-3}$ & 10.669 & $2.588\times 10^8$ & $6.752\times 10^{10} $ & 260.91\\  
$5\times10^{-3}$ & 10.668 & $1.051\times 10^7$ & $3.376\times 10^8$ & 32.13\\  
$10^{-2}$ & 10.668 & $2.510\times 10^6$ & $4.220\times 10^7 $ & 16.81 \\  
$2\times10^{-2}$ & 10.677  & $6.187\times 10^5$ & $5.275\times 10^6$ & 8.52\\
\hline
\end{tabular}
\caption{MLMC complexity tests for D-O barrier option pricing of $\pi_{\mathcal D}$ }
\end{center}
\end{table}

%%%%%%%
\begin{table}[h!]
\begin{center}
\begin{tabular}{ |c | c | c | c | c |}
\hline
 Accuracy & Price & MLMC cost & MC cost & Saving\\ 
 \hline
 \hline
$10^{-3}$ & 0.77200 & 4.674$\times 10^6$ & 4.221$\times 10^8 $ & 90.32\\  
$5\times10^{-3}$ & 0.76926 & 1.571$\times 10^5$ & 2.11$\times 10^6 $ & 13.44\\  
$10^{-2}$ & 0.77015 & $3.809\times 10^4$ & $2.638\times 10^5 $ & 6.93 \\  
$2\times10^{-2}$ & 0.78168  & $1.463\times 10^4$ & $6.596\times 10^4$ & 4.51\\
\hline
\end{tabular}
\caption{MLMC complexity tests for U-O barrier option pricing $\pi_{\mathcal U}$}
\end{center}
\end{table}
\begin{figure}[h!]
\begin{subfigure}{.45\textwidth}
  \centering
  % include first image
  \includegraphics[width=.9\linewidth]{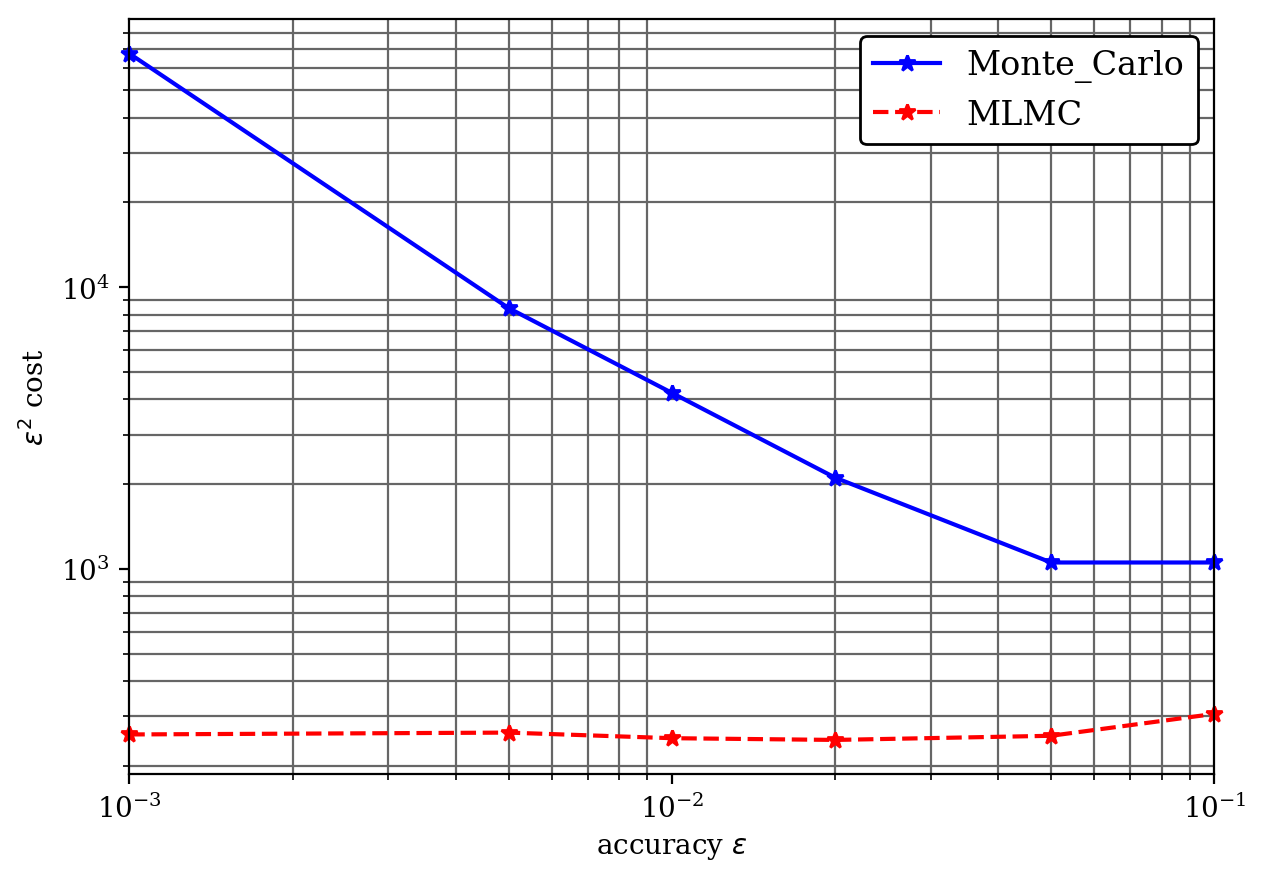}  
  \caption{Approximation of $\pi_{\mathcal D}$}
  \label{fig:sub-first}
\end{subfigure}
\begin{subfigure}{.45\textwidth}
  \centering
  % include second image
  \includegraphics[width=.9\linewidth]{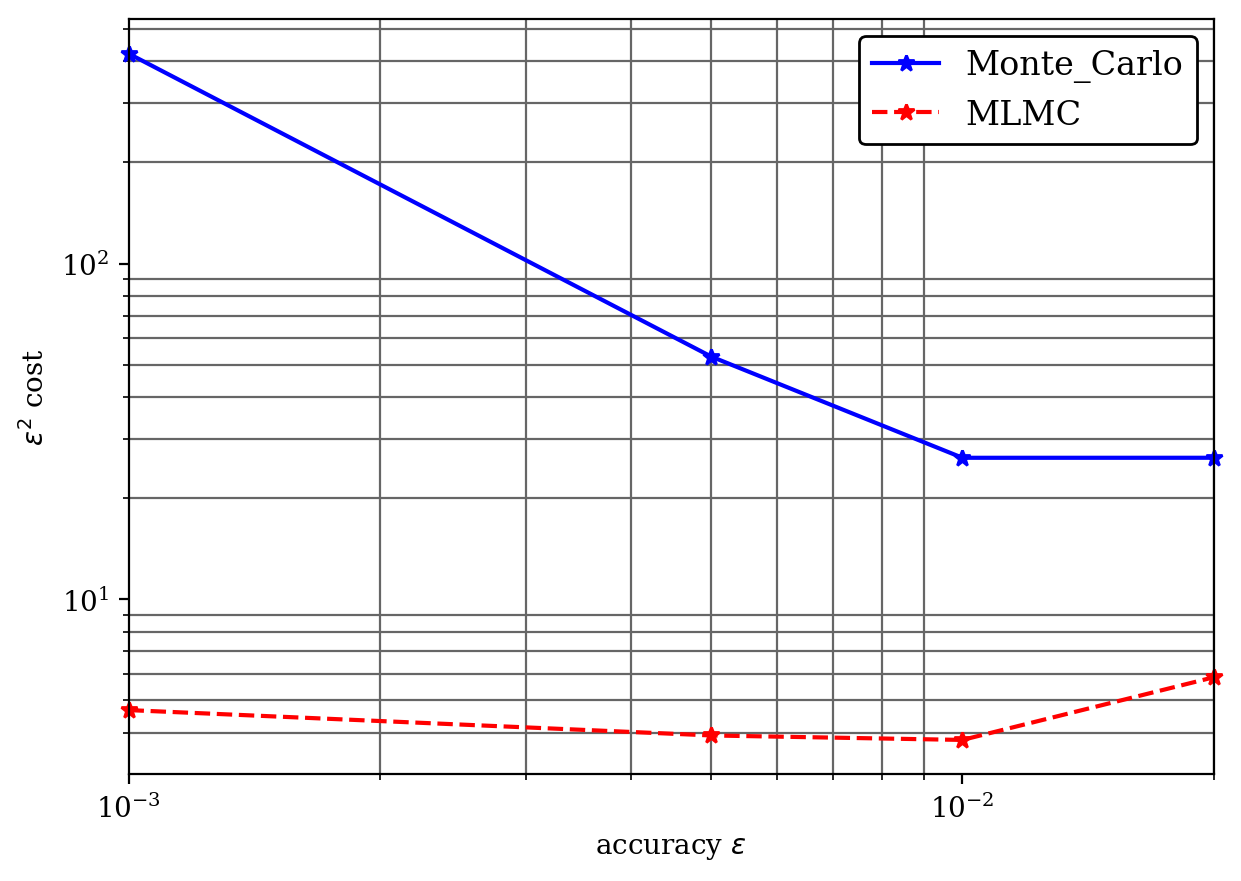}  
  \caption{Approximation of $\pi_{\mathcal U}$}
  \label{fig:sub-second}
\end{subfigure}
\caption{ Comparison for the performances of MLMC vs classical MC algorithm under the CIR model.}
\label{fig:fig}
\end{figure}
The  numerical results illustrates Theorem \ref{MLMC:var}, that is the improved MLMC algorithm reaches for a given precision $\varepsilon$  the optimal time complexity $O(\varepsilon^{-2})$ for D-O and U-O barrier options under the CIR model.
%%%%%%%%%%%%%%%%%%%%%%%%%
%                               CEV 
%%%%%%%%%%%%%%%%%%%%%%%%%
\section{Application to CEV process}\label{sec:CEV}
In this section, we consider the  CEV process solution to
\begin{equation}\label{sde-cev}
dX_t=\mu X_t dt + \sigma X_t^{\alpha}dW_t, \hspace{2mm} t\geq 0, \hspace{2mm}, X_0>0, \;\; \mu\in\mathbb R \;\; \mbox{ and } \alpha>1.
\end{equation}
We consider the problem of pricing  an U-O barrier option $\Pi^{\text{U-O}, X}_{\mathcal D}:=\mathbb{E}\Big[f(X_T)\mathbbm{1}_{\{ \sup_{t\in[0,T]}X_t<{\mathcal D}\}}\Big] $ with barrier $\mathcal D$  where $f$ is a bounded Lipschitz function with Lipschitz constant $[f]_{\text{Lip}}$.
For $\alpha>1$, by Feller's test the solution  of \eqref{sde-cev} is known to be positive 
(see e.g. \cite[Lemma 6.4.4.1]{zbMATH05017324}). Applying the  Lamperti transformation, the process $(Y_t)_{0\le t\le T}$ given by  $Y_t={X_t^{1-\alpha}}$ is well defined on $I=(0,+\infty)$ and satisfies
 \begin{equation}
\label{lamperti}
\left \{
\begin{array}{rll}
 dY_t&=&L(Y_t)dt+ \gamma dW_t\\
Y_0&=&{X_0^{1-\alpha}},\\
\end{array}
\right.\\
\end{equation}
where  $L(y)=(1-\alpha)\left( \mu y - \alpha\frac{\sigma^2}2 y^{-1}\right)$ and $\gamma = \sigma(1-\alpha).$  
Thus, as the map $x\mapsto x^{1-\alpha}$ is decreasing, our initial pricing problem is transformed as follows on the Lamperti transform space
\begin{equation}\label{pricing:prob-cev}
\Pi^{\text{U-O}, X}_{\mathcal D}=\mathbb{E}\Big[g(Y_T)\mathbbm{1}_{\{ \inf_{t\in[0,T]}Y_t>{{\mathcal D}^{1-\alpha}}\}}\Big],
\end{equation}
with $g : x\in\mathbb R \mapsto f(x^{\frac{1}{1-\alpha}})$.
As $\lim\limits_{y \to 0^+}L'(y)= \lim\limits_{y \to 0^+} (1-\alpha)(\mu + \alpha\frac{\sigma^2}2 y^{-2} )=-\infty$, we deduce that  $L$ is decreasing on $(0, \tilde\varepsilon)$ for $\tilde\varepsilon>0$ small enough and it is clearly globally Lipschitz on $[\tilde\varepsilon, +\infty)$ so that assumption \eqref{ine4} is satisfied with $A=\tilde\varepsilon$ and $L'_{A}=(\alpha-1)\left(|\mu| + \alpha\frac{\sigma^2}{2} {\varepsilon}^{-2}\right) $. Also, we easily check assumptions \eqref{ine2} and \eqref{ine5}.  On the one hand, by Itô's formula the process $(Z_t)_{0\le t\le T}$ given by $Z_t=\frac{X_t^{-2(\alpha-1)}}{4(\alpha-1)^2}$ is a CIR process solution to

\begin{align}
\label{cevtocir}
\left \{
\begin{array}{rcl}
dZ_t&=&(a-\kappa Z_t)dt  - \sigma\sqrt{Z_t}dW_t , \nonumber\\
Z_0&=&\frac{X_0^{-2(\alpha-1)}}{4(\alpha-1)^2}, \\
\end{array}
\right.\\
\end{align}
with  $a=\frac{\sigma^2(2\alpha-1)}{4(\alpha-1)}$ and $\kappa=2\mu(\alpha-1)$. Thanks to this  second transformation we deduce that $\sup_{t\in[0,T]}\mathbb E[Y^q_t]<\infty$ for $q>- \frac{2\alpha-1}{2(\alpha-1)}$. On the other hand to check assumption \eqref{ine3tilde} it is enough to show that 
 \begin{equation}\label{cond:mom-cev}
 \sup_{t\in[0,T]}\mathbb{E}\big[|L'(Y_t)L(Y_t)|^p +|L''(Y_t)|^p+|L'( Y_t)|^{(2 \vee p)}+ |L(Y_t)|^p \big]< \infty
 \end{equation}
 which is  satisfied if  $
 \sup_{t\in[0,T]}\mathbb{E}\big[Y_t^{-(4\vee 3p)}\big] <\infty. 
$
This condition is satisfied when   $4< \frac{2\alpha-1}{2(\alpha-1)}$ (which corresponds to  have
$\alpha\in(1,\frac 76))$ and $p<\frac{2\alpha-1}{6(\alpha-1)}$.
\\

Besides, since by Taylor formula we have  $|g(x)-g(y)|\le \frac{[f]_{\text{Lip}}}{\alpha-1} |x-y|(|x|^{-\frac{\alpha}{\alpha-1}}+|y|^{-\frac{\alpha}{\alpha-1}})$ for all  $x,y\in I=(0,+\infty)$, then $g$ satisfies condition \eqref{cond:payoff-CEV} with $\nu=-\frac{\alpha}{\alpha-1}$.  Hence, for $\delta\in(0,1/2)$, if we choose  $\alpha$ such that $1<\alpha<\frac{59}{58}<\frac 76$ then we can find $p$ such that  $\frac{2\alpha-1}{6(\alpha-1)}>p>\frac{(1+\delta)(1+\gamma)5(1+\varepsilon)}{\frac 12 -\delta }>10$. Finally, if we choose  $h_{\ell}$ sufficiently small such that $L'_A=(\alpha-1)\left(|\mu| + \alpha\frac{\sigma^2}{2} {\varepsilon}^{-2}\right) <\frac{1}{2h^{\ell}}$, then Theorem \ref{MLMC:var-CEV} is valid provided that $\inf_{t \in [0, T]}Y_{t}$  has a bounded density in the neighbourhood of the barrier ${\mathcal D}^{1-\alpha}$. 
 By the monotone property of $x\in\mathbb R_+^*\mapsto x^{1-\alpha}$ we have the relationship $\inf_{t\in[0,T]}Y_t= (\sup_{t\in[0,T]}X_t)^{1-\alpha}$, then it is sufficient to prove that 
$\sup_{t \in [0, T]}X_{t}$  has a continuous density in the neighborhood of the barrier which is the aim of the following subsection. 
\begin{rem}
One can also consider the CEV process for  $\alpha\in(\frac12,1)$ solution to 
\begin{equation}\label{sde:cev-2}
dX_t=(a-\kappa X_t)dt + \sigma Y^{\alpha}_tdW_t, \mbox{ with } X_0>0, a>0.
\end{equation}
It can be easily checked that for $a>0$ this SDE is well defined on $I=(0,+\infty)$ (see \cite[Section 3]{alfonsi2013strong}). The associated drift implicit Euler scheme is well defined on $I$ too and satisfy the conditions of our theoretical setting.  However, the condition that $\inf_{t\in[0,T]}X_t$ or $\sup_{t\in[0,T]}X_t$  admits a continuous density in the neighborhood of the barrier seems to be a  challenging problem, since we dont have explicit Laplace transform of the corresponding hitting times as it is the case for the previous CEV process  solution to \eqref{sde-cev}. In counterpart, the efficiency of the MLMC method is still confirmed by our numerical tests  for the model \eqref{sde:cev-2}.
\end{rem}

%%%%%%%%%%%%%%%%%%%%%%%%%%%%
\subsection{Running maximum of the CEV process}
%%%%%%%%%%%%%%%%%%%%%%%%%%%%
Let us denote by  $\tau_{X_0 \uparrow z}:= \inf\{t\geq 0 : X_t=z\}$ the first time that the CEV process
$(X_t)_{t\ge0}$ starting at $X_0$ hits the level $z>X_0$. From  \cite[subsections 5.3.6 and  6.4.5]{zbMATH05017324}, the Laplace transform of the hitting time $\tau_{X_0 \uparrow z}$ is given by 
\begin{align}\label{lap:cev-max}
\bold{E}[\e^{-s\tau_{X_0 \uparrow z}}]&=\Big(\frac{X_0}{z}\Big)^{\beta+\frac{1}{2}}\exp\Big(\frac{\epsilon}{2}c(X_0^{-2\beta}-z^{-2\beta})\Big)\dfrac{W_{k,n}(cX_0^{-2\beta})}{W_{k,n}(cz^{-2\beta})},
%&=\dfrac{U\Big(\frac{s}{2 \mu \beta}, \frac{1}{2\beta}+1, cX_0^{-2\beta}\Big)}{U\Big(\frac{s}{2 \mu \beta}, \frac{1}{2\beta}+1, cz^{-2\beta}\Big)}
\end{align}
with $\epsilon=\mbox{sign}(\mu \beta)$, $n=\frac{1}{4\beta}$, $k= \epsilon \Big( \frac{1}{2}+ \frac{1}{4\beta} \Big)- \frac{s}{2|\mu \beta|}$ and $W_{k,n}$ the Whittaker's function $W_{k,n}(y)=y^{n+\frac{1}{2}} \e^{-y/2}U(n-k+\frac{1}{2}, 2n+1,y)$, where $U$ denotes the  confluent hypergeometric function of second kind defined in \eqref{def:U} and with $\beta=\alpha-1$ and $c=\dfrac{|\mu|}{\beta \sigma^2}$.
\begin{thm}\label{thm:CEV-max}
Let $(X_t)_{0\le t\le T}$ denotes the CEV process solution to  \eqref{sde-cev}.  Then  $\sup_{t\in[0,T]} X_t$ has a continuous density  on any compact set  $K\subset(X_0,+\infty)$, given by
\begin{equation}
\label{CEV-Max}
z\in K\mapsto P_{{\text{CEV, Max}}}(z)=\frac{1}{2\pi}\int_{-\infty}^{+\infty} \e^{(1+iu)T} {\hat \Phi}(z,u)
du,
\end{equation}
with  
\begin{align*}
 {\hat \Phi}(z,u)&= -\frac{c}{\mu} z^{-2\beta-1}\frac{U(\frac{1+iu}{2\mu \beta}, 1+\frac{1}{2\beta}, c X_0^{-2\beta})U(\frac{1+iu}{2\mu \beta}+1, 2+\frac{1}{2\beta}, cz^{-2\beta})}{U(\frac{1+iu}{2\mu \beta}, 1+\frac{1}{2\beta}, c z^{-2\beta})^2}, \mbox{ for } \mu>0\\
\mbox{and}&\\
 {\hat \Phi}(z,u)&= -cz^{-2\beta-1}\left( \frac{2\beta +1}{1+iu}-\frac 1{\mu}\right)\frac{U(1+\frac 1{2\beta}-\frac{1+iu}{2 \mu \beta}, 1+\frac{1}{2\beta}, c X_0^{-2\beta})U(2+\frac 1{2\beta}-\frac{1+iu}{2 \mu \beta}, 2+\frac{1}{2\beta}, cz^{-2\beta})}{U(1+\frac 1{2\beta}-\frac{1+iu}{2 \mu \beta}, 1+\frac{1}{2\beta}, c z^{-2\beta})^2},
\end{align*}
for $\mu<0$.
\end{thm}
\begin{proof}
$\bullet$ {\bf Case $ {\mu>0}$}. 
Let us recall that the process given by $Z_t= \frac{X_t^{-2(\alpha-1)}}{4(\alpha-1)^2}$ is solution to \eqref{cevtocir}.
%Consequently, 
%\begin{align*}
%    \mathbb{P}\Big(\sup_{t\in [0, T]}X_t \geq x \Big)&= \mathbb{P}\Big(\inf_{t\in[0, T]}Z_t\leq  z\Big)
%\end{align*}
Then,  as $\alpha>1$ we get
\begin{align*}
    \mathbb{P}\Big(\sup_{t\in [0, T]}X_t \geq z \Big) &= 
    \mathbb P\Big( (\sup_{t \in[0, T]} X_t)^{-2(\alpha-1)} \leq z^{-2(\alpha-1)}\Big)\\
    &= \mathbb P\Big( \inf_{t\in [0, T]}  X_t^{-2(\alpha-1) }\leq z^{-2(\alpha-1)} \Big)\\  
  &
  =\mathbb{P}\Big(\inf_{t\in[0, T]}Z_t\leq  \frac{z^{-2(\alpha-1)}}{4(\alpha-1)^2}\Big).
\end{align*}
Now, using the same arguments in the proof of Theorem \ref{thm:CIR-min} with $Z_0=\frac{X_0^{-2(\alpha-1)}}{4(\alpha-1)^2}$,  $a=\frac{\sigma^2(2\alpha-1)}{4(\alpha-1)}$, $\kappa=2\mu(\alpha-1)$, we get for $\beta=\alpha-1$ and $c=\dfrac{\mu}{\beta \sigma^2}$ 
\begin{align*}
     \mathbb{P}\Big(\sup_{t\in [0, T]}X_t \geq z \Big)&=\mathbb{P}\left[ \inf_{0\leq s \leq t} Z_s \leq \frac{z^{-2(\alpha-1)}}{4(\alpha-1)^2}\right]\\
    %&= \frac{1}{2\pi i}\int_{1-i\infty}^{1+i\infty}\frac{e^{t s}}{s}\frac{U(s/\kappa,2a/\sigma^2,2\kappa Z_0/\sigma^2)}{U(s/\kappa,2a/\sigma^2,2\kappa z_{\alpha}/\sigma^2)} ds\\
    &=\frac{1}{2\pi}\int_{-\infty}^{+\infty} \e^{(1+iu)t} \Phi(u,z)du,\\
    \mbox{ with } \Phi(u,z)&:= \dfrac{1}{1+iu} \dfrac{U(\frac{1+iu}{2\mu \beta}, 1+\frac{1}{2\beta}, c X_0^{-2\beta})}{U(\frac{1+iu}{2\mu \beta}, 1+\frac{1}{2\beta}, c z^{-2\beta})}.
\end{align*}
and thus we easily get \eqref{CEV-Max}.
\\
\paragraph*{$\bullet$ \bf Case $ {\mu<0}$.}
For this case, by \eqref{lap:cev-max} we have 
\begin{align*}
\bold{E}[\e^{-s\tau_{X_0 \uparrow z}}]=& \exp\Big(-c(X_0^{-2\beta}-y^{-2\beta})\Big)\dfrac{U\Big(1+\frac 1{2\beta}-\frac{s}{2 \mu \beta}, \frac{1}{2\beta}+1, cX_0^{-2\beta}\Big)}{U\Big(1+\frac 1{2\beta}-\frac{s}{2 \mu \beta}, \frac{1}{2\beta}+1, cz^{-2\beta}\Big)}
\end{align*}
and $$\mathbb{P}\big[\sup_{0\leq s \leq t} X_s>z\big]= \frac{\exp\Big(-c(X_0^{-2\beta}-y^{-2\beta})\Big)}{2\pi}\int_{-\infty}^{+\infty}  {\e^{(1+iu)t}} \Phi(z,u) du,$$ 
with $$\Phi(z,u)= \frac{1}{1+iu} \dfrac{U\Big(1+\frac 1{2\beta}-\frac{1+iu}{2 \mu \beta}, \frac{1}{2\beta}+1, cX_0^{-2\beta}\Big)}{U\Big(1+\frac 1{2\beta}-\frac{1+iu}{2 \mu \beta}, \frac{1}{2\beta}+1, cz^{-2\beta}\Big)}.$$
By  \eqref{Kummer-deriv}, we obtain
$$\frac{\partial \Phi}{\partial z}(z,u)= -cz^{-2\beta-1}\left( \frac{2\beta +1}{1+iu}-\frac 1{\mu}\right)\frac{U(1+\frac 1{2\beta}-\frac{1+iu}{2 \mu \beta}, 1+\frac{1}{2\beta}, c X_0^{-2\beta})U(2+\frac 1{2\beta}-\frac{1+iu}{2 \mu \beta}, 2+\frac{1}{2\beta}, cz^{-2\beta})}{U(1+\frac 1{2\beta}-\frac{1+iu}{2 \mu \beta}, 1+\frac{1}{2\beta}, c z^{-2\beta})^2}.$$ Proceeding in the same way as for the previous case, we first apply \eqref{asymp:U+} to get  for $j\in\{0,1\}$ and $v\in\{X_0,z\}$
\begin{align*}
U(1+j+\frac 1{2\beta}-\frac{1+iu}{2 \mu \beta}, 1+j+\frac{1}{2\beta}, c v^{-2\beta})&\sim
\sqrt{\pi}\frac{\left(1+j+\frac 1{2\beta}-\frac{1+iu}{2 \mu \beta}\right)^{-\frac14 +\frac j2+\frac{1}{4\beta} }}{\Gamma\left(1+j+\frac 1{2\beta}-\frac{1+iu}{2 \mu \beta}\right)}\left(c v^{-2\beta}\right)^{-\frac14 -\frac j2-\frac{1}{4\beta} }\\&\times
\exp\left({\frac{c v^{-2\beta}}2-2\sqrt{c v^{-2\beta}\left(1+j+\frac 1{2\beta}-\frac{1+iu}{2 \mu \beta}\right)}}\right)
\end{align*}
and then we use that $\exp\left(-2\sqrt{c v^{-2\beta}\left(1+j+\frac 1{2\beta}-\frac{1+iu}{2 \mu \beta}\right)} \right) \sim\exp\left(-\sqrt{-\frac{c v^{-2\beta}}{ \mu \beta}}(1+i)\sqrt{u}\right)$
as $u\to +\infty$, uniformly in bounded $v$-domain. We then deduce that
\begin{align*}
\frac{\partial \Phi}{\partial z}(z,u)&\sim-cz^{-2\beta-1}\left( \frac{2\beta +1}{1+iu}-\frac 1{\mu}\right)\left(1+\frac 1{2\beta}-\frac{1+iu}{2 \mu \beta}\right)^{-\frac34 -\frac{1}{4\beta} }\left(2+\frac 1{2\beta}-\frac{1+iu}{2 \mu \beta}\right)^{\frac14 +\frac{1}{4\beta} }
\\
&\times
\left(c {X_0}^{-2\beta}\right)^{-\frac14 -\frac{1}{4\beta} }
\left(c {z}^{-2\beta}\right)^{-\frac14 +\frac{1}{4\beta} } \exp\left(\frac c2\left( X_0^{-2\beta}-z^{-2\beta}\right)\right)
\exp\left( \sqrt{-\frac{c}{ \mu \beta}}(1+i)\sqrt{u}(z^{-\beta}-X_0^{-\beta})\right)
\end{align*}
and 
\begin{align*}
\left|\frac{\partial \Phi}{\partial z}(z,u)\right|&\sim
\sqrt{-\frac{{2\beta c}}{\mu}}(X_0)^{\frac 12({\beta}+1)}z^{-\frac 32(\beta+1)}
u^{-\frac 12} \exp\left(\sqrt{-\frac c{\mu\beta}}\sqrt{u}(z^{-\beta}-X_0^{-\beta})\right)
\exp\left(\frac{c}{2}(X_0^{-2\beta}-z^{-2\beta})\right)
\end{align*}
as $u\to +\infty$ uniformly  in any bounded $z$-domain subset of $(X_0,+\infty)$.  As $c, \beta$ and $-\mu$ are positive constants it follows that $\int_{1}^{+\infty} \left|\e^{(1+iu)t}\frac{\partial \Phi}{\partial z}(z,u)\right| du <\infty 
$ is uniformly convergent in any bounded $z$-domain  in $(X_0,+\infty)$.  Concerning the integral from $-\infty$ to $-1$, we proceed as above and use \eqref{asymp:U-}, \eqref{eq:j} and \eqref{eq:y} to get 
 for $j\in\{0,1\}$ and $v\in\{X_0,z\}$
 \begin{align*}
& U(-(-1-j-\frac 1{2\beta}+\frac{1-iu}{2 \mu \beta}),1+j+\frac{1}{2\beta}, c v^{-2\beta})\sim
 \frac{\e^{\frac{c v^{-2\beta}}2}}{\sqrt{\pi}}
 \left(\frac{c v^{-2\beta}}{-1-j-\frac 1{2\beta}+\frac{1-iu}{2 \mu \beta}}\right)^{-\frac{j}2-\frac 1{4\beta}}\\
 &\times \left(c v^{-2\beta}\left(-1-j-\frac 1{2\beta}+\frac{1-iu}{2 \mu \beta}\right) \right)^{-\frac 14}
 \\&\times \cos\left(\pi\left(-1-j-\frac 1{2\beta}+\frac{1-iu}{2 \mu \beta}\right) - 
 2\sqrt{c v^{-2\beta}\left(-1-j-\frac 1{2\beta}+\frac{1-iu}{2 \mu \beta}\right)}
 +\frac{\pi}2 (j+\frac 1{2\beta})+\frac{\pi}4
 \right)
 \end{align*}
as $u\rightarrow +\infty$ uniformly in bounded $v$-domain. 
Using that 
\begin{multline*}
 \cos\left(\pi\left(-1-j-\frac 1{2\beta}+\frac{1-iu}{2 \mu \beta}\right) - 
 2\sqrt{c v^{-2\beta}\left(-1-j-\frac 1{2\beta}+\frac{1-iu}{2 \mu \beta}\right)}
 +\frac{\pi}2 (j+\frac 1{2\beta})+\frac{\pi}4
 \right)
\\\sim \frac 12 e^{-i\pi \left(-\frac 34 - \frac j2 -\frac 1{4\beta}+ \frac 1{2\mu\beta}\right)}
e^{-\frac{\pi}{2\mu\beta} u-(1-i)\sqrt{-\frac{cv^{-2\beta}}{\mu\beta}u}},
\end{multline*}
as $u\rightarrow +\infty$ uniformly in bounded $v$-domain.  Thus, we get
\begin{align*}
\left|\frac{\partial \Phi}{\partial z}(z,u)\right|&\sim
\sqrt{-\frac{{2\beta c}}{\mu}}(X_0)^{\frac 12({\beta}+1)}z^{-\frac 32(\beta+1)}
u^{-\frac 12} \exp\left(\sqrt{-\frac c{\mu\beta}}\sqrt{u}(z^{-\beta}-X_0^{-\beta})\right)
\exp\left(\frac{c}{2}(X_0^{-2\beta}-z^{-2\beta})\right)
\end{align*}
as $u\to +\infty$ uniformly  in any bounded $z$-domain subset of $(X_0,+\infty)$. 
As $c, \beta$ and $-\mu$ are positive constants it follows that $\int_{-\infty}^{-1} \left|\e^{(1+iu)t}\frac{\partial \Phi}{\partial z}(z,u)\right| du <\infty 
$ is uniformly convergent in any bounded $z$-domain  in $(X_0,+\infty)$. We complete the proof by noticing that  $(u,z)\in\mathbb R\times K\mapsto \e^{(1+iu)t}\frac{\partial \Phi}{\partial z}(u,z)$ is a continuous function for any compact set  $K\subset(X_0,+\infty)$ (see e.g. \cite[Theorem B.3]{PinZaf}).
\end{proof}
\subsection{Minimum of CEV process}
Let us denote by  $\tau_{X_0 \downarrow z}:= \inf\{t\geq 0 : X_t=z\}$ the first time that the CEV process
$(X_t)_{t\ge0}$ starting at $X_0$ hits the level $0<z<X_0$. By  \cite[subsections 5.3.6 and  6.4.5]{zbMATH05017324}, the Laplace transform of the hitting time $\tau_{X_0 \downarrow z}:= \inf\{t\geq 0 : X_t=z\}$ is given by 
\begin{align}\label{lap:min-cev}
\bold{E}[\e^{-s\tau_{X_0 \downarrow z}}]&=\Big(\frac{X_0}{z}\Big)^{\beta+\frac{1}{2}}\exp\Big(\frac{\epsilon}{2}c(X_0^{-2\beta}-z^{-2\beta})\Big)\dfrac{M_{k,n}(cX_0^{-2\beta})}{M_{k,n}(cz^{-2\beta})}
\end{align}
with  $\epsilon={\rm sign}(\mu \beta)$, $n=\frac{1}{4\beta}$, $k=\epsilon\Big(\frac{1}{2}+\frac{1}{4\beta}\Big)-\frac{s}{2\beta|\mu|}$ and the Whittaker function $M_{k,n}(y)= y^{n+\frac{1}{2}} \e^{-\frac{y}{2}} {_1F_1}(n-k+\frac{1}{2}, 2n+1, y)$, where $_1F_1$ denotes the  confluent hypergeometric function of the first kind defined in \eqref{Kummer} and with $\beta=\alpha-1$ and $c=\dfrac{|\mu|}{\beta \sigma^2}$.
\begin{thm}\label{thm:CEV-min}
Let $(X_t)_{0\le t\le T}$ denotes the CEV process solution to  \eqref{sde-cev}.  Then  $\inf_{t\in[0,T]} X_t$ has a continuous density  on any compact set  $K\subset(0,X_0)$, given by
\begin{equation}
\label{CEV-Min}
z\in K\mapsto P_{\text{CEV, Min}}(z)=\frac{1}{2\pi}\int_{-\infty}^{+\infty} \e^{(1+iu)T} {\hat \Psi}(z,u)
du,
\end{equation}
with  
\begin{align*}
{\hat \Psi}(z,u)&=  -\frac{cz^{-2\beta-1}}{\mu(1+\frac 1{2\beta})}\dfrac{{_1F_1}(\frac{1+iu}{2\mu \beta}, 1+\frac{1}{2\beta}, c X_0^{-2\beta}){_1F_1}(\frac{1+iu}{2\mu \beta}+1, 2+\frac{1}{2\beta}, c z^{-2\beta})}{{_1F_1}(\frac{1+iu}{2\mu \beta}, 1+\frac{1}{2\beta}, c z^{-2\beta})^2}, \mbox{ for } \mu>0\\
\mbox{and}&\\
{\hat \Psi}(z,u)&=  -{cz^{-2\beta-1}}\left(\frac{2\beta}{1+iu} -\frac 1{\mu(1+\frac 1{2\beta})}\right)\\&\times \dfrac{{_1F_1}(1+\frac{1}{2\beta}-\frac{1+iu}{2 \mu \beta}, 1+\frac{1}{2\beta}, c X_0^{-2\beta}){_1F_1}(2+\frac{1}{2\beta}-\frac{1+iu}{2 \mu \beta}, 2+\frac{1}{2\beta}, c z^{-2\beta})}{{_1F_1}(1+\frac{1}{2\beta}-\frac{1+iu}{2 \mu \beta}, 1+\frac{1}{2\beta}, c z^{-2\beta})^2}, \mbox{ for }
\mu<0.
\end{align*}
\end{thm}
\begin{proof}
$\bullet$ {\bf Case $ {\mu>0}$.} Recalling that the process given by $Z_t= \frac{X_t^{-2(\alpha-1)}}{4(\alpha-1)^2}$ is solution to \eqref{cevtocir},
then,  as $\alpha>1$ we get
\begin{align*}
    \mathbb{P}\Big(\inf_{t\in [0, T]}X_t \leq z \Big) &= 
    \mathbb P\Big( (\inf_{t \in[0, T]} X_t)^{-2(\alpha-1)} \geq z^{-2(\alpha-1)}\Big)\\
    &= \mathbb P\Big( \sup_{t\in [0, T]}  X_t^{-2(\alpha-1) }\geq z^{-2(\alpha-1)} \Big)\\  
  &
  =\mathbb{P}\Big(\sup_{t\in[0, T]}Z_t\geq  \frac{z^{-2(\alpha-1)}}{4(\alpha-1)^2}\Big).
\end{align*}
Now, using the same arguments in the proof of Theorem \ref{thm:CIR-max} with $Z_0=\frac{X_0^{-2(\alpha-1)}}{4(\alpha-1)^2}$,  $a=\frac{\sigma^2(2\alpha-1)}{4(\alpha-1)}$, $\kappa=2\mu(\alpha-1)$, we get for $\beta=\alpha-1$ and $c=\dfrac{\mu}{\beta \sigma^2}$ 
\begin{align*}
     \mathbb{P}\Big(\inf_{t\in [0, T]}X_t \leq z \Big)&=\mathbb{P}\left[ \sup_{0\leq s \leq t} Z_s \geq \frac{z^{-2(\alpha-1)}}{4(\alpha-1)^2}\right]\\
    %&= \frac{1}{2\pi i}\int_{1-i\infty}^{1+i\infty}\frac{e^{t s}}{s}\frac{U(s/\kappa,2a/\sigma^2,2\kappa Z_0/\sigma^2)}{U(s/\kappa,2a/\sigma^2,2\kappa z_{\alpha}/\sigma^2)} ds\\
    &=\frac{1}{2\pi}\int_{-\infty}^{+\infty} \e^{(1+iu)t} \Psi(u,z)du,\\
    \mbox{ with } \Psi(u,z)&:=\dfrac{1}{1+iu} \dfrac{ {_1F_1}(\frac{1+iu}{2\mu \beta}, 1+\frac{1}{2\beta}, c X_0^{-2\beta})}{{_1F_1}(\frac{1+iu}{2\mu \beta}, 1+\frac{1}{2\beta}, c z^{-2\beta})}.
\end{align*}
and thus we get \eqref{CEV-Min}.
\\
\\
\paragraph*{$\bullet$ \bf Case $ {\mu<0}$.}
In this case  by \eqref{lap:min-cev} we have, 
\begin{align*}
\bold{E}[\e^{-s\tau_{X_0 \downarrow z}}]&= \exp\Big(-c(X_0^{-2\beta}-z^{-2\beta})\Big)\dfrac{{_1F_1}\Big(\frac{1}{2\beta}+1-\frac{s}{2 \mu \beta}, \frac{1}{2\beta}+1, cX_0^{-2\beta}\Big)}{{_1F_1}\Big(\frac{1}{2\beta}+1-\frac{s}{2 \mu \beta}, \frac{1}{2\beta}+1, cz^{-2\beta}\Big)}
\end{align*}
and 
\begin{align*}
\mathbb{P}\left[ \inf_{0\leq s \leq t} X_s \geq z\right]&=
\frac{ \exp\Big(-c(X_0^{-2\beta}-z^{-2\beta})\Big)}{2\pi } \int_{-\infty}^{\infty}{\e^{(1+iu)t}} \Psi(z,u) du    
\end{align*}
with 
$\Psi(z,u) =\dfrac{1}{1+iu} \dfrac{{_1F_1}\Big(\frac{1}{2\beta}+1-\frac{1+iu}{2 \mu \beta}, \frac{1}{2\beta}+1, cX_0^{-2\beta}\Big)}{{_1F_1}\Big(\frac{1}{2\beta}+1-\frac{1+iu}{2 \mu \beta}, \frac{1}{2\beta}+1, cz^{-2\beta}\Big)} $. By \eqref{eq:der-F}, we have
\begin{align}\label{eq:dens-min-cev}
\frac{\partial \Psi}{\partial z}(z,u)&= -{cz^{-2\beta-1}}\left(\frac{2\beta}{1+iu} -\frac 1{\mu(1+\frac 1{2\beta})}\right)\\&\times \dfrac{{_1F_1}(1+\frac{1}{2\beta}-\frac{1+iu}{2 \mu \beta}, 1+\frac{1}{2\beta}, c X_0^{-2\beta}){_1F_1}(2+\frac{1}{2\beta}-\frac{1+iu}{2 \mu \beta}, 2+\frac{1}{2\beta}, c z^{-2\beta})}{{_1F_1}(1+\frac{1}{2\beta}-\frac{1+iu}{2 \mu \beta}, 1+\frac{1}{2\beta}, c z^{-2\beta})^2}.
\end{align}
For $j\in\{0,1\}$ and $v\in\{X_0,z\}$, we follow the similar steps that led us to get \eqref{kummer} and  we obtain
\begin{align*}
{_1F_1}(1+j+\frac{1}{2\beta}-\frac{1+iu}{2 \mu \beta},& 1+j+\frac{1}{2\beta}, c v^{-2\beta})\sim
\frac{\e^{\frac{cv^{-2\beta}}{2}}}{\sqrt{4\pi}}
\left(\frac{cv^{-2\beta}}{1+j+\frac 1{2\beta}-\frac{1+iu}{2\mu\beta}}\right)^{-\frac j2-\frac 1{4\beta}}\\&
\times\frac{\Gamma(1+j+\frac1{2\beta})\Gamma(1-\frac{1+iu}{2\mu\beta})}{\Gamma(1+j+\frac{1}{2\beta}-\frac{1+iu}{2 \mu \beta})}
\left(c v^{-2\beta}\left(1+j+\frac{1}{2\beta}-\frac{1+iu}{2 \mu \beta} \right)\right)^{-\frac 14}\\
&\times \exp\left(2\sqrt{c v^{-2\beta}\left(1+j+\frac{1}{2\beta}-\frac{1+iu}{2 \mu \beta} \right)}\right)
\end{align*}
as $u\to+\infty$ uniformly  in bounded $v$-domain. Now, we use %\eqref{eq:gamma} and  
that
$$
 \exp\left(2\sqrt{cv^{-2\beta}\left(1+j+\frac{1}{2\beta}-\frac{1+iu}{2 \mu \beta} \right)}\right)\sim 
 \exp\left(\sqrt{-\frac{cv^{-2\beta}}{\mu\beta}}(1+i)\sqrt{u}\right), \mbox{ as } u\to +\infty
$$
uniformly in bounded $v$-domain, to get 
\begin{align*}
{_1F_1}(1+j+\frac{1}{2\beta}-\frac{1+iu}{2 \mu \beta}, 1+j+\frac{1}{2\beta}, c v^{-2\beta})&\sim
\frac{\e^{\frac{cv^{-2\beta}}{2}}}{\sqrt{4\pi}}
\left({cv^{-2\beta}}\left(1+j+\frac{1}{2\beta}-\frac{1+iu}{2 \mu \beta} \right)\right)^{-\frac j2-\frac 1{4\beta}-\frac 14}\\&
\times{\Gamma(1+j+\frac1{2\beta})}
 \exp\left(\sqrt{-\frac{cv^{-2\beta}}{\mu\beta}}(1+i)\sqrt{u}\right)
\end{align*}
as $u\to+\infty$ uniformly  in bounded $v$-domain.
Hence, by \eqref{eq:dens-min-cev} we deduce that
\begin{align*}
|\frac{\partial \Psi}{\partial z}(z,u)|&\sim
\sqrt{-\frac{{2\beta c}}{\mu}}(X_0)^{\frac 12({\beta}+1)}z^{-\frac 32(\beta+1)}
u^{-\frac 12} \exp\left(\sqrt{-\frac c{\mu\beta}}\sqrt{u}(X_0^{-\beta}-z^{-\beta})\right)
\exp\left(\frac{c}2(X_0^{-2\beta}-z^{-2\beta})\right),
\end{align*}
as $u\to +\infty$ uniformly  in any bounded $z$-domain subset of $(0,X_0)$.  As $c, \beta$ and $-\mu$ are positive constants it follows that $\int_{1}^{+\infty} \left|\e^{(1+iu)t}\frac{\partial K}{\partial z}(z,u)\right| du <\infty 
$ is uniformly convergent in any bounded $z$-domain  in $(0,X_0)$.  Besides, for the integral from $-\infty$ to $-1$, we use similar steps as the ones that led us to get \eqref{eq:F-inf}, and we have
\begin{align*}
{_1F_1}(1+j+\frac{1}{2\beta}-\frac{1-iu}{2 \mu \beta}, 1+j+\frac{1}{2\beta}, c v^{-2\beta})&\sim\frac{\e^{\frac 12 cv^{-2\beta}}}{\sqrt{\pi}} \left(\frac{cv^{-2\beta}}{-1-j-\frac1{2\beta}-\frac{iu-1}{2\mu \beta}}\right)^{-\frac j2-\frac 1{4\beta}}\\&\times
\frac{\Gamma(1+j+\frac 1{2\beta}) \Gamma(-j-\frac 1{2\beta}-\frac{iu-1}{2\mu\beta})}{\Gamma(-\frac{iu-1}{2\mu\beta})}\\
&\left({cv^{-2\beta}}\left(-1-j-\frac1{2\beta}-\frac{iu-1}{2\mu \beta}\right)\right)^{-\frac 14}\\&\hspace{-2cm}\times\cos\left(2\sqrt{cv^{-2\beta}\left(-1-j-\frac1{2\beta}-\frac{iu-1}{2\mu \beta}\right)}-\frac{\pi}2(j+\frac 1{2\beta})-\frac{\pi}{4}\right)
\end{align*}
as $u\to\infty$ uniformly on $v$-bounded domain.  Now, using %\eqref{eq:gamma} and 
 that
$$
\cos\left(2\sqrt{cv^{-2\beta}\left(-1-j-\frac1{2\beta}-\frac{iu-1}{2\mu \beta}\right)}-\frac{\pi}2(j+\frac 1{2\beta})-\frac{\pi}{4}\right)
\\\sim \frac 12 e^{i\pi \left(\frac{j}2+\frac 1{4\beta}+\frac 14\right)}
e^{(1-i)\sqrt{-\frac{cv^{-2\beta}}{\mu\beta}u}},
$$
as $u\to\infty$ uniformly on $v$-bounded domain, we get
\begin{align*}
{_1F_1}(1+j+\frac{1}{2\beta}-\frac{1-iu}{2 \mu \beta}, 1+j+\frac{1}{2\beta}, c v^{-2\beta})&\sim\frac{\e^{\frac 12 cv^{-2\beta}}}{2\sqrt{\pi}} \left({cv^{-2\beta}}\right)^{-\frac j2-\frac 1{4\beta}-\frac 1{4}}
{\Gamma(1+j+\frac 1{2\beta}) }\\
&\left(-1-j-\frac1{2\beta}-\frac{iu-1}{2\mu \beta}\right)^{-\frac j2 - \frac 1{4\beta}-\frac 14}\e^{i\pi \left(\frac{j}2+\frac 1{4\beta}+\frac 14\right)}\\&\times
\exp\left({(1-i)\sqrt{-\frac{cv^{-2\beta}}{\mu\beta}u}}\right)
\end{align*}
and that
\begin{align*}
|\frac{\partial \Psi}{\partial z}(z,-u)|&\sim
\sqrt{-\frac{{2\beta c}}{\mu}}(X_0)^{\frac 12({\beta}+1)}z^{-\frac 32(\beta+1)}
u^{-\frac 12} \exp\left(\sqrt{-\frac c{\mu\beta}}\sqrt{u}(X_0^{-\beta}-z^{-\beta})\right)
\exp\left(\frac{c}2(X_0^{-2\beta}-z^{-2\beta})\right),
\end{align*}
as $u\to +\infty$ uniformly  in any bounded $z$-domain subset of $(0,X_0)$.
Thus,  as $c, \beta$ and $-\mu$ are positive constants we deduce  that $\int_{-\infty}^{-1} \left|\e^{(1+iu)t}\frac{\partial \Psi}{\partial z}(z,u)\right| du <\infty 
$ is uniformly convergent in any bounded $z$-domain  in $(0,X_0)$. We complete the proof by noticing that  $(u,z)\in\mathbb R\times K\mapsto \e^{(1+iu)t}\frac{\partial \Psi}{\partial z}(u,z)$ is a continuous function for any compact set  $K\subset(0,X_0)$ (see e.g. \cite[Theorem B.3]{PinZaf}).
\end{proof}
\subsection{Numerical tests}
For the CEV case, we  consider the pricing problem of the quantities introduced in \eqref{pricing:prob-cev}. 
%rather consider the problem of pricing  an U-O barrier option $\mathbb{E}\Big[g(X_T)\mathbbm{1}_{\{ \sup_{t\in[0,T]}X_t<{\mathcal D}\}}\Big] $ with barrier $\mathcal D$ and  a D-O barrier option $\mathbb{E}\Big[g(X_T)\mathbbm{1}_{\{\inf_{t\in[0,T]}X_t>{\mathcal U}\}}\Big] $ with barrier $\mathcal U$,
%where $(X_t)_{0\le t\le T}$ is the CEV process solution to \eqref{sde-cev}. By the Lamperti transform we reduce ourselves to our original problem of pricing  a D-O option ~$\pi_{\mathcal D^{1-\alpha}}=\mathbb{E}\Big[g(Y_T^{\frac1{1-\alpha}})\mathbbm{1}_{\{ \inf_{t\in[0,T]}Y_t>{{\mathcal D}^{1-\alpha}}\}}\Big]$ with barrier ${\mathcal D^{1-\alpha}}$ and  an U-O option $\pi_{\mathcal U^{1-\alpha}}=\mathbb{E}\Big[g(Y_T^{\frac1{1-\alpha}})\mathbbm{1}_{\{\sup_{t\in[0,T]}Y_t<{\mathcal U}^{1-\alpha}\}}\Big]$ with barrier ${\mathcal U^{1-\alpha}}$, for the underlying asset
% $(Y_t)_{t\in[0,T]}$ solution to \eqref{lamperti}. 
More precisely, we approximate $\Pi^{\text{U-O}, X}_{\mathcal D}$ (resp. $\Pi^{\text{D-O}, X}_{\mathcal U}$)  by the improved MLMC algorithm  $\bar Q_{\mathcal D^{1-\alpha}}$ given in   \eqref{MLMC:D}  (resp.  $\bar P_{\mathcal U^{1-\alpha}}$ given in   \eqref{MLMC:U}), where we used our  interpolated drift implicit scheme
\begin{align*}
    \overline{Y}^n_t&=    \overline{Y}^n_{t_{i}}+ (1-\alpha)\left( \mu \overline{Y}^n_{t_{i+1}} - \alpha\frac{\sigma^2}{2\overline{Y}^n_{t_{i+1}}}\right)(t-t_{i})+\gamma(W_t-W_{t_{i}}),~~~~ \text{for}\, t \in [t_i,t_{i+1}[, 0\le i\le n-1,\\
Y_0&={X_0}^{1-\alpha},
\end{align*}
with $\gamma =  \sigma(1-\alpha).$ 
For $n$ large enough, the positive solution to the above implicit scheme is explicit and given by 
$$
\overline{Y}^n_{t_{i+1}}=\frac{\sqrt{2\sigma^2 \alpha(\alpha-1)(1+\mu(\alpha-1)\frac Tn)\frac Tn + (\gamma(W_{t_{i+1}}-W_{t_i})+\overline{Y}^n_{t_{i}})^2}+\gamma(W_{t_{i+1}}-W_{t_i})+\overline{Y}^n_{t_{i}} }{2+2\mu(\alpha-1)\frac Tn}.
$$
 For the CEV model, we dont have any benchmark price. To illustrate the MLMC complexity performance we choose a set of parameter \eqref{cond:mom-cev}, namely  $\alpha=1.2$, $X_0=100$, $\mu=0.1$, $\sigma=0.2$, $T=1$. The payoff function $g(x)=\e^{-rT}(x^{\frac{1}{1-\alpha}}-K)_+$ is a discounted call function with $r=0.1$. For the U-O option the strike is $K=90$, and the barrier ${\mathcal D}=150$. For the D-O option the strike is $K=100$ and the barrier ${\mathcal U}=90$.   The tables and the figures below confirm the high performance of the improved MLMC. 
 \begin{table}[h!]
\begin{center}
\begin{tabular}{ |c | c | c | c | c |}
\hline
 Accuracy & Price & MLMC cost & MC cost & Saving\\ 
 \hline
 \hline
$10^{-4}$ & $3.0390$ & $8.226\times10^9$ & $7.34\times 10^{13} $ & $8922.33$\\  
$5\times10^{-4}$ & $3.0391$ & $3.17\times10^8$ & $3.67\times 10^{11} $ & 1155.67 \\  
$10^{-3}$ & 3.041 & $7.436\times 10^7$ & $4.587\times 10^{10}$ & 616.91\\  
%$10^{-2}$ & 9.5333 & $7.224\times 10^3$ & $8.710\times 10^5 $ & 120.57 \\  
$10^{-2}$ & 3.0452  & $6.539\times 10^5$ & $5.734\times 10^7$ & 87.69\\
\hline
\end{tabular}
\caption{MLMC complexity tests for the U-O barrier option pricing of $\Pi^{\text{U-O}, X}_{\mathcal D}$}
\end{center}
\end{table}

%%%%%%%
\begin{table}[h!]
\begin{center}
\begin{tabular}{ |c | c | c | c | c |}
\hline
 Accuracy & Price & MLMC cost & MC cost & Saving\\ 
 \hline
 \hline
$5\times10^{-4}$ & 11.102 & 6.483$\times 10^9$ & 1.642$\times 10^{13} $ & 2532.83\\  
$10^{-3}$ & 11.103 & $1.608\times 10^{9}$ & $2.053\times 10^{12} $ & 1276.66 \\  
$5\times10^{-3}$ & 11.106 & 6.379$\times 10^7$ & 2.053$\times 10^{10} $ & 321.77\\  
$10^{-2}$ & 11.094  & $1.587\times 10^7$ & $2.566\times 10^9$ & 161.69\\
\hline
\end{tabular}
\caption{MLMC complexity tests for the D-O barrier option pricing of $\Pi^{\text{D-O}, X}_{\mathcal U}$}
\end{center}
\end{table}
\begin{figure}[h!]
\begin{subfigure}{.45\textwidth}
  \centering
  % include first image
  \includegraphics[width=.9\linewidth]{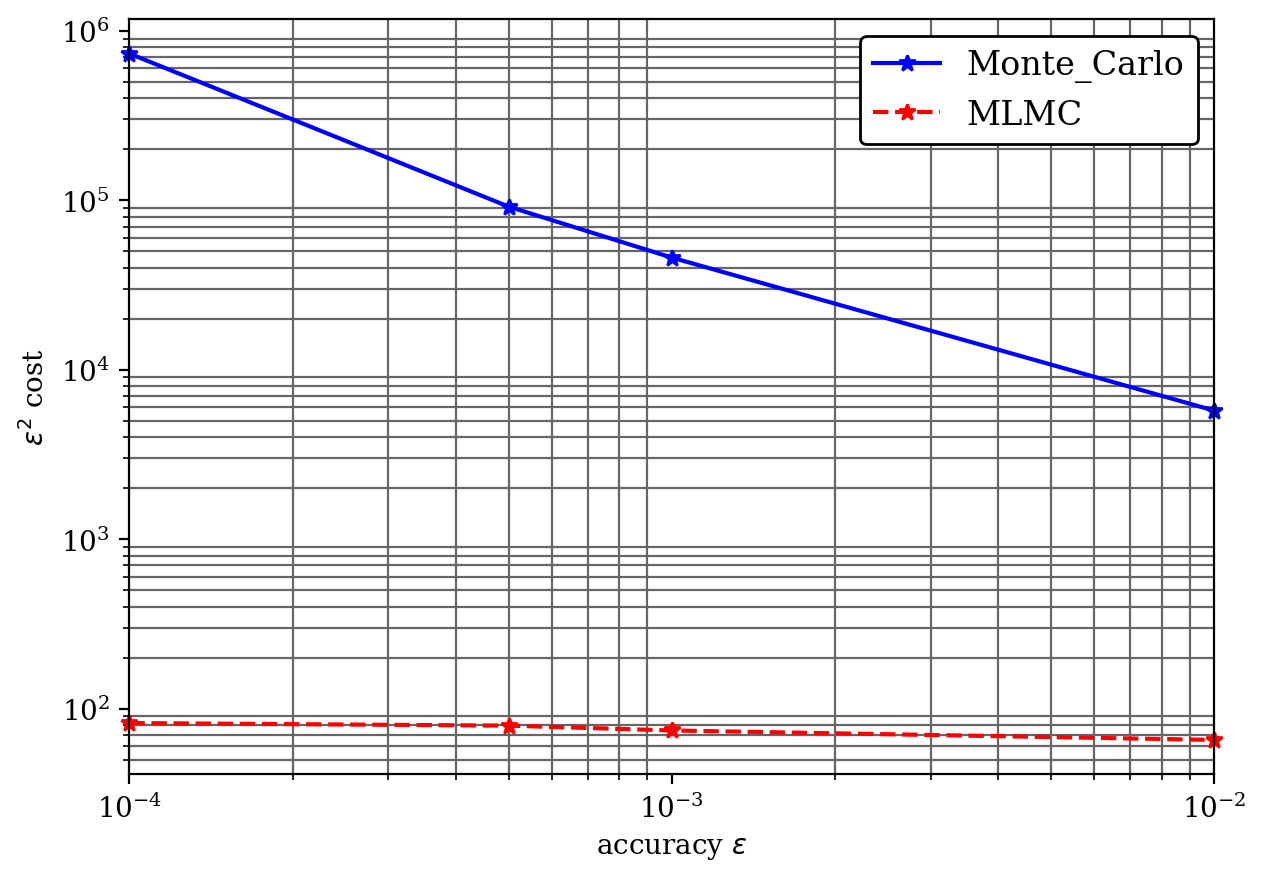}  
  \caption{Approximation of $\Pi^{\text{U-O}, X}_{\mathcal D}$}
  \label{fig:sub-first}
\end{subfigure}
\begin{subfigure}{.45\textwidth}
  \centering
  % include second image
  \includegraphics[width=.9\linewidth]{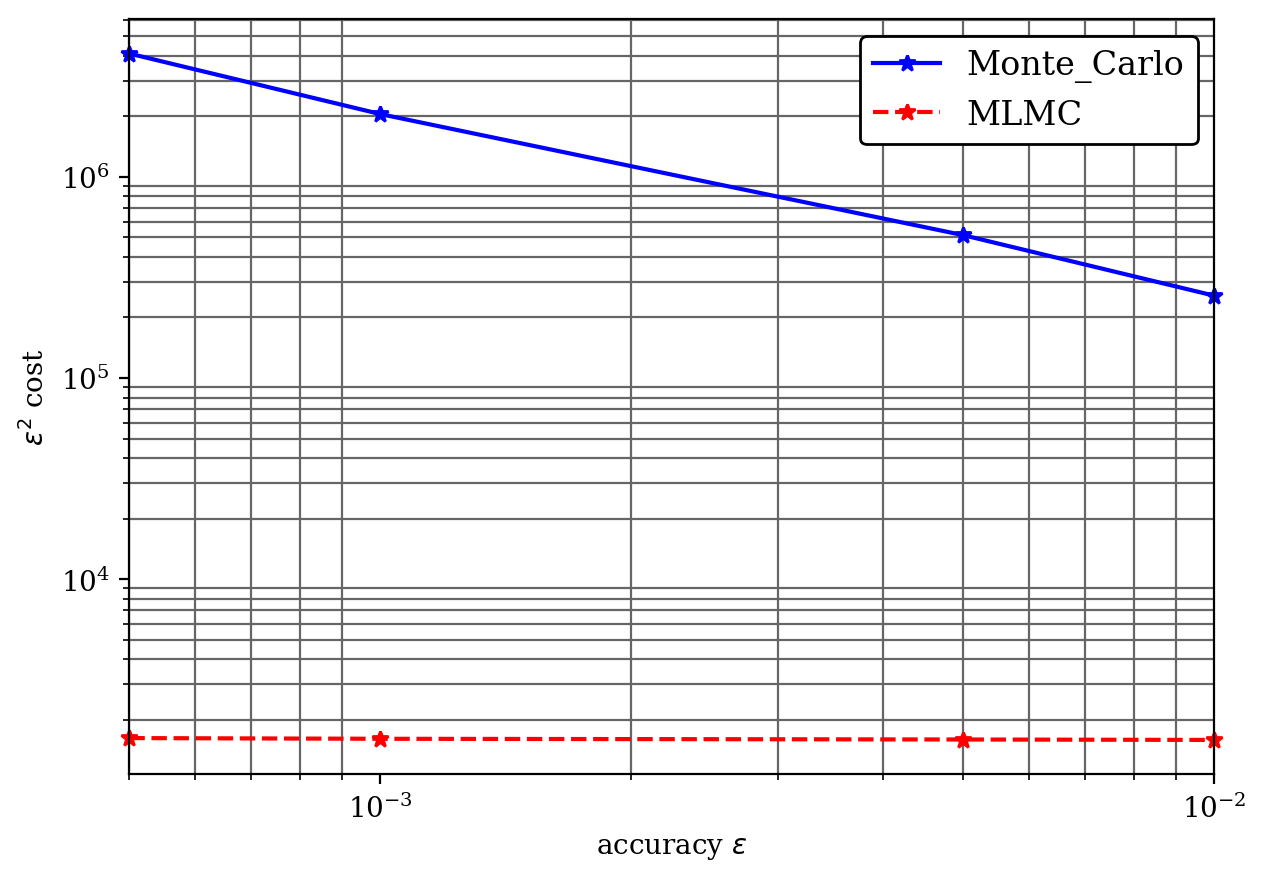}  
  \caption{Approximation of  $\Pi^{\text{D-O}, X}_{\mathcal U}$}
  \label{fig:sub-second}
\end{subfigure}
\caption{ Comparison for the performances of MLMC vs classical MC algorithm under the CEV model}
\label{fig:fig}
\end{figure}
For pricing barrier options under the popular CEV model, the  numerical results  confirm the supremacy of the improved MLMC algorithm  that reaches the optimal time complexity $O(\varepsilon^{-2})$  for a given precision $\varepsilon>0$.   
\section{Conclusion}
In this paper, we proved that the MLMC method for pricing barrier options reaches its optimal time-complexity regime,   when the underlying asset  has non-Lipschitz coefficients.~To apply our  theoretical results for the popular  CIR and CEV processes, we developed semi-explicit formulas for the densities of the running minimum and running maximum of these processes that are of independent interest.  It turns out that under some constraints on the parameters of these models guaranteeing the existence of finite negative moments up to some order, the MLMC method  behaves like an unbiased classic Monte Carlo estimator despite the use of approximation schemes.   It may be interesting to~extend this study by combining this improved version of the MLMC method with importance sampling technics for variance reduction as proposed in \cite{BenHajKeb1,KebLel, BenHajKeb3}, we leave this for a  possible future work.

\section{Acknowledgement}
We are grateful to Professor Christian Bayer for his suggestion to shorten the proofs of theorems 
\ref{thm:CEV-max} and \ref{thm:CEV-min} for the case $\mu>0$.

%\bibliographystyle{abbrv}
%\bibliography{reference}

\end{document}